\documentclass[11pt]{article}
\usepackage{xcolor}

\RequirePackage{amsthm,amsmath,amsfonts,amssymb}
\RequirePackage{natbib}
\RequirePackage[colorlinks,citecolor=blue,urlcolor=blue]{hyperref}
\usepackage{bm}
\usepackage{graphicx}
\usepackage{epstopdf}
\usepackage{algorithm}
\usepackage{algpseudocode}
\usepackage{upgreek}

\usepackage{multirow}
\usepackage{subcaption}

\theoremstyle{definition}
\newtheorem{thm}{Theorem}
\newtheorem{cor}[thm]{Corollary}
\newtheorem{lem}[thm]{Lemma}
\newtheorem{prop}[thm]{Proposition}

\theoremstyle{remark}
\newtheorem{defn}{Definition}
\newtheorem{example}{Example}
\newtheorem{assump}{Assumption}
\newtheorem{rem}{Remark}

\newcommand{\comment}[1]{}

\newcommand{\pw}{{\rm POWER}}
\newcommand{\fdr}{{\rm FDR}}
\newcommand{\sj}{\setminus j}

\DeclareMathOperator {\diag}{diag}

\DeclareMathOperator {\argmin}{argmin}

\DeclareMathOperator {\col}{col}

\title{Stability of a Generalized Debiased Lasso with Applications to Resampling-Based Variable Selection
}
\author{Jingbo Liu\thanks{Department of Statistics, University of Illinois Urbana-Champaign.
Also affiliated with Department of Electrical and Computer Engineering, the Grainger College of Engineering. 
jingbol@illinois.edu} 
}
\date{\today}

\begin{document}

\maketitle

\begin{abstract}
We propose a generalized debiased Lasso estimator based on a stability principle. When a single column of the design matrix is perturbed, the estimator admits a simple update formula that can be computed from the original solution. Under sub-Gaussian designs with well-conditioned covariance, this approximation is asymptotically accurate for all but a vanishing fraction of coordinates in the proportional growth regime.
The proof relies on concentration and anti-concentration arguments to control error terms and sign changes. In contrast, establishing comparable distributional limits (e.g., Gaussianity) under similar assumptions remains open.
As an application, we show that the approximation significantly reduces the computational cost of resampling-based variable selection procedures, including the conditional randomization test and a local knockoff filter.
\end{abstract}

\tableofcontents  
\section{Introduction}\label{sec_intro}
Suppose $A, B \in \mathbb{R}^{n \times p}$ are two matrices differing only in the $j$-th column. Let $Y \in \mathbb{R}^n$ and let $\rho \colon \mathbb{R} \to [0, +\infty)$ be a convex function, 
and define
\begin{align}
\hat{\alpha} &:= \operatorname*{argmin}_{\beta \in \mathbb{R}^p} \left\{ \frac{1}{2n} \|Y - A\beta\|_2^2 + \sum_{k=1}^p \rho(\beta_k) \right\}; \label{eq_alpha} \\
\hat{\beta}  &:= \operatorname*{argmin}_{\beta \in \mathbb{R}^p} \left\{ \frac{1}{2n} \|Y - B\beta\|_2^2 + \sum_{k=1}^p \rho(\beta_k) \right\}. \label{eq_beta}
\end{align}
Motivated in part by accelerating resampling-based variable selection methods \citep{candes2018panning},
we consider how to efficiently solve the second optimization given the solution to the first. 
In general, there is no explicit formula for updating $\hat{\alpha}$ to $\hat{\beta}$, and approximations based on Taylor expansions are not accurate in high-dimensional settings with correlated features. 
It is therefore natural to consider alternative representations of the solutions that preserve stability under local updates.
The main message of this paper is that \emph{debiased estimators satisfy an orthogonality principle that enables stable updates under fairly general conditions}. 
In particular, we show that correlated sub-Gaussian designs are sufficient for the approximation to hold.
To be precise, we establish approximation results for a generalized definition of the debiased estimator, which reduces to the standard definition in the literature under more restrictive assumptions.

The standard definition of debiased Lasso in the high-dimensional regime can be found in \cite{javanmard2014hypothesis}. 
We focus on the Lasso case $\rho(t)=\lambda|t|$ for $\lambda \in [0,\infty)$, where the objective is convex but not strongly convex and lacks smoothness, leading to additional technical difficulties;
the case of smooth and strongly convex $\rho$ is discussed later and is technically simpler.
Suppose that $A$ has i.i.d.\ rows following the normal distribution $\mathcal{N}(0,\Sigma)$,
and $Y=A\alpha +w$, where $w$ is an independent  Gaussian noise vector, and $\alpha\in\mathbb{R}^p$.
The ``number of nonzero coefficients'' is defined by 
\begin{align}
k:=\|\chi^{\alpha}\|_0,
\label{e_k}
\end{align} 
where
\begin{align}
\psi^{\alpha}&:=\frac1{n\lambda}A^{\top}R,
\quad
R:=Y-A\hat{\alpha},
\label{e_psi}
\\
\chi^{\alpha}_j&:=1\{\psi^{\alpha}_j=1\}
-
1\{\psi^{\alpha}_j=-1\}.
\label{e_chi}
\end{align}
Note that $\psi^{\alpha}$ is the subgradient of the $\ell_1$ norm, 
so $k\ge \|\hat{\alpha}\|_0$, although equality is achieved in most cases.
We use the definition \eqref{e_k} instead of $\|\hat{\alpha}\|_0$ since $\hat{\alpha}$ may not be unique, due to the lack of strong convexity of the optimization,
whereas $\psi^{\alpha}$, and hence $\chi^{\alpha}$, is always uniquely defined.
Similarly $\chi^{\alpha}_j$ can be understood as the `essential sign' of $\hat{\alpha}_j$.
Then, 
the debiased Lasso defined in \cite{javanmard2014hypothesis} is 
\begin{align}
\hat{\alpha}^u:=\hat{\alpha}+\frac1{n-k}\Sigma^{-1}A^{\top}(Y-A\hat{\alpha}).
\label{eq_betau}
\end{align}
Under suitable conditions, it has been shown that 
$\hat{\alpha}^u\approx \alpha+ \tau \Sigma^{-1/2}z$,
where $z\sim \mathcal{N}(0,I)$ and $\tau\in(0,\infty)$ is a constant determined by a set of fixed point equations
\citep{javanmard2014hypothesis}.
Rigorously establishing such Gaussian limit properties for general non-Gaussian $A$ in the proportional growth regime ($p$, $n$, and the sparsity level $s$ have fixed ratios) is open; see discussions Section~\ref{sec_related}.

\comment{
and $\Sigma^{-1}$ is a suitable notion of the precision matrix (which can be taken as the inverse of the covariance matrix of the row vector in the case of random design with i.i.d.\ rows).
Equivalently, we can see that $\hat{\beta}^u$ is such that
$\argmin_{\beta\in\mathbb{R}^p}\{
\frac1{2}(1-\frac{k}{n})\|\hat{\beta}^u-\hat{\beta}\|^2_{\Sigma}+\lambda\|\hat{\beta}\|_1\}
$ coincides with the solution to \eqref{eq_beta}.
}

In this work, we propose a generalization of  \eqref{eq_betau} that extends its applicability under more relaxed assumptions. This generalization is similar to the orthogonality principle in the semiparametric and double machine learning literature \citep{chernozhukov2018double},
but we will discuss the differences in Section~\ref{sec_related}.
Let $A_{:\sj}\in\mathbb{R}^{n\times p}$ denote the matrix obtained by excluding the $j$-th column of $A$. 
Define the ``residualized column''
\begin{align}
\check{A}_{:j}&:=A_{:j}-\mu_{:j}.
\label{eq68_1}
\end{align}
where $\mu_{:j}\in\mathbb{R}^n$ is an arbitrary vector.
To control the error of the approximate formula, we will need to choose $\mu_{:j}$ so that the residualized column is ``orthogonal'' to other columns in a sense that will be made precise in \eqref{eq_cond3} (specific examples are provided in the itemized below).
We then define a generalization of the debiased estimator in \eqref{eq_betau}:

\begin{align}
\hat{\alpha}_j^U
:=\hat{\alpha}_j+
\left(
\frac1{n}\check{A}_{:j}^{\top}
(I-P_{\mathcal{A}})A_{:j}
\right)^{-1}
\frac{\check{A}_{:j}^{\top}R}{n},
\label{eq_alpha_U}
\end{align}
where
\begin{align}
\mathcal{A}:=\{l\neq j\colon \chi^{\alpha}_l\neq 0\},
\end{align}
and $P_{\mathcal{A}}$
denotes the projection onto the columns of $A$ corresponding to $\mathcal{A}$ in the case of Lasso,
and $P_{\mathcal{A}}:=A_{:\sj}[A_{:\sj}^{\top}A_{:\sj}+\diag(\rho''(\hat{\alpha}_{\sj}))]^{-1}A_{:\sj}^{\top}$
in the case of general $\rho$.
The inverse in \eqref{eq_alpha_U} is well-defined in the event of interest.
Again, we adopt the convention in \eqref{e_k} when $\hat{\alpha}$ is not unique.
Note that the definition of $\hat{\alpha}^U_j$ uses only $(Y,A)$ and has no reference to $B$, 
hence we can use $\hat{\alpha}^U_j$  to build $t(j,A,Y)$.
Specific examples of $\mu_{:j}$ and the corresponding $\hat{\alpha}^U$ include:

\begin{itemize}
\item 
If $n>p$, we can take $\mu_{:j}$ to be the projection of $A_{:j}$ on the column space of $A_{:\sj}$.
Then $\check{A}_{:j}^{\top}P_{\mathcal{A}}=0$,
and we have $\hat{\alpha}_j^U=
\hat{\alpha}_j+(\check{A}_{:j}^{\top}A_{:j})^{-1}\check{A}_{:j}^{\top}(Y-A\hat{\alpha})
=\|\check{A}_{:j}^{\top}\|_2^{-2}\check{A}_{:j}^{\top}Y=[(A^{\top}A)^{-1}A^{\top}]_{j:}Y$ which is the ordinary least squares estimator.
In the specific case of orthogonal designs, 
 we simply take 
$\mu_{:j}=0$ 
(see Example~\ref{ex1}).
\item 
If $n<p$, the above choice of $\mu_{:j}$ leads to a degenerate definition of $\hat{\alpha}_j^U$.
However, due to the sparsity of $\hat{\alpha}$, other choices of $\mu_{:j}$ may still
ensure \eqref{eq_cond3}.
For the random design case where the rows of $A$ are independent,
we can choose 
$\mu_{:j}=\mathbb{E}[A_{:j}|A_{:\sj}]$. 
Specifically for Gaussian $A$, we have 
$
\check{A}_{:j}:=A_{:j}-A_{:\setminus j}\Sigma_{\setminus j}^{-1}\Sigma_{\sj j}
$ 
in terms of the covariance matrix $\Sigma$.
\end{itemize}

\paragraph*{Contributions}
We will show that
$
\hat{\alpha}^U_j\approx \hat{\alpha}^u_j
$
under suitable assumptions,
but the definition of $\hat{\alpha}^U_j$
is applicable in more general settings.
Under increasingly stronger assumptions, 
we establish the following approximation results:

\begin{itemize}
\item Explicit nonasymptotic error bounds for any given design matrix:
Define $\hat{\beta}^U_j$ analogously to $\hat{\alpha}^U_j$, i.e.
\begin{align}
\hat{\beta}_j^U
:=\hat{\beta}_j+
\left(
\frac1{n}\check{B}_{:j}^{\top}
(I-P_{\mathcal{B}})B_{:j}
\right)^{-1}
\frac{\check{B}_{:j}^{\top}S}{n},
\label{eq_beta_U}
\end{align}
where 
$\check{B}_{:j}:=B_{:j}-\mu_{:j}$,
$S:=Y-B\hat{\beta}$,
and $P_{\mathcal{B}}$ denotes the projection onto the columns of $A$ (equivalently, columns of $B$ since the two matrices differ only in the $j$-th column) corresponding to 
$
\mathcal{B}:=\{l\neq j\colon \chi^{\beta}_l\neq 0\}$.
For any given $(A,B,Y)$, 
we can show
\begin{align}
\check{B}_{:j}^{\top}
(I-P_{\mathcal{B}})
B_{:j}
\hat{\beta}^U_j
&\approx
\check{B}^{\top}_{:j}R
+
\check{B}^{\top}_{:j}(I-P_{\mathcal{A}})A_{:j}
\hat{\alpha}_j
\label{e_approx}
\end{align}
with error controlled in Theorem~\ref{thm_error}.
The linear combination on the right side of \eqref{eq_beta_U} orthogonalizes the estimator so that the approximation error in \eqref{e_approx} becomes the inner product of two incoherent vectors, 
canceling the leading error term.
In Section~\ref{sec_related}, we discuss the connection between this orthogonalization and the one in the double machine learning literature.
If $B_{:j}$ and $A_{:j}$ are independent conditioned on $A_{:\sj}$, we can further show that the right side of \eqref{e_approx} is approximately $\check{B}^{\top}_{:j}R$,
although the right side of  \eqref{e_approx} is already computable without using $\hat{\beta}$.
Thus, for any given $Y$, if we define
\begin{align}
t(j,A,Y)
&:=\frac1{n}\check{A}_{:j}^{\top}
(I-P_{\mathcal{A}})
A_{:j}\hat{\alpha}_j^U;
\label{e_tja}
\\
t(j,B,Y)
&:=\frac1{n}\check{B}_{:j}^{\top}
(I-P_{\mathcal{B}})
B_{:j}\hat{\beta}_j^U,
\label{e_tjb}
\end{align}
then we can compute both $t(j,A,Y)$ 
and approximately
$t(j,B,Y)$ 
using only $\hat{\alpha}$ and without $\hat{\beta}$.

\item Convergence for sub-Gaussian designs:
Consider the case where 
$Y=A\alpha+w$
and design matrices have i.i.d.\ rows with covariance of bounded max and min eigenvalues and with bounded sub-Gaussian variance proxy.
Also assume that $B_{:j}$ and $A_{:j}$ are i.i.d.\ given $A_{:\sj}$.
In this setting, we show that the approximation error in \eqref{e_approx} vanishes asymptotically for almost all $j$ (see Definition~\ref{defn_param} and Theorem~\ref{thm12}).
Our proof only uses certain concentration and anti-concentration properties to give order-wise control of quantities, rather than more precise calculation of limits,
which often requires stronger assumptions (see Section~\ref{sec_related}).

\item 
Reduction to $\hat{\beta}_j^u$:
Further assuming that
$\frac1{n}\check{B}_{:j}^{\top}
(I-P_{\mathcal{A}})B_{:j}$
is bounded away from 0
(which is true when 
$\mathbb{E}[|\check{B}_{1j}|^2]$ is bounded away from 0; see \eqref{e_lb}),
we have from \eqref{e_approx} that
\begin{align}
\hat{\beta}^U_j
&\approx
\frac{\check{B}^{\top}_{:j}R
+
\check{B}^{\top}_{:j}(I-P_{\mathcal{A}})A_{:j}
\hat{\alpha}_j}
{\check{B}_{:j}^{\top}
(I-P_{\mathcal{A}})
B_{:j}}.
\label{e_approx_na}
\end{align}
If we further have 
\begin{align}
\frac1{n}\check{B}_{:j}^{\top}
(I-P_{\mathcal{A}})B_{:j}
\approx 
(1-\frac{k}{n})\Sigma_{j|\sj},
\label{e_app_chi}
\end{align}
then $\hat{\beta}_j^U\approx \hat{\beta}_j^u$.
For example, \eqref{e_app_chi} is true in the case of Gaussian feature vector, 
as a consequence of concentration of the chi-square distribution; see Theorem~\ref{thm_gaussian}
and Remark~\ref{rem10}).
In more general settings, 
\eqref{e_app_chi}
may no longer hold,
since $\mathbb{E}[B_{:j}|A_{:\sj}]=A_{:\sj}\Sigma_{\sj}^{-1}\Sigma_{\sj j}$ may fail (see Remark~\ref{rem_4}).
This is our main motivation for introducing $\hat{\beta}_j^U$.
\end{itemize}

We apply the approximation formulas to accelerate resampling-based variable selection procedures. 
Consider the proportional regime where $n$ grows linearly with $p$, and suppose that solving the Lasso requires computational cost $L$ (for example, $L = O(p^3)$ via least angle regression \cite[p.~93]{hastie2009elements}). 
The knockoff filter requires solving a single regression and hence has complexity $\Theta(L)$, but its statistical power can be reduced due to the increase in the number of variables from $p$ to $2p$. 
We consider two alternative methods that avoid regressing on $2p$ variables but require solving multiple regression problems, and show that their computational cost can be substantially reduced using the proposed approximation formulas:
\begin{itemize}
\item We consider ``local knockoff filter'', a variant of the knockoff method that resamples only a single feature in each Lasso run.
A concurrent work, \cite{guan2025one}, proposed this strategy under the name ``one-at-a-time knockoffs'' and showed experimentally that it is ``substantially more powerful.'' 
While \cite{guan2025one} focused on least squares and ridge regression, exploiting the explicit solution and rank-one matrix update formulas, 
our work extends the method to Lasso regression while maintaining computational complexity $\Theta(L + p^3)$.

\item 
Conditional randomization test (CRT) iteratively updates each feature vector (column of the design matrix) with $K$ conditionally independent sample and computes the corresponding test statistics (e.g., $K=\Theta(p)$ in \cite{candes2018panning}).
When CRT is applied to the variable selection problem, its asymptotic power is often observed to be worse than that of CRT \citep{candes2018panning,li2022causal}.
However, it has a complexity of $\Theta(LpK)$.
To reduce the complexity, distilled CRT has been proposed,
and we will see in Section~\ref{sec_resample}
that $\rm d_0CRT$ from 
\cite{liu2022fast} is asymptotically equivalent to CRT using the debiased coefficients.
Implementing $\rm d_0CRT$ for FDR control requires complexity at least $\Omega(pL)$, for solving $p$ leave-one-out regression problems.
In contrast, using \eqref{e_approx} the complexity is reduced to $O(L+p^2K)$.
\end{itemize}

Since previously the debiased Lasso $\hat{\alpha}^u$ often appears in the literature on asymptotic normality,
and asymptotic normality results can be used to directly estimate the $p$-value of $\hat{\alpha}^u$,
one might ask what is the benefit of resampling $B_{:j}$ and using $\hat{\beta}^U_j$ to estimate the $p$-values.
The answer is that asymptotic normality requires more stringent conditions
than the validity of the update formula.
One simple example is the limiting case where the Lasso is reduced to a least square problem ($n>p$ and $\lambda\to 0$).
In this case, \eqref{e_approx} is in fact \emph{exact equality} regardless of the distributions and the dimensions, whereas normality is approximate and requires distributional assumptions for asymptotic convergence.
For the general $\lambda$ case, 
Example~\ref{ex1} ahead shows that asymptotic normality is false when the column of the design matrix contains a non-vanishing component,
yet our main result
(Theorem~\ref{thm_error}) is still capable of controlling the error of the update formula.
For the case of i.i.d.\ rows of the design matrix, 
as mentioned before,
our proof of the approximation in \eqref{e_approx} only uses certain concentration and anti-concentration properties to control the \emph{order} of the errors, 
rather than more precise characterization of limits such as Gaussian convergence.
Indeed, our Theorem~\ref{thm12} shows asymptotic approximation assuming that the covariance matrix $\Sigma$ of $B_{1:}$ has bounded conditional numbers, and that $B_{1:}$ is a sub-Gaussian vector.
In contrast, a Gaussian limit result for $\hat{\alpha}^u$ with correlated sub-Gaussian designs has been noted as an open problem (see Section~\ref{sec_related}).

\paragraph*{Organization}
Further related work is discussed in Section~\ref{sec_related}.
In Section~\ref{sec_main},
we present main results on general nonasymptotic error bounds and asymptotic analysis for the sub-Gaussian case.
Section~\ref{sec_aformula} proves the nonasymptotic error bound.
In Section~\ref{sec_fdr},
applies the approximation formulas in accelerating resampling-based variable selection algorithms.
Section~\ref{sec_conc} concludes with an outlook for future directions.
In the appendices, we present omitted proofs,
experimental validations of the approximation errors in the debiased and the standard Lasso,
and FDR control for synthetic and semi-real (riboflavin and HIV) data.

\section{Related work}
\label{sec_related}
\subsection{Definitions of the debiased estimator}
Debiasing the Lasso for inference was suggested by \cite{zhang2014confidence,buhlmann2013statistical,van2014asymptotically,javanmard2014hypothesis}.
The replica analysis heuristic calculation in \cite{javanmard2014hypothesis} was perhaps the first to show that $\alpha^u$ in \eqref{eq_betau} satisfies asymptotic normality in the proportional growth regime, with i.i.d.\ $\mathcal{N}(0,\Sigma)$ rows in the design matrix.
More specifically, in a suitable sense there is the approximation
\begin{align}
\hat{\alpha}^u-\alpha
\approx \tau\Sigma^{-1/2}z
\label{e_gapprox}
\end{align}
for some $z\in\mathcal{N}(0,I)$,
where $\alpha$ is the ground truth, $Y=A\alpha+w$, $w\sim \mathcal{N}(0,nI)$, and $\tau$ is the solution to a fixed point equation.
The replica calculation, as well as other existing proof techniques for \eqref{e_gapprox} in the proportional regime (see Section~\ref{sec23}),
are limited to the random design settings, and require
more restricted assumptions than our proof of \eqref{e_approx}.

In the landmark paper \cite{chernozhukov2018double},
a generalized debiased estimator similar in spirit to \eqref{e_approx} was proposed, 
which was dubbed 
``double machine learning'' since it involves regressing the outcome on the controls  
and regressing the treatment on the controls ($\mu_{:j}$ in \eqref{eq68_1}).
To explain the idea, suppose that $p$ is fixed, $n\to\infty$, and the log likelihood is given by $\ell(y,b,\beta)=-\frac1{2}(y-b^{\top}\beta)^2$ where $b\in\mathbb{R}^p$.
The approach of \cite{chernozhukov2018double}
(see also \cite{chernozhukov2017double})
then defines a score $\psi(y,b,\beta,\mu):=\partial_{\beta_j}\ell(y,b,\beta)-\mu\partial_{\beta_{\sj}}\ell(y,b,\beta)$ satisfying a \emph{Neyman orthogonality condition},
ensuring first-order insensitivity with respect to the parameters $(\mu,\beta_{\sj})$.
Data points indexed by the complement $I^c$ of a set $I\subseteq \{1,\dots,n\}$ are used to learn $\hat{\beta}_{\sj}(I^c)$ and $\mu(I^c)=
\sum_{i\in I^c}B_{ij}^{\top}B_{i\sj}(\sum_{i\in I^c}B_{i\sj}^{\top}B_{i\sj})^{-1}$,
the latter following from the orthogonality condition.
Then $\psi(y,b,\hat{\beta}_{\sj}(I^c),\beta_j,\mu(I^c))
=(b_j-\mu(I^c)b_{\sj})(y-b_{\sj}^{\top}\hat{\beta}_{\sj}(I^c)-b_j\beta_j)$,
and the debiased estimator is defined by solving 
$\sum_{i\in I}\psi(Y_i,B_{i:},\hat{\beta}_{\sj}(I^c),\beta_j,\mu(I^c))=0$ for $\beta_j$,
which is 
\begin{align}
\frac{\sum_{i\in I}(B_{ij}-\mu(I^c)B_{i\sj}^{\top})(Y_i-B_{i\sj}\hat{\beta}_{\sj}(I^c))}{\sum_{i\in I}(B_{ij}-\mu(I^c)B_{i\sj}^{\top})B_{ij}}.
\label{e17}
\end{align}
If the sample-splitting estimator $\hat{\beta}_{\sj}(I^c)$ in \eqref{e17} is replaced by the leave-one-out estimator, \eqref{e17} becomes similar to the right side of
\eqref{e_approx} (note that $\hat{\alpha}_j=0$ in  \eqref{e_approx} when $A_{:j}=0$). 
However, \eqref{e_approx} also has a factor $I-P_{\mathcal{B}}$ arising from the $\ell_1$ penalization, 
which is unique for the high-dimensional (proportional) regime and was not discussed in \cite{chernozhukov2018double}.
Our result \eqref{e_approx} establishes the approximate equivalence of orthogonalized estimators 
derived from two seemingly different principles:
\begin{itemize}
\item \emph{Robustness} is the principle behind constructing debiased estimators in double machine learning. 
In \eqref{e17}, if the nuisance parameter estimator $\hat{\beta}_{\sj}(I^c)$ has an error, the error in $Y_i-B_{i\sj}\hat{\beta}_{\sj}(I^c)$ will be in $\col(B_{i\sj})_{i\in I}$,
which is incoherent with 
$(B_{ij}-\mu(I^c)B_{i\sj}^{\top})_{i\in I}$, ensuring that \eqref{e17} is robust.
\item \emph{Stability} in the literature usually refers to the regularity of a function under the perturbation in one entry. 
In \eqref{e_approx}, the observations in one feature are changed, 
yet the estimator can be updated in a well-controlled manner,
manifesting a stability property.
\end{itemize}

\subsection{Leave-one-out analysis}
Leave-one-out analysis is a fruitful approach for establishing limiting distributions or algorithmic properties of regression
\citep{elkaroui2013robust,ma2018implicit,chen2020noisy}, 
and is closely related to techniques of the present paper.
In \cite{elkaroui2013robust}, 
it is shown using the leave-one-out technique that the M-estimator converges asymptotically to a normal distribution  (see also \cite{el2018impact} and \cite{lei2018asymptotics}).
The problem considered there is different from the distribution of the Lasso considered in the present paper: the M-estimation problem concerns the $n>p$ regime, and there is no need for debiasing; the asymptotic normality follows immediately from the rotation invariance of the distribution.
We remark that a duality between M-estimation estimation and penalized least squares was mentioned in \cite{donoho2016high}.
However, the duality only applies when the design matrix of the lasso has orthonormal rows, which does not cover the setting of the present paper.

A leave-one-out analysis for the Lasso was carried out in \cite{javanmard2018debiasing}.
In addition to bounded singular values of $\Sigma$, 
their analysis requires bounded $\ell_1$ norms of the rows of the inverses of the submatrices of $\Sigma$ 
(see \cite[Theorem~3.8]{javanmard2018debiasing}).
The latter condition can be more restrictive than ours in Definition~\ref{defn_param}: for example a $k\times k$ random matrix with independent entries of scale $1/\sqrt{k}$ has spectral norm 
of order $O(1)$, yet the $\ell_1$ norm of each of its row has order $\Omega(\sqrt{k})$ which is unbounded.
Furthermore,  \citet[Theorem~3.8]{javanmard2018debiasing} requires a sublinear sparsity level $s_0=o(n/(\ln p)^2)$.
In that regime, there is no need for the degrees of freedom adjustment factor $\frac1{n-k}$ in \eqref{eq_betau},
and in fact in the approximation formula \eqref{e_approx} it suffices to replace $R$ with the noise $w$ (see \citet[eq.~(61)]{javanmard2018debiasing}).
An extension of the analysis was done in
\cite{bellec2022biasing}, 
where the role of degrees-of-freedom adjustment was highlighted for sparsity level $s_0\gg n^{2/3}$, but still $s_0/p\to 0$ is required.
In contrast, the present paper considers the regime of proportional sparsity level.

\subsection{Asymptotic normality in the high-dimensional regimes}
\label{sec23}
It appears that the first asymptotic normality result for debiased Lasso estimates in the proportional regime for correlated designs was derived in \cite{bellec2019biasing} (see the discussions therein).
The technique of \cite{bellec2019biasing} (see also \cite{bellec2021second}) was based on the \emph{Second Order Stein} theorems bounding the non-Gaussianity of a random variable of the form $z^{\top}f(z)-{\rm div}f(z)$,
where $z\sim\mathcal{N}(0,I)$.
To apply it to the debiased Lasso problem, 
consider given $A_{:\sj}$ (the submatrix of $B$ formed by excluding the $j$-th column), ground truth $\alpha$ and noise $w$.
Let $A_{:j}:=A_{:\sj}\Sigma_{\sj}^{-1}\Sigma_{\sj j}+z\sqrt{(\Sigma^{-1})_{jj}^{-1}}$ and $Y=A\alpha+w$, which are both viewed as functions of $z$.
Let $f(z):=A\hat{\alpha}-Y$.
Then it can be verified that $z^{\top}f(z)-{\rm div}f(z)$ is the debiased Lasso estimate up to a linear transform. 
The method of \cite{bellec2019biasing} made essential uses of the Gaussian random design assumption, e.g.\  Gaussian integration by parts.

Gaussian comparison is another powerful approach for deriving the asymptotic distribution of the Lasso. Building on an earlier idea of \cite{thrampoulidis2015regularized} that constructs a simpler but comparable Gaussian process, 
\cite{miolane2021distribution} proved asymptotic normality of \eqref{eq_betau} (in the Wasserstein distance in $\mathbb{R}^p$) for i.i.d.\ $\mathcal{N}(0,I)$ rows, and \cite{celentano2020lasso} extended the result to i.i.d.\ $\mathcal{N}(0,\Sigma)$ rows.
By nature, the Gaussian comparison argument strongly relies on the Gaussianity of the design matrix.

Characterizing the asymptotic distribution of the Lasso for dependent non-Gaussian designs  is an open challenge
(see comments in \cite{montanari2022universality} and \citet[Remark 4.2]{celentano2020lasso}).
Proof of 
universality based on the Lindeberg-type argument typically assumes  independent entries
\citep{han2023universality,aubin2020generalization}.

Approximate message passing (AMP) is not only an algorithm but also a method of characterizing asymptotic distributions. 
The most common approach for analyzing the state evolution of AMP is through a conditioning technique, 
which shows that vector approximate message passing
works for 
design matrices with a general spectrum but satisfying
right-rotational invariance
\citep{schniter2016vector,fan2022approximate,li2023random,zhong2021approximate}.
In particular, rotation invariance implies that the feature distribution is permutation invariant,
which does not subsume our setting.
Other representative approaches for AMP analysis \citep{bao2023leave,li2024non} assume independent matrix entries.

\subsection{Resampling-based variable selection methods}
\label{sec_resample}
Traditionally, the most well-known variable selection method with guaranteed false discovery rate (FDR) control is the Benjamini-Hochberg procedure 
\citep{benjamini1995controlling}, which typically assumes that the $p$-values are independent or positively correlated.
The knockoff filter \citep{barber2015controlling,candes2018panning} is a  recent popular approach that controls the FDR without such restrictive dependency assumptions.
Intuitively, the knockoff filter creates knockoff features which have the same distribution as the true features, but are conditionally independent of the response,
so that the knockoff statistics can be used as a benchmark/control for understanding the $p$-values. 
Remarkably, the knockoff filter extends such an intuition by offering provable finite sample FDR control via an elegant martingale analysis \citep{barber2015controlling}.

The knockoff filter performs regression on $2p$ features and is often observed to have lower statistical power than methods that resample one feature at a time, such as the conditional randomization test (CRT),
and Gaussian mirror \citep{xing2023controlling}. 
CRT is primarily a hypothesis testing procedure. While it can be used to compute $p$-values and achieve FDR control when combined with BHq, it is computationally intensive in this setting.
The holdout randomization test (HRT; \citep{tansey2022holdout}) improves the efficiency of CRT via sample splitting, at the cost of potential power loss.
Distilled CRT \citep{liu2022fast} is a well-known approach for mitigating the computational cost of CRT.
Assuming that the $j$-th feature is to be tested,
we can express their test statistic (see \citet[Section~2.3]{liu2022fast}) as 
\begin{align}
\frac{|(Y-d_y)^{\top}(B_{:j}-d_x)|}{\|B_{:j}-d_x\|_2^2}
\end{align}
where we follow the notation of \cite{liu2022fast} to denote by $d_x$ and $d_y$ the distilled information. 
According to \citet[Section~2.1]{liu2022fast} we can take $d_x=\mathbb{E}[B_{:j}|B_{:\sj}]$ and $d_y=A\hat{\alpha}$, 
where $A_{:j}=0$ and $A_{:\sj}=B_{:\sj}$ (that is, $A\hat{\alpha}$ is the leave-one-out estimator).
Then $\hat{\alpha}=0$ and $(Y-d_y)^{\top}(B_{:j}-d_x)$ is identical to the right side of \eqref{e_approx}.
Thus our results can be interpreted as rigorously proving the asymptotic equivalence between the debiased estimator and the distilled CRT statistics in the nontrivial proportional regime, which was not covered in \cite{liu2022fast} or \cite{chernozhukov2018double}.
Furthermore, we can take $A$ in \eqref{e_approx} to be the original design matrix (rather than the leave-one-out matrix).
This obviates the need of solving the leave-one-out Lasso for each $j$,
reducing the computational complexity of the original $\rm d_0CRT$ in \cite{liu2022fast} by a factor of $p$.

\section{Main results on approximate formulas}
\label{sec_main}
\paragraph*{Notations} 
We use the standard Landau notations such as $O(n)$, $\Omega(n)$, and $\omega(1)$.
The notation $\tilde{O}(f(n))=O(f(n)(\ln n)^C)$ indicates an upper bound up to a factor of a polynomial of $\ln n$.
To emphasize the dependence on a set of parameters $\mathcal{P}$ in the implicit constants, we may write $O_{\mathcal{P}}()$ or $\tilde{O}_{\mathcal{P}}()$.
The $L_2$ norm and  the operator norm are denoted by $\|\cdot\|_2$ and $\|\cdot\|_{\rm op}$, whereas $\|\cdot \|_0$ denotes the number of nonzero coefficients.
$\lambda_{\max}$ and $\lambda_{\min}$ denote the largest and smallest singular values of a matrix.
$\col()$ denotes the column space of a matrix.
The projection matrix onto a set $\mathcal{A}$ of columns in the design matrix is denoted  $P_{\mathcal{A}}$.
The max and min of two numbers $a,b$ are denoted $a\vee b$ and $a\wedge b$.
For $a,b\in\mathbb{R}^p$, write the empirical distribution $\hat{P}_{ab}:=\frac1{p}\sum_{j=1}^p\delta_{(a_j,b_j)}$.
We use $A_{:\sj}$ to denote the submatrix of $A$ formed by all except the $j$-th column.
The standard normal distribution in $\mathbb{R}^d$ is written as
$\mathcal{N}(0,I_n)$.
$\diag(\cdot)$ can denote either a diagonal matrix with diagonals specified by a vector, or the vector formed by the diagonal values of a matrix.

\subsection{Nonasymptotic error bounds}\label{sec_approx}
Suppose that $A,B\in\mathbb{R}^{n\times p}$ are matrices differing only in the $j$-th column. 
Recall that we define $\hat{\alpha}$ and $\hat{\beta}$ as the solutions to \eqref{eq_alpha}-\eqref{eq_beta} (any choice of the minimizer when the minimizer is not unique),
and $\hat{\alpha}^U_j$ and $\hat{\beta}^U_j$ in \eqref{eq_alpha_U} and \eqref{eq_beta_U}.
Let $\mu_{:j}\in\mathbb{R}^n$ be an arbitrary vector, and  
define 
\begin{align}\check{B}_{:j}:=B_{:j}-\mu_{:j}.
\end{align}
Our main nonasymptotic result is the following:

\begin{thm}\label{thm_error}
Let $Y\in\mathbb{R}^n$ be arbitrary, let $\hat{\alpha}$ and $\hat{\beta}$ be as defined in \eqref{eq_alpha}-\eqref{eq_beta},
and consider $t(\cdot)$ defined in \eqref{e_tja}.
Set $\mathcal{J}:=\{l\colon \chi^{\alpha}_l\neq \chi^{\beta}_l\}$,
and suppose that 
\begin{align}
|\mathcal{J}|&\le \varepsilon n;
\label{eq_cond1}
\\
\|B_{:j}\|_2, \|A_{:j}\|_2&\le D\sqrt{n};
\label{e19}
\\
\max_{\Delta}\|(P_{\mathcal{A}\cup\Delta}-P_{\mathcal{A}\setminus\Delta})\check{B}_{:j}\|_2&\le \Gamma \sqrt{n\varepsilon},
\label{eq_cond3}
\end{align}
where $\mathcal{A}:=\{l\neq j\colon \chi^{\alpha}_l\neq 0\}$, and the max is over $\Delta\subseteq\{1,\dots,p\}\setminus\{j\}$ of size at most $n\varepsilon$,
and $\Gamma,D>0$.
Then 
\begin{align}
\left| 
t(j,B,Y)-\frac1{n}\check{B}_{:j}^{\top}R
-
\frac1{n}\check{B}_{:j}^{\top}(I-P_{\mathcal{A}})A_{:j}
\hat{\alpha}_j
\right|
\le \Gamma D\sqrt{\varepsilon}
(|\hat{\alpha}_j|+|\hat{\beta}_j|).
\label{e22}
\end{align}
Moreover, if $t(j,B,Y)$ above is replaced by $\frac1{n}\check{B}_{:j}^{\top}(I-P_{\mathcal{A}})B_{:j}\hat{\beta}_j^U$,
we can still bound the left side by $\Gamma D\sqrt{\varepsilon}
(|\hat{\alpha}_j|+|\hat{\beta}_j|+|\hat{\beta}_j^U|)$.
\end{thm}

The proof the theorem is given in Section~\ref{sec_aformula}.
Theorem~\ref{thm_error} suggests the formula \eqref{e_approx_na}
for fast calculation of the debiased estimator when the design matrix is updated by one-column, 
since the right side only depends on the result of solving \eqref{eq_alpha}.
Note that Theorem~\ref{thm_error} applies to any given $B$ as long as conditions \eqref{eq_cond1}-\eqref{eq_cond3} are satisfied.
For deterministic designs, 
$\mu_{:j}$ can be taken as \emph{any} vector that ensures \eqref{eq_cond3}.

A basic example of Theorem~\ref{thm_error} is simply $A=B$, in which case $\varepsilon=0$ and $\Gamma=0$.
Then \eqref{e22} simply recovers our definition of the debiased estimator.
As a more advanced example,
consider random designs where $A$ and $B$ are conditionally independent given $A_{:\sj}$. 
Then, we can take $D$ as a constant independent of $n$,
$\Gamma $ as a slowly growing function (e.g.\ polylog of $n$),
and $\varepsilon$ vanishing in $n$. 
Moreover, let 
\begin{align}\mu_{:j}:=\mathbb{E}[B_{:j}|B_{:\sj}]=\mathbb{E}[B_{:j}|A_{:\sj}],
\end{align}
so that \eqref{eq_cond3} is satisfied with high probability when $\varepsilon$ is small, because $\mathcal{A}$ is a function of $(A,Y)$ whereas $\check{B}_{:j}$ is a zero mean vector conditioned on $(A,Y)$.
More precise results will be discussed in Section~\ref{sec_asymp}.

\begin{example}\label{ex1}
We provide a nontrivial example where the update formula is correct as predicted by Theorem~\ref{thm_error}, yet asymptotic normality is false.
Consider a sequence (indexed by $n$) of instances where 
$p=p(n):=\lfloor n/\delta\rfloor$ for some $\delta\in(1,\infty)$, 
$\frac1{\sqrt{n}}w_1,\dots,\frac1{\sqrt{n}}w_n$ are i.i.d.\ according to a certain non-Gaussian distribution, and 
$Y=A\alpha+w$ for some deterministic $\alpha\in\mathbb{R}^p$.
Furthermore, assume that $A$ is independent of $w$ and satisfies the following:
\begin{description}
    \item[C1] 
    $\frac1{n}A^{\top}A=I_p$;
\item[C2] $\liminf_{n\to\infty}\frac1{\sqrt{n}}\max_{1\le i\le n}|A_{i1}|>0$. 
\end{description}
We sample $B_{:1}$ i.i.d.\ as $A_{:1}$ conditioned on $A_{:\setminus 1}$.
Note that for this orthogonal design case,
we have the explicit formulas $\hat{\alpha}=S_{\lambda}(\frac1{n}A^{\top}Y)$ and
$\hat{\beta}=S_{\lambda}(\frac1{n}B^{\top}Y)$,
where $S_{\lambda}$ denotes the elementwise soft-thresholding operator.
Thus the signs of $\hat{\alpha}$ and $\hat{\beta}$ can differ only at $j=1$, 
so that \eqref{eq_cond1} and \eqref{e19} hold with $\varepsilon=\frac1{n}$ and 
$D=1$. 
We further take 
$
\mu=0 
$
so that $\check{A}_{:1}=A_{:1}$ and $\check{B}_{:1}=B_{:1}$.
Then by C1 we see that 
\eqref{eq_cond3} holds with $\Gamma=0$.
Thus \eqref{e22} shows that the update formula $t(1,B,Y)=\frac1{n}\check{B}_{:1}^{\top}R
+
\frac1{n}\check{B}_{:1}^{\top}(I-P_{\mathcal{A}})A_{:1}
\hat{\alpha}_j$ is exact
(no error).
On the other hand, note that 
$
\frac1{n}\check{A}_{:1}^{\top}
(I-P_{\mathcal{A}})A_{:1}
=\frac1{n}A_{:1}^{\top}A_{:1}=1$ and hence by
\eqref{eq_alpha_U}
we obtain
$\hat{\alpha}_1^U=\hat{\alpha}_1+\frac1{n}A_{:1}^{\top}R=\frac1{n}A_{:1}^{\top}Y$, which is simply the least squares estimator.
However it can be shown that $\frac1{n}A_{:1}^{\top}Y$ is not asymptotically Gaussian (either conditioned on $A$ or not), due to C2 above and the non-Gaussianity of the components of $w$. 
\end{example}

The update formula is closely related to the leave-one-out analysis and the asymptotic normality of debiased Lasso, e.g.\ 
\cite{javanmard2018debiasing}, \cite{bellec2019biasing} and \cite{bellec2022biasing}.
To see this, 
observe that by slightly changing the proof of Theorem~\ref{thm_error} to allow different observation vector $Y$ in the two Lasso problems,
we obtain:
\begin{thm}
\label{thm5}
Let $\alpha\in\mathbb{R}^p$ and $w\in\mathbb{R}^n$.
Suppose that $A$ and $B$ differ only in the $j$-th column.
Define
\begin{align}
\hat{\beta}:&=\argmin_{\gamma\in\mathbb{R}^p}
\left\{\frac1{2n}\|A_{:\sj}\alpha_{\sj}+w-B\gamma\|_2^2+\lambda\|\gamma\|_1\right\};
\\
\hat{\alpha}:&=\argmin_{\gamma\in\mathbb{R}^p}
\left\{\frac1{2n}\|A\alpha+w-A\gamma\|_2^2+\lambda\|\gamma\|_1\right\}.
\end{align}
Let $\mu_{:j}\in\mathbb{R}^n$ be arbitrary, 
define 
$\check{A}_{:j}:=A_{:j}-\mu_{:j}$,
$\mathcal{J}:=\{l\colon \chi^{\alpha}_l\neq \chi^{\beta}_l\}$,
and suppose that 
we have 
\eqref{eq_cond1}, 
\eqref{e19},
and 
\begin{align}
\max_{\Delta}\|(P_{\mathcal{B}\cup\Delta}-P_{\mathcal{B}\setminus\Delta})\check{A}_{:j}\|_2&\le \Gamma \sqrt{n\varepsilon},
\end{align}
where $\mathcal{B}:=\{l\neq j\colon \chi^{\alpha}_l\neq 0\}$, and the max is over $\Delta\subseteq\{1,\dots,p\}\setminus\{j\}$ of size at most $n\varepsilon$,
for some $\Gamma >1$ and $D>1$.
Then
\begin{align}
&\quad\left|\frac1{n}\check{A}_{:j}^{\top}(I-P_{\mathcal{B}})A_{:j}
(\hat{\alpha}_j^U
-\alpha_j)
-\frac1{n}\check{A}_{:j}^{\top}S
-
\frac1{n}\check{A}_{:j}^{\top}(I-P_{\mathcal{B}})B_{:j}
\hat{\beta}_j
\right|
\nonumber\\
&\le D\Gamma \sqrt{\varepsilon}
(|\hat{\beta}_j|+|\hat{\alpha}_j-\alpha_j|+|\hat{\alpha}_j^U-\alpha_j|).
\label{eq65}
\end{align}
\end{thm}
The proof can be found in Section~\ref{pfthm5}.
Observe that if $A$ and $B$ are independent conditioned on $A_{:\sj}$, \eqref{eq65} suggests that 
\begin{align}
\hat{\alpha}_j^U
-\alpha_j
\approx
\frac{\check{A}_{:j}^{\top}S}
{\check{A}_{:j}^{\top}(I-P_{\mathcal{B}})A_{:j}}
\approx
\frac{\check{A}_{:j}^{\top}S}
{(n-k)\Sigma_{j|\sj}}
\label{e_approxa}
\end{align}
which approximately follows $\mathcal{N}(0,(n-k)^{-2}\Sigma_{j|\sj}^{-1}\|S\|_2^2)$ under appropriate regularity conditions.
This recovers the asymptotic normality of the debiased Lasso previously proved by other methods; see
Section~\ref{sec_intro}.

In contrast to its debiased version, the Lasso estimator has update formula in a more restricted setting:
\begin{thm}\label{thm_approx2}
Suppose that for some $\tau\in(0,1)$, $\Gamma,\check{D}>0$,
we have \eqref{eq_cond1},
\eqref{eq_cond3},
and 
\begin{align}
\frac1{n}\|\mu_{:j}\|_2^2
&\le \tau^2;
\\
\|\check{B}_{:j}\|_2,
\|\check{A}_{:j}\|_2&\le \check{D}\sqrt{n}.
\end{align}
Then
\begin{align}
&\quad \left|\frac1{n}B_{:j}^{\top}(I-P_{\mathcal{A}})B_{:j}\hat{\beta}_j
-
\mathsf{S}_{\lambda}
\left(
\frac1{n}B_{:j}^{\top}R
+
\frac1{n}B_{:j}^{\top}(I-P_{\mathcal{A}})A_{:j}
\hat{\alpha}_j
\right)
\right|
\nonumber\\
&\le 
(\tau^2+2\tau\Gamma\sqrt{\varepsilon}+\varepsilon\Gamma^2)|\hat{\beta}_j|
+
(\tau^2+\tau(\Gamma\sqrt{\varepsilon}+\check{D})+\Gamma\sqrt{\varepsilon}\check{D})|\hat{\alpha}_j|,
\end{align}
where ${\sf S}_{\lambda}(x):=(x-\lambda)1_{x>\lambda}
-(x+\lambda)1_{x<-\lambda}$ is the soft-thresholding function.
\end{thm}

The proof the theorem is given in Section~\ref{sec_aformula}.
Theorem~\ref{thm_approx2} suggests the approximate formula 
\begin{align}
\hat{\beta}_j
\approx
\left(\frac1{n}B_{:j}^{\top}(I-P_{\mathcal{A}})B_{:j}
\right)^{-1}
\mathsf{S}_{\lambda}
\left(
\frac1{n}B_{:j}^{\top}R
+
\frac1{n}B_{:j}^{\top}(I-P_{\mathcal{A}})A_{:j}
\hat{\alpha}_j
\right).
\end{align}
The approximation is good if $\tau, \varepsilon\to0$. Note that $\tau\to 0$ implies that the features are approximately independent.
In the proof  of Theorem~\ref{thm_error},
error terms of the form $\frac1{n}\check{B}_{:j}^{\top}(P_{\mathcal{A}\cup\Delta}-P_{\mathcal{A}\setminus\Delta})A_{:j}$ asymptotically vanishes since $\check{B}_{:j}$ and $(P_{\mathcal{A}\cup\Delta}-P_{\mathcal{A}\setminus\Delta})A_{:j}$ are uncorrelated.
On the other hand, 
in the proof of Theorem~\ref{thm_approx2},
error terms of the form $\frac1{n}B_{:j}^{\top}(P_{\mathcal{A}\cup\Delta}-P_{\mathcal{A}\setminus\Delta})A_{:j}$ arise, and it is not necessarily vanishing unless the features are independent.
In the experiments, we also observe much better approximation of the debiased estimator under correlated features.

\subsection{Asymptotic error bounds}
\label{sec_asymp}

Using Theorem~\ref{thm_error},
we can show that asymptotically and for most $j$, 
the approximation error is negligible,
under the following condition:

\begin{defn}\label{defn_param}
We say condition $\mathcal{P}=(\delta,\kappa_1,M_2,\sigma^2,C_{\rm sg})$
is satisfied 
(for some $n$) if:
\begin{itemize}
\item $p(n)=\lfloor n/\delta \rfloor$, where $\delta\in(0,1)$;

\item $A(n)\in\mathbb{R}^{n\times p(n)}$,
where the rows are i.i.d.\ following a distribution $Q(n)$ with zero mean and covariance $\Sigma(n)$.
We have  $\max\{\lambda_{\max}(\Sigma(n)),\lambda_{\min}^{-1}(\Sigma(n))\}\le \kappa_1\in(1,\infty)$, 
and $Q(n)$ is $C_{\rm sg}$-sub-Gaussian, i.e., for $v\sim Q(n)$ and any $t\in\mathbb{R}^p$, we have
\begin{align}
\mathbb{E}[\exp(\left<t,v\right>)]
\le \exp(C_{\rm sg}\|t\|_2^2);
\label{e_sg}
\end{align}

\item The noise $w(n)\sim \mathcal{N}(0,n\sigma^2I_n)$;

\item $Y(n)=A(n)\alpha(n)+w(n)$ and $\frac1{p}\|\alpha(n)\|_2^2 \le M_2$.
\end{itemize}
We will drop the $(n)$ in these notations when there is no confusion.
Recall that $O_{\mathcal{P}}(n)$ indicates that the hidden constant depends only on $\mathcal{P}$ (otherwise, it may depend on other constants such as $\lambda$ or $\kappa_2$ introduced later).
\end{defn}

\begin{rem}
The key properties of $A$ needed are (for some $\kappa>0$
and with high probability),
1) $\|A\|_{\rm op}\le \kappa\sqrt{n}$;
2) $\check{A}_{:j}$ is ${\rm  polylog}(n)$-sub-Gaussian conditioned on $A_{:\sj}$;
3) $\lambda_{\min}(A_{:\mathcal{S}})$ for all $|\mathcal{S}|\le \kappa^{-1}n$.
The proof only uses concentration inequalities to control the order of the error terms,
and small ball probability to lower bound the singular value of a random matrix following \cite{koltchinskii2015bounding}.
We expect that the asymptotic convergence results in this section can be extended to a possibly misspecified setting of $Y=\mu+w$, 
where $\|\mu\|_2^2=O(n^2)$ conditioned on a high $1-o(\frac1{p})$ probability set of $A$, 
and $w\sim \mathcal{N}(0,n\sigma^2I_n)$ is independent of $(A,\mu)$;
this is because the sparsity assumption in Lemma~\ref{lem7}  is ensured for $\|Y\|_2^2=O(n^2)$ with sufficiently large $\lambda$  as in the proof of Corollary~\ref{cor7}.
Moreover, it is expected that the Gaussian noise assumption can be relaxed to more a general small ball probability condition to ensure that not too many subgradients are near the boundary (i.e., close but not equal to $\pm1$).
\end{rem}

For each $j$, we generate $B^{(j)}_{:j}$ by setting $B^{(j)}_{:\sj}=A_{:\sj}$,
and independently sampling $B^{(j)}_{:j}$
according to $P_{A_{:j}|A_{:\sj}}$.
For each $j$, recall $P_{\mathcal{A}}$ and $P_{\mathcal{B}}$ defined around \eqref{eq_alpha_U} and \eqref{eq_beta_U},
and set
\begin{align}
P_j:=P_{\mathcal{A}};
\quad 
P^{(j)}:=P_{\mathcal{B}}.
\label{e_pj}
\end{align}
We then have a general asymptotic approximation result:

\begin{thm}\label{thm12}
Suppose that $\mathcal{P}$ is satisfied,
and $\lambda$ is larger than a threshold depending only on $\mathcal{P}$ but not $n$ (see Corollary~\ref{cor7} for the precise bound on $\lambda$ needed).  
Then, except for a set of $j$ of expected cardinality $\tilde{O}(pn^{-1/18})$,
we have $|t(j,B^{(j)},Y)-t_j|=
\tilde{O}_{\mathcal{P}}(n^{-1/18})$, where
$t(j,B^{(j)},Y)$ is as in 
\eqref{e_tjb}, and 
\begin{align}
t_j
:=
\frac1{n}\check{B}^{(j)\top}_{:j}R
+
\frac1{n}\check{B}^{(j)\top}_{:j}(I-P_j)A_{:j}
\hat{\alpha}_j.
\label{e_131}
\end{align}
In fact, the same asymptotic bound also holds if $\check{B}^{(j)\top}_{:j}(I-P_j)A_{:j}
\hat{\alpha}_j$ is dropped. 
\end{thm}
The proof is given in Section~\ref{sec_sub}.
In Theorem~\ref{thm12}
we bound the approximation error in computing $t(j,B^{(j)},Y)
=\check{B}^{(j)\top}_{:j}
(I-P^{(j)})
B^{(j)}_{:j}\hat{\beta}_j^{(j)U}$
from the update formula.
A natural question is, 
what about the approximation error for $\hat{\beta}_j^{(j)U}$ itself?
That is, by Theorem~\ref{thm12},
do we have 
\begin{align}
\hat{\beta}_j^{(j)U}
\approx
\frac{\frac1{n}\check{B}_{:j}^{\top}R
+
\frac1{n}\check{B}_{:j}^{\top}(I-P_j)A_{:j}
\hat{\alpha}_j}{\frac1{n}\check{B}_{:j}^{\top}(I-P_j)B_{:j}}?
\label{e_approxj}
\end{align}
As we will see in Theorem~\ref{thm_lbdd},
for \eqref{e_approxj} to hold it suffices
make an additional assumption of 
\begin{align}
\mathbb{E}[|\check{B}^{(j)}_{1j}|^2]
\ge \kappa_2^{-1},
\quad
\forall j=1,\dots, p,
\label{e_lb}
\end{align}
for some $\kappa_2>0$ independent of $n$.
\eqref{e_lb} is true, for example, if 
$\mu_{:j}=A_{:\sj}\Sigma_{\sj}^{-1}\Sigma_{\sj j}$ (i.e., the linear predictor is optimal), under the assumption of $\mathcal{P}$.
In the meantime, this is essentially also necessary:
if $\mathbb{E}[|\check{B}^{(j)}_{:j}|^2]=0$,
the right side of \eqref{e_approxj} is undefined.
We have the following answer to \eqref{e_approxj}:

\begin{thm}\label{thm_lbdd}
Suppose that $\mathcal{P}$ and \eqref{e_lb} are satisfied,
and $\lambda$ is larger than a threshold depending only on $\mathcal{P}$ but not $n$ (see Corollary~\ref{cor7} for the precise bound on $\lambda$ needed). 
Then, except for a set of $j$ of expected cardinality $\tilde{O}_{\mathcal{P}}(p\cdot n^{-1/18})$,
the difference between
the left and right sides of \eqref{e_approxj} is at most $\tilde{O}_{\mathcal{P}}(n^{-1/18})$.
\end{thm}
The proof is given in Section~\ref{sec_nonvanish}. 
Theorem~\ref{thm12}
and Theorem~\ref{thm_lbdd} can be used
for variable selection under the false discovery rate (FDR) control,
because they bound the approximation error for all but a small fraction of $j$,
and a vanishing fraction of coordinates does not change the asymptotic FDR and power of the selection algorithm.
To formalize the notion of ``approximation in most coordinates'', recall the notion of Levy-Prokhorov metric which metricizes weak convergence of probability measures (see for example \cite{bobkov2016proximity}):

\begin{defn}
Levy-Prokhorov metric,
denoted as $\pi$, between two probability measures $\mu$ and $\nu$ on a metric space $(X, d)$ is defined as:
\[
\pi(\mu, \nu) = \inf \{\epsilon > 0 : \mu(A) \leq \nu(A^\epsilon) + \epsilon \text{ and } \nu(A) \leq \mu(A^\epsilon) + \epsilon \text{ for all Borel sets } A \subseteq X\}
\]
where $A^\epsilon = \{x \in X : d(x, A) < \epsilon\}$ denotes the $\epsilon$-neighborhood of the set $A$.
\end{defn}

\begin{cor}\label{cor_lp}
In the setting of Theorem~\ref{thm12},
let $\hat{t}:=(t(j,B^{(j)},Y))_{j=1}^p$ and 
$\tilde{t}:=(t_j)_{j=1}^p$.
Then 
$\mathbb{E}[\pi(\hat{P}_{\hat{t}\alpha},\hat{P}_{\tilde{t}\alpha})]
=\tilde{O}_{\mathcal{P}}(n^{-1/18})$.
Similarly,
in the setting of Theorem~\ref{thm_lbdd},
$\mathbb{E}[\pi(\hat{P}_{\hat{\gamma}\alpha},\hat{P}_{\tilde{\gamma}\alpha})]
=\tilde{O}_{\mathcal{P}}(n^{-1/18})$,
where $\hat{\gamma}_j$ and $\tilde{\gamma}_j$ denote the left and the right sides of \eqref{e_approxj}.
\end{cor}

Corollary~\ref{cor_lp} follows directly from
Theorem~\ref{thm12}, Theorem~\ref{thm_lbdd} and the definition of the Levy-Prokhorov metric, so the proof is omitted.
An analysis of the asymptotic impact of the approximation error on variable selection algorithms is given in 
Section~\ref{sec_fdr}.
Finally, we have the following result that the generalized debiased Lasso coefficients reduces to the traditional debiased Lasso coefficients in the case of Gaussian designs,
whose proof is deferred to Section~\ref{sec_gaussian}.

\begin{thm}\label{thm_gaussian}
Consider the setting of Definition~\ref{defn_param}, 
where Condition $\mathcal{P}$ is satisfied with Gaussian row distribution $Q=\mathcal{N}(0,\Sigma)$.
Assume that $\lambda\ge \lambda_{\mathcal{P}}$ is large enough (for some $\lambda_{\mathcal{P}}$ depending on $\mathcal{P}$).
Then with probability at least $1-\tilde{O}_{\mathcal{P}}(p^{-1/18})$,
there is a set of $j$ of cardinality at least $p(1-3p^{-1/18})$,
such that 
\begin{align}
\frac1{n}\check{B}_{:j}^{\top}(I-P_{\mathcal{B}})B_{:j}
=
\Sigma_{j|\sj}(1-\|\hat{\beta}^{(j)}\|_0/n)
+O_{\mathcal{P}}(n^{-1/9}),
\label{e_g156}
\\
\frac1{n}\check{A}_{:j}^{\top}(I-P_{\mathcal{A}})A_{:j}
=
\Sigma_{j|\sj}(1-\|\hat{\alpha}\|_0/n)
+O_{\mathcal{P}}(n^{-1/9}),
\label{e_g156.1}
\end{align}
(where $P_{\mathcal{B}}=P^{(j)}$ and $P_{\mathcal{A}}=P_j$), and consequently we have 
$
|\hat{\beta}^{U(j)}_j
-\hat{\beta}^{u(j)}_j|
=O_{\mathcal{P}}(n^{-1/9})
$
and 
$|\hat{\alpha}^U_j
-\hat{\alpha}^u_j|
=O_{\mathcal{P}}(n^{-1/9})$.
\end{thm}

\begin{rem}
\label{rem10}
In proving Theorem~\ref{thm_gaussian},
the key properties we used include \eqref{e_g156}, \eqref{e_g156.1},
and the fact that the linear estimator is optimal (i.e.\ $\mu_{:j}:=A_{:\setminus j}\Sigma_{\setminus j}^{-1}\Sigma_{\sj j}$).
Thus it is plausible that the result may be extended to a more general class of (possibly non-Gaussian) distributions.
We
conjecture that the result can be generalized to distributions for which property $\mathcal{P}$ holds and the linear estimator is optimal
(in which case \eqref{e_lb} must be true).
\end{rem}

\subsection{Beyond $\ell_1$ regularizer}
In this section, we extend the previous results to the case where the regularization term $\lambda\|\beta\|_1$ is replaced by $\sum\rho(\beta_j)$,
where $\rho$ is a general function satisfying:
\begin{align}
\rho(0)&=\psi(0)=0;
\\
\psi'&:=\rho''\in [\zeta_{2-},\zeta_{2+}];
\\
|\psi''|&\le \zeta_3,
\end{align}
where $\zeta_{2-},\zeta_{2+},\zeta_3>0$ are constants.
The proofs for such $\rho$ is simpler than the Lasso case, due to strong convexity.
Define 
\begin{align}
\bar{G}_{\mathcal{A}}
&:=\frac1{n}A_{:\setminus j}^{\top}A_{:\setminus j}
+\diag(\psi'(\hat{\alpha}_{\sj})),
\\
P_{\mathcal{A}}&:=\frac1{n}A_{:\sj}
\bar{G}_{\mathcal{A}}^{-1}
A_{:\sj}^{\top}.
\end{align}
The following theorems establish the nonasymptotic and asymptotic error bounds analogous to the $\ell_1$ case, whose proofs can be found in the appendix.
\begin{thm}\label{thm7}
Suppose that 
\begin{align}
\|A\|_{\rm op}
\vee
\|B\|_{\rm op}
\le D\sqrt{n},
\label{e40}
\\ 
\|\bar{G}_{\mathcal{A}}^{-1}
A_{:\sj}^{\top}
\check{B}_{:j}
\|_{\infty}
\le \Gamma\sqrt{n}.
\label{e41}
\end{align}
Then 
\begin{align}
\left|t(j,B,Y)-\frac1{n}\check{B}_{:j}^{\top}R
-\frac1{n}\check{B}_{:j}^{\top}(I-P_{\mathcal{A}})A_{:j}\hat{\alpha}_j\right|
\le 
\frac{c_2^{1.5}\zeta_3}{c_1^2\lambda \zeta_{2-}\sqrt{n}}\Gamma D^3
(|\hat{\alpha}_j|+|\hat{\beta}_j|)
\label{e43}
\end{align}
where $c_2:=D^2+\lambda \zeta_{2+}$ and $c_1:=\lambda \zeta_{2-}$.
\end{thm}

\begin{thm}\label{thm8}
Assume condition $\mathcal{P}$,
and let $C>0$ be arbitrary.
Define 
\begin{align}
\mathcal{G}:=\{\textrm{\eqref{e40} and \eqref{e41} for all $j$}\}.
\end{align}
Let ${\sf R}_j$ be the right side of \eqref{e43}.
Then we have $\mathbb{P}[\mathcal{G}^c]=O_{\mathcal{P},\zeta,C}(n^{-C})$ and 
$\mathbb{E}\left[\sum_{j=1}^n {\sf R}_j^2 1_{\mathcal{G}}\right]=O_{\mathcal{P},\zeta,C}(\ln^2 n)$.
\end{thm}

\section{Proof of the approximation formula}\label{sec_aformula}
The goal of this section is to prove Theorem~\ref{thm_error},
Theorem~\ref{thm5},
and 
Theorem~\ref{thm_approx2}.
\subsection{Intuitions}
Recall the optimization problems given in  \eqref{eq_alpha}\eqref{eq_beta}. Intuitively an update formula is possible because the Taylor expansion is asymptotically correct; the challenge though lies in the non-differentiability of the objective function and in showing that error is indeed negligible in the high-dimensional setting.
Before the proof let us first give a heuristic derivation of the approximate update formula.
Let $R:=Y-A\hat{\alpha}$ and $S:=Y-B\hat{\beta}$ denote the residuals.
From the normal equations we have
\begin{align}
-\frac1{n}A_{:\setminus j}^{\top}R
+\lambda\psi^{\alpha}_{\setminus j}&=0_{p-1},
\label{eq2_1}
\\
-\frac1{n}A_{:\setminus j}^{\top}S
+\lambda\psi^{\beta}_{\setminus j}&=0_{p-1},
\label{eq2_2}
\\
-\frac1{n}B_{:j}^{\top}S
+\lambda\psi^{\beta}_j&=0,
\label{eq2_3}
\end{align}
where the subdifferential $\psi^{\alpha}\in[-1,1]^p$
is intuitively the derivative of (the non-differentiable function) $\|\cdot\|_1$ at $\hat{\alpha}$.
Because of the non-differentiability, $\psi^{\alpha}$ is not a function $\hat{\alpha}$. But as a heuristic argument, we pretend that $\psi^{\alpha}=\psi(\hat{\alpha})$, where $\psi\colon \mathbb{R}^p\to\mathbb{R}^p$ applies element-wise the function of the derivative of the absolute value function.
We also pretend that $\psi(\cdot)$ is a smooth function so that we can Taylor expand $\psi(\cdot)$ around $\hat{\alpha}$. 
Now using
\begin{align}
S-R=A_{:\setminus j}(\hat{\alpha}_{\setminus j}-\hat{\beta}_{\setminus j})+A_{:j}\hat{\alpha}_j-B_{:j}\hat{\beta}_j,
\label{eqn3}
\end{align}
and ignoring higher order terms in the Taylor expansion, we obtain
\begin{align}
\bar{G}(\hat{\alpha}_{\setminus j}-\hat{\beta}_{\setminus j})
+u\hat{\alpha}_j
-v\hat{\beta}_j&\approx 0_{p-1},
\label{eq55}
\\
-\frac1{n}B_{:j}^{\top}R
-v^{\top}(\hat{\alpha}_{\setminus j}-\hat{\beta}_{\setminus j})
-\frac1{n}B_{:j}^{\top}A_{:j}\hat{\alpha}_j
+\frac1{n}B_{:j}^{\top}B_{:j}\hat{\beta}_j+\lambda
\psi(\hat{\beta}_j)&\approx 0,
\label{eq25}
\end{align}
where we defined
\begin{align}
\bar{G}&:=\frac1{n}A_{:\setminus j}^{\top}A_{:\setminus j}
+\lambda\diag(\psi'(\hat{\alpha}_{\setminus j}));
\label{e_barg}
\\
u&:=\frac1{n}A_{:\setminus j}^{\top}A_{:j};
\quad
v:=\frac1{n}A_{:\setminus j}^{\top}B_{:j}.
\end{align}
Cancelling $\hat{\alpha}_{\setminus}-\hat{\beta}_{\setminus j}$, we can solve for $\hat{\beta}_j$ to obtain
\begin{align}
(\frac1{n}\|B_{:j}\|^2-v^{\top}\bar{G}^{-1}v)\hat{\beta}_j
+\lambda\psi(\hat{\beta}_j)
\approx(\frac1{n}B_{:j}^{\top}A_{:j}
-v^{\top}\bar{G}^{-1}u)\hat{\alpha}_j
+\frac1{n}B_{:j}^{\top}R.
\label{eq59}
\end{align}
Using the matrix inversion formula, we see that $\bar{G}^{-1}$ is nonzero only in the principal submatrix corresponding to the nonzeros of $\hat{\alpha}_{\setminus j}$.
In the high dimension setting, supposing that the entries of $A$ are i.i.d.\ with unit variance, we have
$
(1-\frac{k}{n})\hat{\beta}_j
+\lambda\psi(\hat{\beta}_j)
\approx \frac1{n}B_{:j}^{\top}R
$,
therefore 
$
(1-\frac{k}{n})\hat{\beta}_j^{u}
\approx \frac1{n}B_{:j}^{\top}R
$,
which recovers the known formula of the debiased estimator \eqref{eq_betau} for i.i.d.\ features.
\begin{rem}
In the case of correlated features,
it is tempting to compute $\hat{\beta}_j$ from \eqref{eq59} using soft thresholding,
which is the idea in proving
Theorem~\ref{thm_approx2}. 
But as we will see the approximation error will not vanish unless the feature correlations are sufficiently small. 
On the other hand, with some additional algebra, we can show that the approximation error for the debiased estimator vanishes even when features have non-vanishing correlations.
\end{rem}

\subsection{Proof of Theorem~\ref{thm_error}}
To deal with the non-differentiability of  $\psi$ in \eqref{e_barg}, define
\begin{align}
\Lambda=\diag\left(\left\{\frac{\psi^{\alpha}_l-\psi^{\beta}_l}{\hat{\alpha}_l-\hat{\beta}_l}\right\}_{l\neq j}\right)
\label{eq_d}
\end{align}
where the $\frac{0}{0}$ case in \eqref{eq_d}
is resolved by setting $\frac{\psi^{\alpha}_l-\psi^{\beta}_l}{\hat{\alpha}_l-\hat{\beta}_l}=0$ if $\psi^{\alpha}_l=\psi^{\beta}_l\notin (-1,1)$, 
and $\frac{\psi^{\alpha}_l-\psi^{\beta}_l}{\hat{\alpha}_l-\hat{\beta}_l}=+\infty$ if $\psi^{\alpha}_l=\psi^{\beta}_l\in (-1,1)$.
Then set 
\begin{align}
G:=\frac1{n}A_{:\setminus j}^{\top}A_{:\setminus j}
+\lambda \Lambda.
\label{e_g}
\end{align}
Note that by replacing $\bar{G}$ with $G$ in \eqref{eq59}, exact equality is achieved, and we have
\begin{align}
(\frac1{n}\|B_{:j}\|^2-v^{\top}G^{-1}v)\hat{\beta}_j
+\lambda\psi^{\beta}_j
=
\frac1{n}B_{:j}^{\top}R
+
(\frac1{n}B_{:j}^{\top}A_{:j}
-v^{\top}G^{-1}u)\hat{\alpha}_j.
\label{eq_69}
\end{align}
Note that 
\begin{align}
\frac1{n}B_{:j}^{\top}R&=\frac1{n}\check{B}_{:j}^{\top}R
+\frac1{n}\mu_{:j}^{\top}S
+\frac1{n}\mu_{:j}^{\top}(R-S),
\label{eq66}
\end{align}
and next we will simplify the last term in \eqref{eq66} using \eqref{eqn3} and 
Proposition~\ref{prop_err}:
\begin{align}
\frac1{n}\mu_{:j}^{\top}(R-S)
&=
-\frac1{n}\mu_{:j}^{\top}
[A_{:\setminus j}(\hat{\alpha}_{\sj}-\hat{\beta}_{\sj})
+A_{:j}\hat{\alpha}_j
-B_{:j}\hat{\beta}_j]
\\
&=-\frac1{n}\mu_{:j}^{\top}A_{:\sj}G^{-1}v\hat{\beta}_j
+
\frac1{n}\mu_{:j}^{\top}A_{:\sj}G^{-1}u\hat{\alpha}_j
-
\frac1{n}\mu_{:j}^{\top}A_{:j}\hat{\alpha}_j
+
\frac1{n}\mu_{:j}^{\top}B_{:j}\hat{\beta}_j.
\end{align}
Note that even if $G$ is not invertible, we can show that $A_{:\sj}G^{-1}A_{:\sj}^{\top}$ is well-defined via limit.
Collecting terms and using \eqref{eq68_1} and $\lambda\psi^{\beta}_j
-\frac1{n}\mu_{j:}^{\top}S=\frac1{n}\check{B}_{:j}^{\top}S$, we see 
\eqref{eq_69} becomes:
\begin{align}
&\quad(\frac1{n}\|B_{:j}\|^2
-\frac1{n}\mu_{:j}^{\top}B_{:j}
-\frac1{n}\check{B}_{:j}^{\top}A_{:\sj}G^{-1}v)\hat{\beta}_j
+\frac1{n}\check{B}_{:j}^{\top}S
\nonumber
\\
&=\frac1{n}\check{B}_{:j}^{\top}R
+
(\frac1{n}\check{B}_{:j}^{\top}A_{:j}
-
\frac1{n}\check{B}_{:j}^{\top}A_{:\sj}G^{-1}u)\hat{\alpha}_j.
\label{eq69}
\end{align}
Next we estimate the coefficients on the two sides of \eqref{eq69}.
Define 
\begin{align}
\bar{G}_{\mathcal{A}}&:=\frac1{n}A_{:\setminus j}^{\top}A_{:\setminus j}
+\lambda\bar{\Lambda},
\label{e_ga}
\end{align}
where $\bar{\Lambda}$ is a diagonal matrix, $
\bar{\Lambda}_{ll}:=0$
if $l\in\mathcal{A}$
and
$\bar{\Lambda}_{ll}:=+\infty$
otherwise.
Define $\bar{G}_{\mathcal{B}}$ analogously but with $\mathcal{A}$ above replaced by $\mathcal{B}:=\{l\neq j\colon \chi^{\beta}_l\neq 0\}$.
Then the third term in the coefficient of $\hat{\beta}_j$ is
\begin{align}
\frac1{n}\check{B}_{:j}^{\top}A_{:\sj}G^{-1}v
&=\frac1{n^2}\check{B}_{:j}^{\top}A_{:\sj}\bar{G}_{\mathcal{B}}^{-1}A_{:\sj}^{\top}B_{:j}
+
\frac1{n^2}\check{B}_{:j}^{\top}A_{:\sj}(G^{-1}-\bar{G}_{\mathcal{B}}^{-1})A_{:\sj}^{\top}B_{:j}
\label{eq70}
\\
&=\frac1{n}\check{B}_{:j}^{\top}P_{\mathcal{B}}B_{:j}
+
\frac1{n^2}\check{B}_{:j}^{\top}A_{:\sj}(G^{-1}-\bar{G}_{\mathcal{B}}^{-1})A_{:\sj}^{\top}B_{:j}
\end{align}
where 
$P_{\mathcal{B}}:=\frac1{n}A_{:\sj}\bar{G}_{\mathcal{B}}^{-1}A_{:\sj}^{\top}$ denotes the projection onto the span of the columns corresponding to the indices $\mathcal{B}$.
Therefore the coefficient of $\hat{\beta}_j$ on the left side of \eqref{eq69} differs from $\frac1{n}\check{B}_{:j}^{\top}(I-P_{\mathcal{B}})B_{:j}$ by at most
\begin{align}
\frac1{n^2}|\check{B}_{:j}^{\top}A_{:\sj}(G^{-1}-\bar{G}_{\mathcal{B}}^{-1})A_{:\sj}^{\top}B_{:j}|
&\le 
\frac1{n^2}\left\|
\check{B}_{:j}^{\top}A_{:\sj}(G^{-1}-\bar{G}_{\mathcal{B}}^{-1})A_{:\sj}^{\top}
\right\|_2
\left\|B_{:j}\right\|_2
\\
&\le 
\frac1{n}
\left\|
\check{B}_{:j}^{\top}(P_{\mathcal{A}\cup\mathcal{J}}-P_{\mathcal{A}\setminus\mathcal{J}})
\right\|_2\left\|B_{:j}\right\|_2
\label{e52}
\\
&\le \frac1{n}\cdot \Gamma \sqrt{|\mathcal{J}|}\cdot D\sqrt{n},
\end{align}
where \eqref{e52} follows from Lemma~\ref{lem5}.
Similarly, for one term in the coefficient for $\hat{\alpha}_j$ in \eqref{eq69},
\begin{align}
\left|\frac1{n}\check{B}_{:j}^{\top}A_{:\sj}G^{-1}u
-\frac1{n}\check{B}_{:j}^{\top}P_{\mathcal{A}}A_{:j}
\right|
&=
\frac1{n^2}\left|\check{B}_{:j}^{\top}A_{:\sj}(G^{-1}-\bar{G}_{\mathcal{A}}^{-1})A_{:\sj}^{\top}A_{:j}
\right|
\\
&\le 
\frac1{n^2}\left\|\check{B}_{:j}^{\top}A_{:\sj}(G^{-1}-\bar{G}_{\mathcal{A}}^{-1})A_{:\sj}^{\top}\right\|_2
\left\|A_{:j}
\right\|_2
\\
&\le \Gamma D\sqrt{\varepsilon}.
\label{eq57}
\end{align}
Then \eqref{eq69} yields
\begin{align}
&\quad\left|\frac1{n}\check{B}_{:j}^{\top}(I-P_{\mathcal{B}})B_{:j}\hat{\beta}_j^U-\frac1{n}\check{B}_{:j}^{\top}R
-
\frac1{n}\check{B}_{:j}^{\top}(I-P_{\mathcal{A}})A_{:j}
\hat{\alpha}_j
\right|
\le \Gamma D\sqrt{\varepsilon}
(|\hat{\alpha}_j|+|\hat{\beta}_j|).
\end{align}
Finally the proof is completed by 
\begin{align}
\frac1{n}|\check{B}_{:j}^{\top}(P_{\mathcal{A}}-P_{\mathcal{B}})B_{:j}|&\le \frac{D}{\sqrt{n}}\|\check{B}_{:j}^{\top}(P_{\mathcal{A}}-P_{\mathcal{B}})\|_2
\\
&\le D\Gamma \sqrt{\varepsilon}.
\end{align}

\subsection{Proof of Theorem~\ref{thm5}}
\label{pfthm5}
Similar to Theorem~\ref{thm_error},
but make 
changes to \eqref{eqn3}, \eqref{eq55}, \eqref{eq25}, \eqref{eq59}, 
and Section~\ref{sec_s21}
by replacing $\hat{\beta}_j$ with $\hat{\beta}_j-\alpha_j$,
and switch the notations $A\leftrightarrow B$,
$\hat{\alpha}\leftrightarrow \hat{\beta}$.

\subsection{Proof of Theorem~\ref{thm_approx2}}
Recall that we showed  \eqref{eq_69}.
Now
\begin{align}
v^{\top}G^{-1}v-\frac1{n}B_{:j}^{\top}P_{\mathcal{A}}B_{:j}
&=
\frac1{n}B_{:j}^{\top}EB_{:j}
\\
&=\frac1{n}\mu_{:j}^{\top}E\mu_{:j}
+\frac{2}{n}\mu_{:j}^{\top}E\check{B}_{:j}
+\frac1{n}\check{B}_{:j}^{\top}E\check{B}_{:j},
\end{align}
where $E:=\frac1{n}A_{:\sj}GA_{:\sj}^{\top}-P_{\mathcal{A}}$,
and recall that $\mu_{:j}=B_{:j}-\check{B}_{:j}$.
As shown in Lemma~\ref{lem5},
\begin{align}
-(P_{\mathcal{A}\cup \mathcal{J}}-P_{\mathcal{A}\setminus \mathcal{J}})
\preceq
E,E^2
\preceq P_{\mathcal{A}\cup \mathcal{J}}-P_{\mathcal{A}\setminus \mathcal{J}}
\end{align}
where we note that $P_{\mathcal{A}\cup \mathcal{J}}-P_{\mathcal{A}\setminus \mathcal{J}}$ is a projection matrix.
Therefore,
\begin{align} 
\left|v^{\top}G^{-1}v-\frac1{n}B_{:j}^{\top}P_{\mathcal{A}}B_{:j}\right|
&\le 
\frac1{n}\|\mu_{:j}\|_2^2
+
\frac{2}{n}\|\mu_{:j}\|_2\|E\check{B}_{:j}\|_2
+
\frac1{n}\check{B}_{:j}^{\top}(P_{\mathcal{A}\cup \mathcal{J}}-P_{\mathcal{A}\setminus \mathcal{J}})\check{B}_{:j},
\label{eq33}
\\
&\le 
\tau^2+2\tau\Gamma\sqrt{\varepsilon}+\varepsilon\Gamma^2.
\label{eq34}
\end{align}
Similarly, we have
\begin{align}
\left|v^{\top}G^{-1}u-\frac1{n}B_{:j}^{\top}P_{\mathcal{A}}A_{:j}\right|
&\le \frac1{n}\|\mu_{:j}\|_2^2
+
\frac1{n}\|\mu_{:j}\|_2
(\|E\check{B}_{:j}\|_2
+\|E\check{A}_{:j}\|_2)
+\frac1{n}\|E\check{B}_{:j}\|_2\|\check{A}_{:j}\|_2
\\
&\le 
\tau^2+\tau(\Gamma\sqrt{\varepsilon}+\check{D})+\Gamma\sqrt{\varepsilon}\check{D}.
\end{align}
Then the claim follows from \eqref{eq_69}.

\section{Application in false discovery rate control}
\label{sec_fdr}
In this section, we provide a theoretical analysis showing that, in certain settings, 
the power of the knockoff filter can be strictly improved by the local knockoff filter and the conditional randomization test. 
Although these methods have been studied empirically in the literature and are typically at least $p$ times slower than knockoffs, 
we show that this $p$-fold computational overhead can be eliminated by using the approximation formula to accelerate the computation.

\subsection{Review of the knockoff  filter and its limitation}
Consider the variable selection problem with observations  $Y\in\mathbb{R}^n$ and $A\in\mathbb{R}^{n\times p}$ 
The false discovery rate (FDR) is defined as
\begin{align}
\fdr:=\mathbb{E}\left[
FDP\right],
\end{align}
where $FDP:=\frac{|H_0\wedge \hat{H}_1|}{|\hat{H}_1|}$, 
$H_0\subseteq \{1,2,\dots,p\}$ is the set of true null variables,
and $\hat{H}_1\subseteq \{1,2,\dots,p\}$ is the set of selected variables.
A good variable selection algorithm is expected to control the FDR below a given budget, while ensuring a large statistical power:
\begin{align}
\pw:=\mathbb{E}\left[
\frac{|H_1\wedge \hat{H}_1|}{|H_1|}
\right].
\end{align}

The present paper focuses on the Model-X knockoff framework \citep{candes2018panning}, which aligns with our emphasis on resampling-based procedures. In this framework, a knockoff matrix $\tilde{A}$ is constructed so that $(A, \tilde{A})$ is exchangeable, while $\tilde{A}$ is conditionally independent of $Y$.
Then we regress $Y$ on the matrix $[A,\tilde{A}]$, 
so that the test statistics for $\tilde{A}$ can be used to estimate the number of false discoveries.
For a full description, see \cite{candes2018panning}.

The exchangeability condition may create challenges in the construction of the knockoff distribution,
and the increase in the number of features (from $p$ to $2p$) can often induce a power loss
\citep{weinstein2017power}\citep{li2022causal}.
In this section, we provide a simple example which always suffers from this suboptimality, regardless of the choice of the knockoff distribution (Theorem~\ref{thm_19}).
For simplicity, consider the $2p<n$ regime, and we simply use the least squares estimator rather than the Lasso.
We recall the following result from \cite{liu2019power}
which gives a necessary and sufficient condition for asymptotic consistency of the knockoff filter in this regime:
\begin{prop}\label{prop17}{[Informal; see Theorem~5 and Proposition~6 in \cite{liu2019power}  for precise statements]}
Let $\pw^{(n)}$ be the power of the knockoff filter with nominal FDR budget $q\in(0,1)$ 
for sample size $n$.
Suppose the standard distributional limit assumption is true.
A necessary and sufficient condition for $\lim_{n\to \infty}\pw^{(n)}$ to converge to 1 is that the empirical distribution of $(\underline{\Theta}_{jj}^{(n)})_{j=1}^{2p}$ converges weakly to a point mass at 0, where $\underline{\Theta}^{(n)}$ is the inverse covariance (precision) matrix of the true and the knockoff variables.
\end{prop}
The setting of \cite{liu2019power}
is to use the debiased Lasso coefficients as test statistics, but when the Lasso regularization $\lambda=0$ we recover the case of least squares statistics.
Moreover, \cite{liu2019power} assumes the existence of the standard distributional limit defined in \cite{javanmard2014hypothesis},
so the empirical distribution of $(\underline{\Theta}_{jj}^{(n)})_{j=1}^{2p}$ has a weak limit,
and by
``$\lim_{n\to \infty}\pw^{(n)}$ to converge to 1''
in Proposition~\ref{prop17} we mean $\lim_{n\to \infty}\pw^{(n)}$ is viewed as a function of this limiting distribution.

The intuition for Proposition~\ref{prop17} is as follows:
Let $\underline{\alpha}\in\mathbb{R}^{2p}$ be the true coefficients and zero paddings for the knockoff coefficients.  
There exists $\tau>0$,
bounded above and below, such that 
\begin{align}
\hat{\underline{\alpha}}\approx \underline{\alpha}+\tau\underline{\Theta}^{1/2}z
\end{align}
where $z\sim \mathcal{N}(0,I_{2p})$.
Therefore if variables are selected based on a threshold test for the coefficients of $\hat{\underline{\alpha}}$, 
then asymptotic consistency is true only if most diagonal entries of $\underline{\Theta}$ vanish (equivalently, the empirical distribution of the diagonal entries must converge to zero).

For Gaussian knockoff filters, it is known
\citep{candes2018panning}
that selecting a knockoff distribution satisfying exchangeability is equivalent to choosing $s$ such that the joint covariance matrix of the true and knockoff variables,
\begin{align}
\underline{\Sigma}:=
\begin{pmatrix}
\Sigma  & \Sigma-\diag(s)
\\
\Sigma-\diag(s)  & \Sigma
\end{pmatrix}
\label{e_sigmau}
\end{align}
is positive semidefinite.
Thus we see that  the Schur complement  $2S-S\Sigma^{-1}S$ plays a key role.
We now establish an auxiliary result that will help building a suboptimality example.
\begin{lem}\label{lem15}
Set $\Sigma^{-1}:=E\in\mathbb{R}^{p\times p}$ as the matrix whose entries are all 1.
Suppose that $s\in(0,\infty)^p$, $S:=\diag(s)$,
and we have $S^{-1}\succeq \frac1{2}\Sigma^{-1}$.
Let $d_1$, \dots, $d_p$ be the diagonals of the positive semidefinite matrix $(2S-S\Sigma^{-1}S)^{-1}$. 
Then $\frac1{p}|\{j\colon d_j>\frac{p}{10}\}|\ge \frac{3}{5}$.
\end{lem}
Proof can be found in Section~\ref{sec_pvar}.
Note that $\Sigma^{-1}$ in the lemma is not invertible, but we do not need its inverse; alternatively, we may perturb it to make it invertible, and then pass the final result to a limit.

\begin{defn}\label{defn_oracle}
Given any statistic $\hat{\alpha}$ (Lasso, debiased Lasso, or OLS), computed using $Y$ and the true feature matrix $A\in\mathbb{R}^{n\times p}$,
the \emph{oracle threshold algorithm}
selects $\{j\colon |\hat{\alpha}_j|>t\}$,
where $t$ is a deterministic number for which the FDR is exactly the budget $q$.
\end{defn}

The oracle threshold algorithm is not realistic since the threshold is not data-driven.
Nevertheless, it serves as a natural benchmark,
and has been considered in, e.g.\ \cite{ke2024power}, under the name \emph{prototype method}.
We then have the following result, whose proof is omitted since it is analogous to Proposition~\ref{prop17}.
\begin{prop}
Assume the $p<n$ regime and $\hat{\alpha}$ is the OLS solution.
Under the standard distributional limit assumption, a necessary and sufficient condition for asymptotic consistency of the oracle threshold algorithm is that the empirical distribution of $(\Theta_{jj}^{(n)})_{j=1}^p$ converges weakly to a point mass at 0, where $\Theta^{(n)}$ is the precision matrix of the true variables.
\end{prop}

By Schur's complement theorem, it is easy to see that 
the diagonals of $\underline{\Theta}$ dominate the corresponding diagonals of $\Theta$, 
so the knockoff algorithm is asymptotically consistent only if the oracle threshold algorithm is. 
Using Lemma~\ref{lem15}, we can construct an example in the following theorem exhibiting a strict separation:
the knockoff filter is not asymptotically consistent, no matter how $s$ is chosen in \eqref{e_sigmau}, 
even though the oracle threshold algorithm remains asymptotically consistent.

\begin{thm}\label{thm_19}
Let $\Sigma^{-1} = a_p E$, where $E$ is as in Lemma~\ref{lem15}, and $a_p$ is a sequence satisfying $\limsup_{p\to\infty} a_p = 0$ and $p a_p \ge 1$. 
Then
\[
\frac{1}{2p}\big|\{j : (\underline{\Sigma}^{-1})_{jj} > 1/10\}\big| \ge \frac{3}{5}.
\]
In particular, while the diagonal entries of $\Sigma^{-1}$ converge to $0$, those of the knockoff precision matrix $\underline{\Sigma}^{-1}$ do not (in the sense of weak convergence of empirical distributions).
\end{thm}
Proof can be found in Appendix~\ref{sec_pvar}.

\subsection{Local knockoff filter}
\label{sec_local}
As mentioned in Section~\ref{sec_intro}, 
a variant of the knockoff method that resamples only a single feature in each Lasso run, referred to as the \emph{local knockoff filter} (or ``one-at-a-time knockoffs'' in the concurrent work \citep{guan2025one}), often achieves higher statistical power. 
In fact, its asymptotic power is approximately that of the oracle threshold algorithm in Definition~\ref{defn_oracle}. 
However, this improvement comes at the cost of increased computational complexity, namely $\Theta(pL)$, where $L$ denotes the cost of solving a single regression problem. 
In this section, we show how the approximation formula in Section~\ref{sec_approx} can be used to reduce this complexity by a factor of $p$.

In the local knockoff method (See Algorithm~\ref{alg1}),
the exchangeability of $(A,\tilde{A})$ in the knockoff filter is weakened to the conditional exchangeability of columns $(A_{:j},\tilde{A}_{:j})|A_{:\sj}$, for each $j=1,\dots,p$.
Then for each $j$, one regression $Y$ on $(A_{:j},A_{:,\sj})$ and on $(\tilde{A}_{:j},A_{:,\sj})$,
thus generating a pair of test statistics which have the same distribution if $j\in H_0$.
To see why the algorithm controls the FDR, note that 
$|\{j\colon |\hat{\alpha}_j^U|>t\}|$ is the number of selected variables, 
whereas
$|\{j\colon |\hat{\gamma}_j|>t\}|\ge |\{j\in H_0\colon |\hat{\gamma}_j|>t\}|$ approximately controls the number of false discoveries.

The two debiased estimators $\hat{\alpha}^u$ and $\hat{\alpha}^U$ in \eqref{eq_betau}
and 
\eqref{eq_alpha_U} have similar asymptotic behaviors under sufficient regularity conditions.
We shall use $\hat{\alpha}^U$ since it arises more naturally in the derivations of the update formula and also appears to induce smaller error in numerical experiments.
To analyze the computational complexity, note that 
\begin{itemize}
\item Although the definition of $\hat{\alpha}^U$ may appear to require $O(p^4)$ computation time since each $P_j$ requires $O(p^3)$ computation time directly from the formula of the projection matrix,
we can actually compute each $P_j$ in $O(p^2)$ time by the rank one update formula and hence  $\hat{\alpha}^U$ in $O(p^3)$ time (see  Section~\ref{sec_compute}).
Thus, the computation time for $\hat{\alpha}^U$ is the same order as obtaining $\hat{\gamma}$ using the approximate update formula.
\item In the preprocessing step, the conditional means $[\mu_{:1},\dots,\mu_{:p}]$ can be calculated in $O(p^3)$ time, if 
it can be approximated using the linear estimator
$-A[\Theta-\diag\diag \Theta]\cdot (\diag\diag \Theta)^{-2}$, 
where $\Theta$ denotes the precision matrix.
Otherwise, we follow the practice in \cite{liu2022fast}
 by regressing each $A_{:j}$ on $A_{:\sj}$.
\end{itemize}
In summary, the computational complexity of Algorithm~\ref{alg1} is $O(p^3 + L)$, 
which becomes $O(p^3)$ when $L = O(p^3)$ (e.g., when solving the Lasso via least angle regression \citep{hastie2009elements}). 
This matches the computational complexity of the knockoff filter.

\begin{algorithm}[ht]
\caption{Local Knockoff Filter (Debiased statistics; see footnote for other statistics)}
\label{alg1}
\begin{flushleft}
\textbf{Input:} 
$A\in\mathbb{R}^{n\times p}$, $Y\in\mathbb{R}^n$,
$\lambda>0$, FDR threshold $q$.
Assume known $\mu_{:1},\dots,\mu_{:p}$.
\end{flushleft}
\begin{algorithmic}
\State Compute the debiased Lasso solution $\hat{\alpha}^U$ for $(A,Y)$ (using Algorithm~\ref{algo2} for fast computation).
\State Initialize $\hat{\gamma}
\in \mathbb{R}^p$.
\For{$1\le j\le p$} 
\State Sample $X$ from the distribution of $A_{:j}|A_{:\sj}$, independently of $A_{:j}$ and $Y$.
\State Construct $B\in\mathbb{R}^{n\times p}$ where $B_{:\sj}=A_{:\sj}$ and $B_{:j}=X$.
\State Compute the debiased Lasso solution $\hat{\beta}^U$ for $(B,Y)$  (possibly approximately using \eqref{e_approx_na}).
\State $\hat{\gamma}_j\gets \hat{\beta}^U_j$.
\EndFor
\State Choose $T$ as the smallest $t$ such that 
$
\frac{|\{j\colon |\hat{\gamma}_j|>t\}|}{|\{j\colon |\hat{\alpha}_j^U|>t\}|}
\le q.
$
\end{algorithmic}
\begin{flushleft}
\textbf{Output:} Selected set of variables is $\{j\colon |\hat{\alpha}_j^U|>T\}$.
\\
* It is possible to use to other statistics. For the Lasso coefficients, we replace $\hat{\alpha}^U_j$ and $\hat{\beta}^U_j$ by $\hat{\alpha}_j$ and $\hat{\beta}_j$. 
We can also replace $\hat{\alpha}^U_j$ and $\hat{\beta}^U_j$ by 
$t(j, A, Y)$ in \eqref{e_tja}  and 
$t_j$  in \eqref{e_131}
if the condition \eqref{e_lb} is not satisfied.
\end{flushleft}
\end{algorithm}

We now show that the local knockoff filter guarantees FDR control under certain asymptotic assumptions.
In order for the proof to proceed smoothly, 
we introduce
\begin{align}
T(\epsilon):=\inf\left\{t\colon
\frac{|\{j\colon |\hat{\gamma}_j|>t-\epsilon\}|}{|\{j\colon |\hat{\alpha}_j^U|>t\}|}
\le q
\right\}.
\label{e_te}
\end{align}
for any given $\epsilon\in(0,1)$.
Note that $T(0)=T$ in
Algorithm~\ref{alg1},
and the numerator in \eqref{e_te} is an overestimate of the number of false discoveries when $\epsilon>0$.
In practical applications though, 
it may not be necessary to use $\epsilon>0$,
if the empirical measure of $(\hat{\gamma}_j)$ is not concentrated around one point (consider for example, the setting where the Gaussian limit property \citep{javanmard2014hypothesis} is true, so that the empirical measure is a Gaussian-smoothed density).

\begin{assump}\label{assump_2}
Consider a sequence of inputs to Algorithm~\ref{alg1} indexed by $n$, where $Y=A\alpha+w$, $\alpha$ is deterministic, and $A$ and $w$ are independent.
Moreover assume that
\begin{enumerate}
\item \label{l_1} $\liminf_{n\to \infty}\frac1{p}\mathbb{E}[|\{j\colon |\hat{\alpha}^U_j|>T(\epsilon)\}|]
>0$.
\item \label{l_2} The Levy-Prokhorov distance
between the empirical measure of $(\hat{\alpha}^U_j,\alpha_j)_{j=1}^p$,
denoted $\hat{P}_{\hat{\alpha}^U\alpha}$,
and its mean,
$P_{\hat{\alpha}^U\alpha}$,
converges to 0 in probability. 
\item \label{l_3} Let $\hat{\gamma}_j$ and $\tilde{\gamma}_j$ be the debiased estimator and its approximation for the $j$-th knockoff, i.e.\ the left and right sides of \eqref{e_approxj}.
The Levy-Prokhorov distance between the  empirical distributions $\hat{P}_{\hat{\gamma}\alpha}$ and  $\hat{P}_{\tilde{\gamma}\alpha}$
converges to 0 in probability.
This implies that the Levy-Prokhorov distance between their expectations, 
$P_{\hat{\gamma}\alpha}$ and  $P_{\tilde{\gamma}\alpha}$,
converges to 0.
\item \label{l_4} $\hat{P}_{\tilde{\gamma}\alpha}$ converges to $P_{\tilde{\gamma}\alpha}$ in probability.
\end{enumerate}
\end{assump}

\begin{rem}
It is possible to prove the convergence of the empirical measures (Assumption~\ref{assump_2} part~\ref{l_2} and \ref{l_4}) under more explicit conditions.
For example, if the distributions of the row of the design matrix and the noise satisfy the Poincare inequality, 
we can control the variance of $\int fd\hat{P}_{\tilde{\gamma}\alpha}$ for any Lipschitz $f$ by a gradient calculation
(see  \cite{bobkov2010concentration}).
We omit the details here since concentration is expected to hold in broader settings (for example, when the distribution of the row vectors have disconnected support, the Poincare inequality fails, but the concentration of the empirical distribution may still be true). 
\end{rem}

\begin{rem}
It is possible to justify Assumption~\ref{assump_2} part~\ref{l_3}
by applying Theorem~\ref{thm_approx2}, Theorem~\ref{thm12},
or 
Theorem~\ref{thm_lbdd},
depending on the distributional assumptions or the choice of the test statistics.
For example, Theorem~\ref{thm12} implies
Corollary~\ref{cor_lp},
which in turn implies  $\mathbb{E}[\pi(\hat{P}_{\hat{\gamma}\alpha},\hat{P}_{\tilde{\gamma}\alpha})]=O(n^{-1/18})$.
\end{rem}

The following consequence of Algorithm~\ref{alg1} is rather direct:
\begin{thm}\label{thm14}
Fix any $\epsilon>0$,
and consider Algorithm~\ref{alg1} with nominal FDR $q$ and with $T$ replaced by $T(\epsilon)$.
Under Assumption~\ref{assump_2},
we have
\begin{align}
\lim_{n\to\infty}\fdr
\le q.
\end{align}
\end{thm}
Proof can be found in Section~\ref{sec_pvar}.

\begin{rem}\label{rem_pw}
If $\epsilon$ is small and if $\hat{P}_{\tilde{\gamma}\alpha}[(T(\epsilon)-\epsilon,\infty)\times \{0\}]$ is close to $\hat{P}_{\tilde{\gamma}}[(T(\epsilon)-\epsilon,\infty)]$ 
(which is the case if most hypotheses are null), 
then the bounds in the proof are also essentially tight,
which indicates that $T(\epsilon)$ is not selected too conservatively and so the algorithm does not lose too much power compared to the oracle threshold algorithm. 
\end{rem}

\subsection{Fast conditional randomization test}
As discussed in Section~\ref{sec_intro},
the complexity of exactly solving CRT is $O(KpL)$,
where $L$ denotes the complexity of solving lasso.
In this section, we adopt the approximate update formula  \eqref{e_approx_na} to compute $\hat{\gamma}^b_j$,
which we call approx-CRT-db;
see Algorithm~\ref{alg_crt}.
Following the same argument as in Theorem~\ref{thm14}, 
approx-CRT-db still guarantees asymptotic FDR control.
Note that computing $Y-A\hat{\alpha}$ has complexity $O(L)$, 
and then computing the product $\check{B}_{:j}^{\top}(Y-A\hat{\alpha})$  takes time $O(n)=O(p)$ in the proportion regime.
This implies that running CRT approximately takes only $O(L+p^2K)$ time.
As discussed in Section~\ref{sec_resample}, approx-CRT-db is asymptotically equivalent to a version of distilled CRT in \cite{liu2022fast}, but exactly implementing the latter takes $O(Lp+p^2K)$ times.

\begin{algorithm}[ht]
\caption{Conditional Randomization Test (CRT)}
\label{alg_crt}
\begin{flushleft}
\textbf{Input:} 
$A\in\mathbb{R}^{n\times p}$, $Y\in\mathbb{R}^n$,
$\lambda>0$, FDR threshold $q$,
number of repetitions $K\in\{1,2,\dots\}$.
Assume known $\mu_{:1},\dots,\mu_{:p}$,
and a conditional sampling oracle.
\end{flushleft}
\begin{algorithmic}
\State Compute the debiased Lasso solution $\hat{\alpha}^U$ for $(A,Y)$.
\For{$1\le j\le p$} 
\For{$1\le b\le K$} 
\State Sample $X$ from the distribution of $A_{:j}|A_{:\sj}$, independently of $A_{:j}$ and $Y$.
\State Construct $B\in\mathbb{R}^{n\times p}$ where $B_{:\sj}=A_{:\sj}$ and $B_{:j}=X$.
\State Compute the debiased Lasso solution $\hat{\beta}^U_j$ for $(B,Y)$,
and set $\hat{\gamma}^b_j\gets \hat{\beta}^U_j$.
\EndFor
\State Set
$
p_j
:=\frac1{K+1}
\left(
1+\sum_{b=1}^K
1\{|\hat{\alpha}^U_j|
\le
|\hat{\gamma}^b_j|
\}
\right).
$
\EndFor
\end{algorithmic}
\begin{flushleft}
\textbf{Output:} Select a set of variables by feeding $(p_j)_{j=1}^p$ to the Benjamini-Hochberg procedure at level $q$.
\end{flushleft}
\end{algorithm}

\section{Conclusion and future work}
\label{sec_conc}
For a given design matrix and observation, 
we derived a formula for updating a debiased Lasso coefficient under perturbation in one column,
and upper bounded the nonasymptotic error (Theorem~\ref{thm_error}).
Then by general concentration and anti-concentration machinery, 
vanishing approximation error in all but a vanishing fraction of  coordinates was established under mild assumptions
(Theorem~\ref{thm12}).
The heart of the argument lies in establishing a stability result for the signs of the Lasso coefficients,
which implies that a heuristic Taylor expansion argument can be used for updating the debiasing Lasso coefficients.
Interestingly, updating the Lasso coefficients (without debiasing) requires much stronger assumptions (Theorem~\ref{thm_approx2}).
As a consequence, several FDR control algorithms based on feature resampling can be implemented much faster with asymptotically zero impact,
if debiased Lasso coefficients are used as the test statistics.

We expect that the asymptotic error control in Theorem~\ref{thm12} can be extended to even more general matrix classes. 
In the variable selection literature, 
a common assumption for the design matrix is the \emph{factor model} \citep{fan2020ipad}.
We expect that Theorem~\ref{thm12} can be extended to the setting of $A=QU$, where $Q\in\mathbb{R}^{n\times r}$ is  a deterministic matrix with bounded singular values, and $U$ is a random matrix from the class of Definition~\ref{defn_param}.
Furthermore, the proof technique for Theorem~\ref{thm12} should be extendable beyond $\ell_1$ penalties.

More broadly, we expect that some of our stability type analysis can be adapted to other related problems.
For example, 
\emph{algorithmic stability} \citep{bousquet2000algorithmic} is defined as the stability of a function of the training data when one data point is removed,
which implies desirable generalization properties and predictive inference guarantees \citep{zrnic2023post,kim2023black}.
Our approximation error bounds for debiased Lasso based on the number of sign changes may be adapted to establish algorithmic stability or differential privacy guarantees.
Furthermore, while the original definition of the debiased coefficients may be non-differentiable in $(A,Y)$, its approximation formula is piecewise differentiable.
This suggests the possibility of establishing concentration or Gaussian limit results for the empirical distribution of the debiased Lasso through the leave-one-out type analysis leveraging the Poincare inequality or Stein's method,
under conditions more general than (or at least not covered by) existing approaches based on vector approximate message  passing \citep{li2023spectrum,venkataramanan2022estimation}
or Lindeberg's universality argument \citep{han2023universality}.
 
Finally, the recent work of \cite{bao2023leave} also showcased the power and generality of the leave-one-out approach,
by applying it in deriving nonasymptotic error bounds for the AMP state evolution,
which differs from the previous proofs using the conditioning technique that require rotational invariance.
While \cite{bao2023leave} focused on the case of independent matrix entries, 
it may be that an extension to the vector sub-Gaussian case is possible by combining techniques in our paper.

\appendix
\section{Errors in the projection matrices}
\label{sec_ep}
Recall $G$, $\bar{G}_{\mathcal{A}}$ and  $\bar{G}_{\mathcal{B}}$ defined in Section~\ref{sec_aformula}.
Their inverses may not be well-defined as matrices since the diagonal values may be $+\infty$.
However, the inverses can be defined as linear operators on the column space of $A_{:\sj}$.
That is, we observe that the following map
\begin{align}
a
\mapsto
A_{:\sj}(A_{:\sj}^{\top}
A_{:\sj}
+\diag(a))^{-1}A_{:\sj}^{\top}
\label{e_gmap}
\end{align}
is well-defined from $[0,+\infty]^{p-1}$ to the set of positive semidefinite matrices,
since the minimum eigenvalue of $A_{:\sj}^{\top}
A_{:\sj}
+\diag(a)$ as an operator on the column space of $A_{:\sj}$ is positive and hence admits an inverse.
Details of this continuity proof can be found in Lemma~\ref{lem16}.
Since our final bounds will only depend on $G^{-1}$
through $A_{:\sj}G^{-1}A_{:\sj}^{\top}$ (and similarly for $\bar{G}_{\mathcal{A}}$ and $\bar{G}_{\mathcal{B}}$), 
we can write $G^{-1}$ with the understanding that it is well-defined restricted to the column space of $A_{:\sj}$.
We now show that $\frac1{n}A_{:\sj}G^{-1}A_{:\sj}^{\top}$ is close to the projection matrices $P_{\mathcal{A}}:=\frac1{n}A_{:\sj}\bar{G}_{\mathcal{A}}^{-1}A_{:\sj}^{\top}$ and $P_{\mathcal{B}}:=\frac1{n}A_{:\sj}\bar{G}_{\mathcal{B}}^{-1}A_{:\sj}^{\top}$.
\begin{lem}\label{lem5}
Let $\mathcal{A}:=\{l\neq j\colon \psi^{\alpha}_l=\pm1\}$, 
$\mathcal{B}:=\{l\neq j\colon \psi^{\beta}_l=\pm1\}$, 
and $\mathcal{J}:=\{l\neq j\colon \chi^{\alpha}_l\neq \chi^{\beta}_l\}$.
We have
\begin{align}
-(P_{\mathcal{A}\cup \mathcal{J}}-P_{\mathcal{A}})
\preceq
\frac1{n}A_{:\sj}
(\bar{G}_{\mathcal{A}}^{-1}-G^{-1})
A_{:\sj}^{\top}
\preceq P_{\mathcal{A}}
-P_{\mathcal{A}\setminus \mathcal{J}}
\label{eq135}
\end{align}
where $\preceq$ denotes the ordering of positive semidefinite matrices.
Consequently, 
 \begin{align}
\left[ \frac1{n}A_{:\sj}
(\bar{G}_{\mathcal{A}}^{-1}-G^{-1})
A_{:\sj}^{\top}\right]^2
\preceq P_{\mathcal{A}\cup\mathcal{J}}
-P_{\mathcal{A}\setminus \mathcal{J}}.
\label{e_gag}
\end{align}
We also have 
\begin{align}
\left[ \frac1{n}A_{:\sj}
(\bar{G}_{\mathcal{B}}^{-1}-G^{-1})
A_{:\sj}^{\top}\right]^2
\preceq P_{\mathcal{A}\cup\mathcal{J}}
-P_{\mathcal{A}\setminus \mathcal{J}}.
\end{align}
\end{lem}
\begin{proof}
Recall that $\bar{\Lambda}\in\mathbb{R}^{(p-1)\times (p-1)}$ is the diagonal matrix where $\bar{\Lambda}_{ll}=0$ if $l\in\mathcal{A}$ and $\bar{\Lambda}_{ll}=+\infty$ if otherwise.
Note that 
$\Lambda$ is a diagonal matrix with nonnegative entries where $\Lambda_{ll}=\bar{\Lambda}_{ll}$ unless $l\in\mathcal{J}$.
Define $\Lambda^+$ as the diagonal matrix with $\Lambda^+_{ll}=0$ iff $l\in \mathcal{A}\setminus\mathcal{J}$ and $+\infty$ otherwise. 
Then from the rule of resolving the $\frac{0}{0}$ case in the definition of $\Lambda$, we see that $\Lambda^+$ dominates $\Lambda$, so
\begin{align}
\bar{G}_{\mathcal{A}}^{-1}-G^{-1}
&=\left(\frac1{n}A_{:\setminus j}^{\top}A_{:\setminus j}
+\lambda \bar{\Lambda}
\right)^{-1}
-
\left(\frac1{n}A_{:\setminus j}^{\top}A_{:\setminus j}
+\lambda \Lambda
\right)^{-1}
\nonumber\\
&\preceq
\left(\frac1{n}A_{:\setminus j}^{\top}A_{:\setminus j}
+\lambda \bar{\Lambda}
\right)^{-1}
-
\left(\frac1{n}A_{:\setminus j}^{\top}A_{:\setminus j}
+\lambda \Lambda^+
\right)^{-1}.
\end{align}
Define $\Lambda^-$ as the diagonal matrix with $\Lambda^-_{ll}=0$ iff $l\in \mathcal{A}\cup\mathcal{J}$ and $+\infty$ otherwise. Then we see that $\Lambda^-$ dominated by $\Lambda$, so
\begin{align}
\bar{G}_{\mathcal{A}}^{-1}-G^{-1}
&\succeq
\left(\frac1{n}A_{:\setminus j}^{\top}A_{:\setminus j}
+\lambda \bar{\Lambda}
\right)^{-1}
-
\left(\frac1{n}A_{:\setminus j}^{\top}A_{:\setminus j}
+\lambda \Lambda^-
\right)^{-1}.
\end{align}
Then \eqref{e_gag} follows from Lemma~\ref{lem_proj}.
By the same argument we also have 
\begin{align}
\left[ \frac1{n}A_{:\sj}
(\bar{G}_{\mathcal{B}}^{-1}-G^{-1})
A_{:\sj}^{\top}\right]^2
\preceq 
P_{\mathcal{B}\cup\mathcal{J}}
-P_{\mathcal{B}\setminus \mathcal{J}}
=P_{\mathcal{A}\cup\mathcal{J}}
-P_{\mathcal{A}\setminus \mathcal{J}},
\end{align}
where the last equality can be verified using the definition of $\mathcal{J}$.
\end{proof}

\begin{prop}
\label{prop_err}
For $\Lambda, u, v$ defined in  Section~\ref{sec_aformula}, we have
\begin{align}
A_{:\sj}(\hat{\alpha}_{\setminus j}-\hat{\beta}_{\setminus j})
=A_{:\sj}\left(\frac1{n}A_{:\setminus j}^{\top}A_{:\setminus j}
+
\lambda \Lambda
\right)^{-1}
\left(v\hat{\beta}_j-u\hat{\alpha}_j
\right)
\label{e_err}
\end{align}
where the inverse is defined 
by a continuity argument for \eqref{e_gmap}
(see Lemma~\ref{lem16}).
\end{prop}
\begin{proof}
Note that from the normal equations in Section~\ref{sec_aformula},
\begin{align}
\frac1{n}A_{:\setminus j}^{\top}A_{:\setminus j}
(\hat{\alpha}_{\setminus j}-\hat{\beta}_{\setminus j})
+\lambda(\psi_{\setminus j}^{\alpha}-\psi_{\setminus j}^{\beta})
=
v\hat{\beta}_j-u\hat{\alpha}_j.
\end{align}
Therefore for any $\delta,\Delta\in\mathbb{R}^{p-1}$,
\begin{align}
\frac1{n}A_{:\setminus j}^{\top}A_{:\setminus j}
(\hat{\alpha}_{\setminus j}-\hat{\beta}_{\setminus j}+\delta)
+\lambda(\psi_{\setminus j}^{\alpha}-\psi_{\setminus j}^{\beta}+\Delta)
=
v\hat{\beta}_j-u\hat{\alpha}_j
+\frac1{n}A_{:\setminus j}^{\top}A_{:\setminus j}\delta+\lambda\Delta.
\end{align}
Then defining $\Lambda_{\delta,\Delta}$ 
as the diagonal matrix whose diagonal values are $\frac{\psi^{\alpha}_l-\psi^{\beta}_l+\Delta_l}{\hat{\alpha}_l-\hat{\beta}_l+\delta_l}$, 
$l\neq j$,
we have
\begin{align}
\hat{\alpha}_{\setminus j}-\hat{\beta}_{\setminus j}+\delta
=\left(\frac1{n}A_{:\setminus j}^{\top}A_{:\setminus j}
+
\lambda \Lambda_{\delta,\Delta}
\right)^{-1}
\left(v\hat{\beta}_j-u\hat{\alpha}_j
+\frac1{n}A_{:\setminus j}^{\top}A_{:\setminus j}\delta+\lambda\Delta\right)
\end{align}
whenever $\delta,\Delta$ are such that the above matrix inverse is defined.
Now we can take a particular vanishing sequence of  $(\delta,\Delta)$ such that the coordinates of $\Lambda_{\delta,\Delta}$ converges to coordinates of $\Lambda$ on the extended real line.
Assuming without loss of generality that $A_{:\sj}$ has nonzero columns, 
we can also assume that  $\Delta\in\col(A_{:\sj}^{\top})$.  
We can show that the map \eqref{e_gmap} is continuous (see Lemma~\ref{lem16}), 
so \eqref{e_err} follows.
In fact, from this proof we can see that \eqref{e_err} remains valid when $\frac{0}{0}$ in the definition of $\Lambda$ is set to \emph{any} value in $[0,\infty]$.

\end{proof}


\begin{lem}\label{lem_proj}
Suppose that $P$ is the projection matrix onto a subspace in $\mathbb{R}^n$. 
If $A\in\mathbb{R}^n$ is a symmetric matrix satisfying $-P\preceq A \preceq P$, then $A^2\preceq P$.
\end{lem}
\begin{proof}
Without loss of generality, assume that $P=
\begin{pmatrix}
I_{r\times r} &0\\
0 & 0
\end{pmatrix}$, for some $r\le n$. 
If $A=\begin{pmatrix}
A_{11} &A_{12}\\
A_{21} & A_{22}
\end{pmatrix}$, where $A_{11}\in\mathbb{R}^{r\times r}$,
we claim that $A_{12}$, $A_{21}$, $A_{22}$ are zero.
Indeed, let $u\in\mathbb{R}^r$ and $v\in\mathbb{R}^{n-r}$ be arbitrary row vectors. We have 
\begin{align}
(u,v)A(u,v)^{\top}
=uA_{11}u^{\top}
+2uA_{12}v^{\top}
+vA_{22}v^{\top}
\end{align}
which, by our assumption, is bounded between $\|u\|_2^2$ and $-\|u\|_2^2$ for all $u$ and $v$, so it follows that $A_{12}$, $A_{21}$, $A_{22}$ are zero. Moreover $A_{11}$ must have all eigenvalues between $-1$ and 1, so $A^2\preceq P$ follows.
\end{proof}

\section{Sign stability}
\subsection{Control of change of residual}
\label{sec_s21}
The residual change $R-S$ is easier to control than $\hat{\alpha}-\hat{\beta}$.
From the normal equations in Section~\ref{sec_aformula}, we have
\begin{align}
\frac1{n}A_{:\sj}^{\top}(S-R)+\lambda \Lambda(\hat{\alpha}_{\sj}-\hat{\beta}_{\sj})=0.
\label{e110}
\end{align}
If $0\cdot\infty$ is encountered in \eqref{e110},
we can add perturbations as in \eqref{prop_err},
and pass to a limit in the end, so the end result still holds.
Solving $\hat{\alpha}_{\sj}-\hat{\beta}_{\sj}$ and plugging it in the normal equations, we have
\begin{align}
S-R=-\frac1{\lambda n}A_{:\sj}\Lambda^{-1}A_{:\sj}^{\top}(S-R)
+A_{:j}\hat{\alpha}_j-B_{:j}\hat{\beta}_j.
\label{eq85}
\end{align}
Therefore
\begin{align}
S-R=\left[I+\frac1{\lambda n}A_{:\sj}\Lambda^{-1}A_{:\sj}^{\top}\right]^{-1}(A_{:j}\hat{\alpha}_j-B_{:j}\hat{\beta}_j).
\label{eq86}
\end{align}
Since the eigenvalues of $\left[I+\frac1{\lambda n}A_{:\sj}\Lambda^{-1}A_{:\sj}^{\top}\right]^{-1}$ are upper bounded by 1, we have 
\begin{align}
\|S-R\|_2\le \left\|A_{:j}\hat{\alpha}_j-B_{:j}\hat{\beta}_j\right\|_2.
\label{eq82}
\end{align}

\subsection{Control of $\psi^{\alpha}$ near the edge}
In this section we prove sufficient conditions under which $\psi_m(Y)\in (-1,-1+\epsilon]\cup[1-\epsilon,1)$ with small probability whenever $\epsilon>0$ is small.
In this section,
we assume that $A$ is an arbitrary fixed matrix,
$Y$ is a Gaussian random vector with a certain mean, 
and define $\psi()$ and $\chi()$ as the subgradient and the signs of the Lasso solution as functions of $Y$, using the definitions in Section~\ref{sec_intro}.

\begin{lem}\label{lem7}
Let $\epsilon,\delta_0\in(0,1)$, $k,m\in\{1,2,\dots,p\}$,
and $\kappa,\Gamma\in(0,\infty)$.
Let $A$ be deterministic, and $Y\sim \mathcal{N}(\mu,n\sigma^2I_n)$ for some $\mu\in\mathbb{R}^n$.
Let $\Omega_1$ be the set of $y$ such that $|\{l\colon \chi_l(y)\neq 0\}|\le k$.
Let $\Omega_2$ be the set of $y$ satisfying
$
\frac1{n}\left<y-\mu,A_{:m}\right>\ge -\Gamma.
$
Now suppose that 
\begin{align}
\frac1{n}\inf_PA_{:m}^{\top}(I-P)A_{:m}\ge \kappa
\label{e_cond1}
\end{align}
where the infimum is over projection $P$ onto the span of $k$ columns of $A$ excluding the $m$-th column;
and that
\begin{align}
\mathbb{P}\left[Y+ U\frac{2\epsilon\lambda}{\kappa}A_{:m}\notin\Omega_1\right]
\vee 
\mathbb{P}\left[Y- U\frac{2\epsilon\lambda}{\kappa}A_{:m}\notin\Omega_1\right]
\le \delta_0,
\label{e_cond2}
\end{align}
where $U$ is a uniform random variable on $[0,1]$ independent of $Y$;
and that
\begin{align}
\frac1{n}\|A_{:m}\|^2\le 2(\ln 3-1)\Gamma^2\sigma^2;
\label{e_cond3}
\\
P_Y[\Omega_2^c]\le \delta_0.
\label{e_cond4}
\end{align}
Then
\begin{align}
\mathbb{P}[1-|\psi_m(Y)|\in(0,\epsilon]]\le 
\frac{12\epsilon\lambda}{\kappa}\Gamma 
+
6\delta_0.
\label{eq87}
\end{align}
\end{lem}
\begin{rem}
In the proportional growth regime,
\eqref{e_cond2} is not much different than $\mathbb{P}\left[Y\notin\Omega_1\right]\le \delta_0$, since $\frac{2\epsilon\lambda}{\kappa}A_{:m}$ is small compared to $Y$.
If $k$ grows linearly with $n$ and $k/n<1$,
then we can bound $\kappa$ on the order of a constant using the restricted singular value of $A$.
We can take $\Gamma$ to be slowly growing with $n$, 
say ${\rm polylog}(n)$.
Also suppose that $\|\mu\|_2=O(n)$ and then take $\lambda$ sufficiently large. Then $\delta_0$ vanishes in $n$ quickly.

\end{rem}

\begin{proof}
In this proof we assume without loss of generality that $m=1$.
For any $y$, set $\mathcal{Z}(y):= \{l=1\dots,p\colon \psi_l(y)\in(-1,1)\}$.
We may omit the argument $y$ when there is no ambiguity.
The normal equation can be written as
\begin{align}
-\frac1{n}A_{:\mathcal{Z}^c}^{\top}(y-A\hat{\alpha})
+\lambda\psi_{\mathcal{Z}^c}&=0;
\label{eq90}
\\
-\frac1{n}A_{:\mathcal{Z}}^{\top}(y-A\hat{\alpha})
+\lambda\psi_{\mathcal{Z}}&=0,
\label{eq91}
\end{align}
from which we obtain
\begin{align}
\hat{\alpha}_{\mathcal{Z}^c}&=(A_{:\mathcal{Z}^c}^{\top}A_{:\mathcal{Z}^c})^{-1}(A_{:\mathcal{Z}^c}^{\top}y-n\lambda\psi_{\mathcal{Z}^c});
\label{eq92}
\\
\lambda \psi_{\mathcal{Z}}
&=\frac1{n}A_{:\mathcal{Z}}^{\top}(I-A_{:\mathcal{Z}^c}(A_{:\mathcal{Z}^c}^{\top}A_{:\mathcal{Z}^c})^{-1}A_{:\mathcal{Z}^c}^{\top})y
+\lambda A_{:\mathcal{Z}}^{\top}A_{:\mathcal{Z}^c}(A_{:\mathcal{Z}^c}^{\top}A_{:\mathcal{Z}^c})^{-1}\psi_{\mathcal{Z}^c}.
\label{eq93}
\end{align}
In the proof we only need to study $\psi$, which is uniquely determined by $y$ (even though $\hat{\alpha}$ may not).
We observe that $\mathbb{R}^n$ is partitioned into
polyhedra (intersection of open or closed half spaces)
according to the value of $\chi(\cdot)$, and within each such polyhedron,
$\psi$ is a linear function, hence differentiable.
Define 
\begin{align}
\mathcal{B}_0:=\{y\colon \psi_1(y)\in[1-\epsilon,1)
\},
\label{e_sball}
\end{align}
and let $y_0\in \mathcal{B}_0$ be arbitrary, and let $y_t:=y_0-tA_{:1}$.
If for a neighborhood of $t$, the vector $\chi(y_t)$ does not change, then note that $\psi_{\mathcal{Z}^c}(y_t)$ is a vector consisting of $\pm 1$ which does not change in such a neighborhood. 
As long as $1\in\mathcal{Z}(y_t)$, we have
\begin{align}
\frac{d}{dt}(\lambda\psi_1(y_t))
=
-\frac1{n}A_{:1}^{\top}(I-A_{:\mathcal{Z}^c}(A_{:\mathcal{Z}^c}^{\top}A_{:\mathcal{Z}^c})^{-1}A_{:\mathcal{Z}^c}^{\top})
A_{:1}\le  0,
\label{es27}
\end{align}
so that $\psi_1(y_t)$  decreases monotonically along $t$.
Moreover the inequality in \eqref{es27} is strict when $y_t\in\Omega_1$, due to \eqref{e_cond1}.
Next, we want to show that $\psi_1$ decreases sufficiently fast along $(y_t)_{t\in [0,T)}$.
Set $T:=\frac{2\epsilon\lambda}{\kappa}$ and define
\begin{align}
\mathcal{G}:=
\left\{
y\colon \mathbb{P}[y-TUA_{:1}\in\Omega_1]>\frac1{2}
\right\}.
\end{align}
Then we have 
\begin{align}
\delta_0
&\ge 
\mathbb{P}[Y-TUA_{:1}\notin\Omega_1]
\\
&=\int\mathbb{P}[y-TUA_{:1}\notin \Omega_1] dP_Y(y)
\\
&\ge \frac1{2}P_Y[\mathcal{G}^c].
\end{align}
Now suppose that $y_0\in \mathcal{G}\cap\mathcal{B}_0$;
by the definitions of $\mathcal{G}$ and $\Omega_1$,
for at least half (with respect to the Lebesgue measure) of $t\in[0,T]$, 
we have $|\{l\colon \chi_l(y_t)\neq 0\}|\le k$ so that
$\frac{d}{dt}\psi_1(y_t)\le -\frac{\kappa}{\lambda}$. 
Therefore $\psi_1(y_T)\le \psi_1(y_0)-\frac{T\kappa}{\lambda}< 1-\epsilon$.
Since $\psi_1(y_0)\ge 1-\epsilon$ and since $\psi_1(y_t)$ is continuous in $t$, there is some $t\in[0,T)$ such that $y_t=1-\epsilon$.
By considering the largest $t$ such that $y_t=1-\epsilon$, we see that $\mathcal{B}_0\cap\mathcal{G}$ is contained in $\mathcal{B}_0'$, where 
\begin{align}
\mathcal{B}_0':=\{y\colon\psi_1(y)= 1-\epsilon,\,
\psi_1(y-tA_{:1})<1-\epsilon,\,\forall t>0\}+\{tA_{:1}\colon t\in [0,T)\}.
\end{align}
To conclude the proof
it suffices to upper bound the probability of $\mathcal{B}_0'$.
Define the sets 
\begin{align}
\mathcal{B}_l'':=\mathcal{B}_0'\cap\Omega_2-TlA_{:1},
\end{align}
$l=0,1,2,\dots, L$, where $L:=\lfloor\frac1{T\Gamma }\rfloor$.
Then $\{\mathcal{B}_l''\}_{l=0}^L$ is a non-intersecting family of sets, because for any $y\in\{y\colon\psi_1(y)= 1-\epsilon,\,
\psi_1(y-tA_{:1})<1-\epsilon,\,\forall t>0\}$, 
it is not possible that $y-tA_{:1}\in \{y\colon\psi_1(y)= 1-\epsilon\}$ for some $t>0$.
Let 
\begin{align}
D:=2(\ln 3-1)\Gamma^2.
\end{align}
Then for any $y\in\Omega_2$ and $l\in \{0,1,2,\dots,L\}$, we have
\begin{align}
\frac{f(y-lTA_{:1})}{f(y)}
&=\exp\left(-\frac{\|y-\mu-lTA_{:1}\|^2}{2n\sigma^2}
+\frac{\|y-\mu\|^2}{2n\sigma^2}
\right)
\\
&=\exp\left(
-\frac{l^2T^2\|A_{:1}\|^2}{2n\sigma^2}
+\frac{\left<y-\mu,lTA_{:1}\right>}{n\sigma^2}
\right)
\\
&\ge \exp(-\frac{D}{2\Gamma ^2}-1),
\end{align}
so that $P_Y[\mathcal{B}_l'']\ge \exp(-\frac{D}{2\Gamma ^2}-1)
P_Y[\mathcal{B}_0'\cap\Omega_2]$ for $l=0,1,\dots,L$,
and hence
\begin{align}
1\ge \frac1{T\Gamma }\exp(-\frac{D}{2\Gamma ^2}-1)
P_Y[\mathcal{B}_0'\cap\Omega_2].
\end{align}
Then it follows that
\begin{align}
P_Y[\mathcal{B}_0]
&\le P_Y[\mathcal{B}_0\cap\mathcal{G}]+P_Y[\mathcal{G}^c]
\\
&\le P_Y[\mathcal{B}_0']+P_Y[\mathcal{G}^c]
\\
&\le P_Y[\mathcal{B}_0'\cap\Omega_2]+P_Y[\Omega_2^c]+P_Y[\mathcal{G}^c],
\end{align}
or equivalently,
\begin{align}
\mathbb{P}[1-\psi_1(Y)\le \epsilon]\le 
T\Gamma \exp\left(1+\frac{D}{2\Gamma ^2}\right)
+2\delta_0.
\end{align}
Finally, applying a similar argument after the substitution $A_{:1}\leftarrow -A_{:1}$, we obtain
\eqref{eq87} via the union bound.
\end{proof}

\subsection{Control of change of signs}
\begin{cor}\label{cor_change}
Suppose that deterministic $A$ and $B$ differ only in the $j$-th column.
For a given $j\in\{1,\dots,p\}$, let $\hat{\alpha}$ and $\hat{\beta}$ be computed as in Section~\ref{sec_intro}.
Suppose that the assumption in Lemma~\ref{lem7} holds for all $m=1,\dots, p$ and for both $A$ and $B$ and some $\epsilon\in(0,1/2)$.
Then
\begin{align}
\mathbb{E}|\{m\colon \chi^{\alpha}_m\neq \chi^{\beta}_m\}|
&\le 
4+2p\left(\frac{12\epsilon\lambda}{\kappa}\Gamma 
+
6\delta_0\right)
\nonumber\\
&\quad+
\frac{4}{n^2\lambda^2\epsilon^2}\|A\|_{\rm op}^2\mathbb{E}\left\|A_{:j}\hat{\alpha}_j-B_{:j}\hat{\beta}_j\right\|_2^2.
\label{eq113}
\end{align}
\end{cor}

\begin{proof} The cardinality in question is decomposed into 
$
|\{m\colon\chi^{\alpha}_m=0,\chi^{\beta}_m\neq 0\}|
+
|\{m\colon\chi^{\alpha}_m\neq 0,\chi^{\beta}_m= 0\}|
+|\{m\colon\chi^{\alpha}_m> 0,\chi^{\beta}_m< 0\}|
+|\{m\colon\chi^{\alpha}_m< 0,\chi^{\beta}_m> 0\}|
$ and we proceed by bounding the 4 terms separately.
First, we have
\begin{align}
&\quad
|\{m\colon\chi^{\alpha}_m=0,\chi^{\beta}_m\neq 0\}|
\nonumber\\
&\le |\{m\colon\chi^{\alpha}_m=0,1-|\psi^{\alpha}_m|\le \epsilon\}|
+ |\{m\colon1-|\psi^{\alpha}_m|> \epsilon, \chi^{\beta}_m\neq0\}|
\label{eq107}
\end{align}
where $m\in\{1,\dots,p\}$.
The expectation of the first term on the right side of \eqref{eq107} is bounded by $p\left(\frac{12\epsilon\lambda}{\kappa}\Gamma 
+
6\delta_0\right)$ according to Lemma~\ref{lem7};
to bound the second term, 
note that $1-|\psi^{\alpha}_m|> \epsilon, \chi^{\beta}_m\neq0$ implies $|\psi^{\alpha}_m-\psi^{\beta}_m|\ge \epsilon$,
so 
\begin{align}
|\{m\colon1-|\psi^{\alpha}_m|> \epsilon, \chi^{\beta}_m\neq0\}|
\le 
\frac1{\epsilon^2}\|\psi^{\alpha}_{\sj}
-\psi^{\beta}_{\sj}
\|_2^2+1.
\end{align}
But from the normal equations and \eqref{eq82}, we have
\begin{align}
\|\lambda(\psi^{\alpha}_{\setminus j}
-\psi^{\beta}_{\setminus j})
\|_2
&=\frac1{n}\|A_{:\setminus j}^{\top}(R-S)\|_2
\\
&\le \frac1{n}\|A\|_{\rm op}\left\|A_{:j}\hat{\alpha}_j-B_{:j}\hat{\beta}_j\right\|_2.
\label{eq112}
\end{align}
Therefore,
\begin{align}
\mathbb{E}|\{m\colon\chi^{\alpha}_m=0,\chi^{\beta}_m\neq 0\}|
&\le 
1+p\left(\frac{12\epsilon\lambda}{\kappa}\Gamma 
+
6\delta_0\right)
\nonumber\\
&\quad+
\frac1{n^2\lambda^2\epsilon^2}\|A\|_{\rm op}^2\mathbb{E}\left\|A_{:j}\hat{\alpha}_j-B_{:j}\hat{\beta}_j\right\|_2^2.
\end{align}
We bound $|\{m\colon\chi^{\alpha}_m\neq 0,\chi^{\beta}_m= 0\}|$ by the same argument.
Moreover, since $\chi^{\alpha}_m>0$ and $\chi^{\beta}_m<0$ imply $|\psi^{\alpha}_m-\psi^{\beta}_m|\ge 2$,
we have
\begin{align}
|\{m\colon\chi^{\alpha}_m>0,\chi^{\beta}_m<0\}|
\le \frac1{4}\|\psi^{\alpha}_{\sj}-\psi^{\beta}_{\sj}\|_2^2+1.
\end{align}
The same bound holds for $|\{m\colon\chi^{\alpha}_m<0,\chi^{\beta}_m>0\}|$, and the proof is finished by using \eqref{eq112}.
\end{proof}

\section{The case of sub-Gaussian designs}\label{sec_sub}
In this section we prove Theorem~\ref{thm12}.

\subsection{Anticoncentration}
In this section we prove auxiliary results that will be used in justifying the conditions in Corollary~\ref{cor_change}.
The following Lemma is a consequence of 
\cite[Theorem 3.1]{koltchinskii2015bounding},
which lower bounds the minimum singular values of a random matrix with i.i.d.\ rows,
under the assumption that the projection of a row in each direction has bounded (both from the above and below) $L_1$ and $L_2$ norms.
\begin{lem}\label{lem6}
Assuming that  condition $\mathcal{P}$ is satisfied.
There exist $c>0$ (depending only on $\mathcal{P}$) such that the following holds: For all $n\ge 1/c$,
with probability at least $1-\exp(-cn)$ we have 
\begin{align}
\min_{\mathcal{S}\subseteq \{1,\dots,p\}\colon
|\mathcal{S}|\le cn}
\lambda_{\min}(A_{:\mathcal{S}})>c.
\label{e92}
\end{align}
In particular, \eqref{e92} implies that 
for each $j$,
\begin{align}
\frac1{n}\inf_PA_{:j}^{\top}
(I-P)A_{:j}\ge c^2,
\label{e93}
\end{align}
where the infimum is over projection $P$ onto the span of no more than $cn-1$ columns of $A$ excluding the $j$-th column.
\end{lem}

\begin{proof}
Consider arbitrary $\mathcal{S}\subseteq \{1,\dots,p\}$
of cardinality no more than $cn$.
For any unit vector $t$, let $v_t:=\left<t,A_{1\mathcal{S}}\right>$, and we have
\begin{itemize}
\item $\kappa_1^{-1/2}\le \sqrt{\mathbb{E}[v_t^2]}\le \kappa_1^{1/2}$; 
\item Using the Taylor expansion it is easy to see that $e^{-x}+e^x\ge |x|^3/4$ for all $x\in\mathbb{R}$, therefore 
$
\mathbb{E}[|v_t|^3]\le 
8\exp(C_{\rm sg})
$
and hence
\begin{align}
\sqrt{\mathbb{E}[v_t^2]}&\le 
\kappa_1^{3/2}
\mathbb{E}^2[v_t^2]
\\
&\le \kappa_1^{3/2}\mathbb{E}[|v_t|] \mathbb{E}[|v_t|^3]
\\
&\le 8\kappa_1^{3/2}\exp(C_{\rm sg})
\mathbb{E}[|v_t|].
\end{align}
\end{itemize}
Now define $b:=8\kappa_1^{3/2}\exp(C_{\rm sg})$.
The above two itemized verify the condition of \cite[Theorem 3.1]{koltchinskii2015bounding}, 
therefore, there exist universal constants $c_0,c_1,c_2$ such that
when 
$
n\ge c_0b^4\kappa_1^2(cn+1)$
we have the following bound on the singular value
\begin{align}
\lambda_{\min}(\frac1{\sqrt{n}}A_{:\mathcal{S}})
\ge c_2\kappa_1^{-1/2}
b^{-2}
\label{e96}
\end{align}
with probability at least 
$
1-\exp(-c_1b^4n)
$.
Now \eqref{e92} follows by taking the union bound, noting that the number of subsets of no more than $cn$ columns is bounded by $\exp(2nc\ln\frac1{c}+o(n))$ for small enough $c>0$.

Now consider any $j$, and let $P$ be the projection onto the column space of $A_{:\mathcal{S}\setminus \{j\}}$,
where $\mathcal{S}$ is an arbitrary subset containing $j$ and of size at most $cn$.
Let $\theta\in\mathbb{R}^{|\mathcal{S}|-1}$ be a vector such that $(I-P)A_{:j}=A_{:j}-A_{:\mathcal{S}\setminus\{j\}}\theta$.
Then \eqref{e92} implies that
\begin{align}
\|(I-P)A_{:j}\|_2
&\ge \|(1,\theta)\|_2\lambda_{\rm min}(A_{:\mathcal{S}})
\\
&\ge c
\end{align}
which establishes the second claim.
\end{proof}

\begin{prop}\label{prop_kbound}
For the Lasso problem with data $(A,Y)$, we have 
\begin{align}
K\le \frac{\|A\|_{\rm op}^2}{n^2\lambda^2}\|Y\|_2^2,
\end{align}
where $K:=\|\chi\|_0$ and $\chi$ denotes the signs of the Lasso solution as defined in the introduction.
\end{prop}
\begin{proof}
From the normal equation $\lambda\psi
=\frac1{n}A^{\top}R$,
we have 
\begin{align}
\frac1{n^2}\|A\|_{\rm op}^2\|R\|^2
\ge \lambda^2\|\psi\|_2^2
\ge
\lambda^2K
\end{align}
The claim follows since by the optimality condition we have $\|R\|_2^2\le \|Y\|_2^2$.
\end{proof}

\subsection{Expected sign changes}
Recall the sign vector $\chi^{\alpha}$ defined in Section~\ref{sec_intro},
and define $\chi^{\beta^{(j)}}$ analogously. 
We show that when the Lasso regularization parameter $\lambda$ is sufficiently large (but independent of $n$),
the expected number of sign changes is $o(n)$.
\begin{cor}\label{cor7}
Consider any set of parameters $\mathcal{P}$.
There exist $c=c_{\mathcal{P}}>0$, such that for all
\begin{align}
\lambda^2\ge 
\frac{2}{0.9c^3}(\frac{4M_2}{c^2\delta}+\sigma^2(1+2\sqrt{c}+2c)),
\label{e_lam}
\end{align}
there exists a set $\mathcal{E}$ of $A,B^{(1)},\dots,B^{(p)}$ 
satisfying 
\begin{align}
\frac1{p}\sum_{j=1}^p\mathbb{E}[|\{m\colon \chi_m^{\alpha}
\neq 
\chi_m^{\beta^{(j)}}\}|
1_{\mathcal{E}^c}]
=
\tilde{O}_{\mathcal{P}}(\inf_{\epsilon\in(0,1/2)}\{\lambda n\epsilon+\frac1{\lambda^2\epsilon^2}\})
=
\tilde{O}_{\mathcal{P},\lambda}(n^{2/3})
\end{align}
and 
$\mathbb{P}[\mathcal{E}]\le (1+p)\exp(-cn)$
for all $n>n_{\mathcal{P},\lambda}$.
\end{cor}
\begin{proof}
The proof follows from Corollary~\ref{cor_change} and Lemma~\ref{lem7}.
We choose  $c$ (depending on $\mathcal{P}$) small enough,
such that for all sufficiently (depending on $\mathcal{P}$) large $n$,
there exists a set $\mathcal{E}_0$ of $A$ such that 
\begin{align}
\mathbb{P}[A\in \mathcal{E}_0]
\le e^{-cn},
\label{e184}
\end{align}
and the following hold for all $A\in\mathcal{E}_0^c$:
\begin{itemize}
\item 
\eqref{e92} holds.

\item $\|A\|_{\rm op}\le \sqrt{n}/c$, which is a standard matrix concentration result (see Lipschitz maximal inequality in  
\cite[Example 5.10]{van2014probability}; the argument extends to sub-Gaussian row vectors).

\item For each $m\in\{1,\dots,p\}$, $\mathbb{P}[\frac1{n}\left<w,A_{:m}\right>\ge  -\ln(n)|A]\ge 1-e^{-cn}$.
This is indeed possible, 
because $w\sim \mathcal{N}(0,n\sigma^2 I)$, and  $\mathbb{P}[\frac1{n}\left<w, A_{:m}\right> \ge -\ln n| \|A\|_{\rm op}\le \sqrt{n}/c]\ge 1-e^{-cn}$ for small $c$ and large $n$.
\end{itemize}
Since each $B^{(j)}$ has the same distribution as $A$, the above items remain true if $A$ is replaced by $B^{(j)}$ in the statements, 
and we define $\mathcal{E}_j$ as the corresponding set of $B^{(j)}$.
We see that for $A\in\mathcal{E}_0^c$,
the conditions \eqref{e_cond1}-\eqref{e_cond4} of Lemma~\ref{lem7} hold for regressing $Y=A\alpha+w$ on $A$ with 
\begin{align}
\Gamma&\leftarrow \frac{\ln n}{c};
\\
\delta_0&\leftarrow 2e^{-cn};
\\
\kappa&\leftarrow c^2;
\\
k&\leftarrow cn-1.
\end{align}
Indeed, \eqref{e_cond1},\eqref{e_cond3} and \eqref{e_cond4} follow directly from the itemized above. 
To verify \eqref{e_cond2},
note that from Proposition~\ref{prop_kbound},
if $K$ is the number of nonzero coefficients when solving the Lasso for the data $(A,A\alpha+w\pm U\frac{2\epsilon\lambda}{c^2}A_{:j})$, we have 
\begin{align}
K&\le \frac{\|A\|_{\rm op}^2}{n^2\lambda^2}\|A\alpha+w\pm U\frac{2\epsilon\lambda}{c^2}A_{:j}\|_2^2
\\
&\le \frac{2}{n\lambda^2c^2}
(\frac{n^2}{c^2\delta}M_2+\|w\|_2^2+\lambda O_{\mathcal{P}}(n)).
\end{align}
Using the concentration of the chi-squared distribution 
(see for example 
\cite[p43]{boucheron2013inequalities} or \cite[p57]{liu2020second}),
we have 
\begin{align}
\mathbb{P}[K>cn-1 |
A\in\mathcal{E}_0^c]\le 2e^{-cn}
\label{e_Kbound}
\end{align} 
for large (depending on $\mathcal{P}$ and $\lambda$) $n$ if $\lambda$ satisfies \eqref{e_lam}.
Similarly, we can verify that for $B^{(j)}\in\mathcal{B}_j^c$,
the conditions \eqref{e_cond1}-\eqref{e_cond4} of Lemma~\ref{lem7} hold for regressing $Y=A\alpha+w$ on $B^{(j)}$.
Now 
set $\mathcal{E}:=\bigcup_{j=0}^p\mathcal{E}_0$.
By the union bound, 
\begin{align}
\mathbb{P}[\mathcal{E}]
\le (1+p)\exp(-cn).
\end{align}
By Corollary~\ref{cor_change}, we obtain that for each $j$,
\begin{align}
\mathbb{E}[|\{m\colon \chi_m^{\alpha}
\neq 
\chi_m^{\beta^{(j)}}\}|
1_{\mathcal{E}^c}]
&\le 4
+
2p(\frac{12\epsilon\lambda\ln n}{c^2\sigma^2}
+
12e^{-cn})
\nonumber\\
&\quad+\frac{8}{\lambda^2\epsilon^2c^4}
(\mathbb{E}[|\hat{\alpha}_j^2|1_{\mathcal{E}^c}]
 +
\mathbb{E}[|\hat{\beta}_j^{(j)2}|1_{\mathcal{E}^c}]).
\label{e111}
\end{align}
We have 
\begin{align}
\mathbb{E}[\|\hat{\alpha}\|_2^21_{\mathcal{E}^c}]
&\le \frac{c^2}{n}\mathbb{E}[\|A\hat{\alpha}\|_2^21_{\mathcal{E}^c}]
\\
&\le \frac{2c^2}{n}\mathbb{E}[(\|Y\|_2^2+\|R\|_2^2)1_{\mathcal{E}^c}]
\\
&=O_{\mathcal{P}}(n).
\label{e_114}
\end{align}

Next we will show that  $\sum_{j=1}^p\mathbb{E}[|\hat{\beta}_j^{(j)2}|1_{\mathcal{E}^c}]=O(n)$.
Fix $j\in\{1,\dots,p\}$,
and define $A(t)\in\mathbb{R}^{n\times p}$ such that $A_{:\sj}(t)=A_{:\sj}$ and $A_{:j}(t)=(1-t)A_{:j}+tB^{(j)}_{:j}$.
Thus $A(0)=A$ and $A(1)=B^{(j)}$.
Let $Y(t):=A(t)\alpha+w$,
and let $\hat{\alpha}(t)$ be the Lasso solution for the data $(A,Y(t))$.
Thus $\hat{\alpha}(0)=\hat{\alpha}$, and $\hat{\alpha}(1)$ has the same distribution as $\hat{\beta}^{(j)}$ under the event $\mathcal{E}_0^c\cap\mathcal{E}_j^c$ (since $B^{(j)}$ and $A$ are identically distributed).
Using \eqref{e_lam} and Proposition~\ref{prop_kbound},
we see that under the event $\mathcal{E}_0^c\cap\mathcal{E}_j^c$, we have $\|\hat{\alpha}(t)\|_0\le cn-1$ for all $t\in[0,1]$. 
Let $\mathcal{S}(t)$ be the set of nonzero entries of $\hat{\alpha}(t)$.
From \eqref{eq92}, we see that  under the event $\mathcal{E}_0^c\cap\mathcal{E}_j^c$, we have 
\begin{align}
\left|\frac{d}{dt}\hat{\alpha}_j(t)\right|
&\le 
|\alpha_j|e_j^{\top}(A_{:\mathcal{S}(t)}^{\top}A_{:\mathcal{S}(t)})^{-1}A_{:\mathcal{S}(t)}^{\top}(A_{:j}-B^{(j)}_{:j})
\\
&\le |\alpha_j|\lambda_{\min}^{-2}(A_{:\mathcal{S}(t)})
\|A_{:\mathcal{S}(t)}\|_{\rm op}\|A_{:j}-B^{(j)}_{:j}\|_2
\\
&=O_{\mathcal{P}}(1)|\alpha_j|
\end{align}
for any $t\in[0,1]$ and in particular $|\hat{\alpha}_j(0)-\hat{\alpha}_j(1)|=O_{\mathcal{P}}(1)$ (uniformly over $j$).
Thus
\begin{align}
\mathbb{E}[|\hat{\beta}_j^{(j)}|^21_{\mathcal{E}_0^c\cap\mathcal{E}_j^c}]
&=
\mathbb{E}[|\hat{\alpha}_j(1)|^21_{\mathcal{E}_0^c\cap \mathcal{E}_j^c}]
\\
&\le 
2\mathbb{E}[|\hat{\alpha}_j(0)|^21_{\mathcal{E}_0^c\cap \mathcal{E}_j^c}]
+
2\mathbb{E}[|\hat{\alpha}_j(0)-\hat{\alpha}_j(1)|^21_{\mathcal{E}_0^c\cap \mathcal{E}_j^c}]
\\
&\le 
2\mathbb{E}[|\hat{\alpha}_j|^21_{\mathcal{E}_0^c\cap \mathcal{E}_j^c}]+
O_{\mathcal{P}}(1)\alpha_j^2.
\label{e115}
\end{align}
Thus using \eqref{e_114} we establish that $\sum_{j=1}^p\mathbb{E}[|\hat{\beta}_j^{(j)2}|1_{\mathcal{E}^c}]=O_{\mathcal{P}}(n)$.
The claim then follows from  \eqref{e111}.
\end{proof}

\begin{rem}
From \eqref{e115}, we can see that $\mathbb{E}[|\hat{\beta}_j^{(j)}|^2]= O(\mathbb{E}[|\hat{\alpha}_j|^2] \vee |\alpha_j|^2)$.
However, we cannot obtain that $|\hat{\beta}_j^{(j)}|^2= O(|\hat{\alpha}_j|^2 \vee |\alpha_j|^2)$;
this is because $\hat{\alpha}(1)$ and 
$\hat{\alpha}$ have the same distribution under the event $\mathcal{E}_0^c\cap\mathcal{E}_j^c$, 
but they are not equal pointwise.
\end{rem}

\subsection{Conditioning sub-Gaussians}
We establish auxiliary results for controlling $\check{B}_{:j}^{(j)}$  conditioned on $(A,Y)$.
\begin{lem}\label{lem8}
Suppose that $(X,U)$ is a pair of random variables, where $X\in\mathbb{R}$ satisfies 
\begin{align}
\mathbb{E}[\exp(\frac{X^2}{K})]\le e,
\end{align}
$e$ being the natural base of the logarithm,
$K>0$,
and $U$ takes values in an arbitrary alphabet.
Then for any $\epsilon\in(0,1)$, there is a set $\mathcal{E}$ of $U$
 with probability at most $\epsilon$ such that for each $u\in\mathcal{E}^c$,
 \begin{align}
 \mathbb{E}[\exp(\frac{X^2}{4K\ln\frac{e}{\epsilon}})]\le e.
 \end{align}
\end{lem}
\begin{proof}
Let $\lambda=\sqrt{K\ln\frac{e}{\epsilon}}$ and $L=\frac{K\ln\frac{e}{\epsilon}}{\ln\frac{e}{2}}$.
Define $\mathcal{E}$ as the set of $u$ satisfying
\begin{align}
\mathbb{E}[\exp(\frac{X^2}{K})|U=u]>\frac{e}{\epsilon}
\end{align}
and hence $\mathbb{P}[U\in\mathcal{E}]\le \epsilon$ by the Markov inequality.
For any $u\in\mathcal{E}^c$,
\begin{align}
\mathbb{E}[\exp(\frac{X^2}{L})1_{|X|\le \lambda}|U=u]
\le 
\exp(\frac{\lambda^2}{L})
\le \frac{e}{2},
\end{align}
and
\begin{align}
\mathbb{E}[\exp(\frac{X^2}{L})1_{|X|> \lambda}|U=u]
&\le 
\exp(-\frac{\lambda^2}{K}+\frac{\lambda^2}{L})\mathbb{E}[\exp(\frac{X^2}{K})1_{|X|> \lambda}|U=u]
\\
&\le \exp(-\frac{\lambda^2}{K}+\frac{\lambda^2}{L})\frac{e}{\epsilon}
\\
&\le \frac{e}{2}.
\end{align}
Thus $\mathbb{E}[\exp(\frac{X^2}{L})|U=u]\le e$, and the claim follows.
\end{proof}

\begin{lem}\label{lem9}
Given $\mathcal{P}$, there exist $c>0$ such that for any $j\in\{1,\dots, p\}$,
there exist a set $\mathcal{E}$ of $A_{:\sj}$ such that $\mathbb{P}[\mathcal{E}]\le n^{-10}$ and conditioned on any $A_{:\sj}\in \mathcal{E}^c$, 
$
\check{B}_{:j}^{(j)}
=B_{:j}^{(j)}-\mu_{:j}$
is $O_{\mathcal{P}}(\ln n)$-sub-Gaussian.
\end{lem}
\begin{proof}
By Jensen's inequality and using Definition~\ref{defn_param},
\begin{align}
\mathbb{E}[\exp(\lambda \mu_{1j})]
&=\mathbb{E}[\exp(\lambda \mathbb{E}[A_{1j}|A_{1\sj}])]
\\
&\le \exp(\lambda \mathbb{E}[A_{1j}])
\\
&\le \exp(C_{\rm sg}\lambda^2),\quad \forall \lambda\in\mathbb{R}.
\label{e131}
\end{align}
Therefore both $B_{1j}^{(j)}$ and $\mu_{1j}$ are sub-Gaussian, 
and so is $\check{B}_{1j}^{(j)}$. 
By the equivalence of definitions of sub-Gaussian \citep{vershynin2018high},
there exist $K=K_{\mathcal{P}}$ satisfying 
\begin{align}
\mathbb{E}[\exp(\frac{\check{B}_{1j}^{(j)2}}{K})]\le e.
\end{align}
Applying Lemma~\ref{lem8} with $\epsilon=n^{-11}$ and the union bound, we see that there exists $\mathcal{E}$ of $A_{:\sj}$ such that $\mathbb{P}[\mathcal{E}]\le n^{-10}$ and conditioned on any $A_{:\sj}\in \mathcal{E}^c$, 
\begin{align}
\mathbb{E}[\exp(\frac{\check{B}_{ij}^{(j)2}}{4K\ln(en^{11})})| A_{:\sj}]\le e
\end{align}
and the claim follows since $\check{B}_{1j}^{(j)2},\dots,
\check{B}_{nj}^{(j)2}$ are independent conditioned on $A_{:\sj}$.
\end{proof}

\subsection{Concentration bounds}
We need the following auxiliary result to control the projection of $\check{B}_j^{(j)}$ onto small subspaces.
\begin{prop}\label{thm_union}
[Bound on the maximum norm of sub-Gaussian vectors]
Suppose that $X\in\mathbb{R}^n$ is a zero mean random vector satisfying 
\begin{align}
\mathbb{E}[e^{\left<u,\frac{X}{\sigma}\right>}]\le e^{\|u\|^2},\quad
\forall u\in\mathbb{R}^n
\end{align}
for some $\sigma>0$.
Let $\mathcal{L}$ be a finite collection of subspaces in $\mathbb{R}^n$ of dimension $d$.
Then for any $r>0$,
\begin{align}
\mathbb{P}[\max_{L\in\mathcal{L}}
\|P_LX\|_2\ge r\sigma]
\le 
\exp\left(
-\frac{\sqrt{d}r}{2}+2d+a_d+\ln|\mathcal{L}|
\right),
\end{align}
where $P_L$ denotes the projection onto $L$
and $(a_d)_{d=1}^{\infty}$ is a sequence that converges to 0.
\end{prop}

\begin{proof}
To control the norm of $P_LX$,
we approximate a ball by a polytope which has not too many facets, 
and bound the probability that $P_LX$ lies outside the polytope using the Markov inequality and union bound.
For any subspace $L$ of dimension $d$ and $\theta\in(0,\pi/2)$, it is well-known that we can find a convex polytope $C$ contained in $\{x\colon \|x\|_2\le r\}$ and containing $\{x\colon \|x\|_2\le r\cos\theta\}$, such that the number of facets is $\frac1{\sin^{d+o(d)}\theta}$ (see for example \cite{boroczky2003covering} or
\cite[Section~4]{liu2023soft}).
Then
\begin{align}
\mathbb{P}\left[\frac{\|P_LX\|_2}{\sigma}\ge r\right]
&\le
\mathbb{P}\left[\frac{P_LX}{\sigma}\notin\mathcal{C}\right]
\\
&\le \sum_{u}\mathbb{P}[\left<\sqrt{d}u,X/\sigma\right>\ge \sqrt{d}r\cos\theta]
\\
&\le \sum_{u}\exp(-\sqrt{d}r\cos\theta)\mathbb{E}\left[\exp\left(\left<\sqrt{d}u,X/\sigma\right>\right)\right]
\\
&\le \frac1{\sin^{d+o(d)}\theta}
\exp(-\sqrt{d}r\cos\theta+d)
\end{align}
where the sum is over outward normal unit vectors of $C$.
With the specific choice of $\cos \theta=\frac1{2}$ we have $\sin\theta=\frac{\sqrt{3}}{2}>e^{-1}$, so that $\mathbb{P}\left[\frac{\|P_LX\|_2}{\sigma}\ge r\right]\le \exp(-\frac{\sqrt{d}r}{2}+2d+o(d))$, and the claim follows by taking the union bound.

Alternatively, a bound on the norm of a sub-Gaussian vector can also be obtained through the argument of 
\citet[Exercise~6.3.5]{vershynin2009high}.
\end{proof}

\begin{rem}
From the proof of Proposition~\ref{thm_union},
we can see that we do need the sub-Gaussianity in Definition~\ref{defn_param} to get vanishing errors in the approximation formula (sub-exponential is not sufficient).
\end{rem}

\begin{lem}\label{lem11}
Fix parameters $\mathcal{P}$, 
and let $\varepsilon_n$ be an arbitrary sequence satisfying 
$\varepsilon_n\le\frac1{e^2\delta}$
and 
$\ln\frac1{\varepsilon_n}=O_{\mathcal{P}}(\ln n)$.
Then for all sufficiently large $n$,
we have 
\begin{align}
\max_{\Delta}\|(P_{\mathcal{A}\cup\Delta}-P_{\mathcal{A}\setminus\Delta})\check{B}^{(j)}_{:j}\|_2
&\le
10\sqrt{c\varepsilon_n n\ln n}\ln \frac1{\varepsilon_n}
\label{e129}
\end{align}
for all $j$ with probability at least $1-n^{-8}$, 
where $\mathcal{A}:=\{l\neq j\colon \chi^{\alpha}_l\neq 0\}$, and the max is over $\Delta\subseteq\{1,\dots,p\}\setminus\{j\}$ of size at most $n\varepsilon_n$.
\end{lem}
\begin{proof}
For each $j$, let $\mathcal{E}$ be as in Lemma~\ref{lem9}.
Note that $\mathcal{A}$ is a function of $(A,Y)$, and $P_{\mathcal{A}\cup\Delta}-P_{\mathcal{A}\setminus\Delta}$ is a projection onto a subspace of dimension at most $n\varepsilon_n$.
For any $A$ such that $A_{:\sj}\in \mathcal{E}^c$ we have that
$
\check{B}_{:j}^{(j)}$  is $c\ln n$-sub-Gaussian (for some $c$ depending on $\mathcal{P}$),
and using Proposition~\ref{thm_union} we obtain, for any $r>0$,
\begin{align}
\mathbb{P}\left[\left.\max_{\Delta}
\|(P_{\mathcal{A}\cup\Delta}-P_{\mathcal{A}\setminus\Delta})\check{B}^{(j)}_{:j}\|_2\ge r\sqrt{c\ln n}
\right|
A\right]
\le 
L\exp\left(
-\frac{\sqrt{n\varepsilon_n}r}{2}+2n\varepsilon_n+a_{\lfloor n\varepsilon_n\rfloor}
\right),
\label{e135}
\end{align}
where $L$ denotes the number of candidate subspaces, and we have the bound
\[
L \le \sum_{k=0}^{\lfloor n\varepsilon_n\rfloor} \binom{p}{k}
\le (\lfloor n\varepsilon_n\rfloor+1)\binom{p}{\lfloor n\varepsilon_n\rfloor},
\]
where the second inequality holds when $\lfloor n\varepsilon_n\rfloor \le p/2$.
Using the basic property of binomial coefficients, we have ${p\choose \lfloor n\varepsilon_n\rfloor}\le \exp(\delta^{-1}nh(\delta\varepsilon_n))$ where \[
h(\delta\varepsilon_n)
:=
\delta\varepsilon_n \ln\frac{1}{\delta\varepsilon_n}
+
(1-\delta\varepsilon_n)\ln\frac{1}{1-\delta\varepsilon_n}
<
2\delta\varepsilon_n \ln\frac{1}{\delta\varepsilon_n}
\]
where the last step used $\delta\varepsilon_n \le  e^{-2}$.
With $r\leftarrow 10\sqrt{n\varepsilon_n}\ln\frac1{\varepsilon_n}$, 
\eqref{e135} is upper bounded by 
$\exp(-n\varepsilon_n\ln\frac1{\varepsilon_n})$ for sufficiently large $n$.
The claim then follows by combining \eqref{e135} with the bound on $\mathbb{P}[\mathcal{E}]$ in Lemma~\ref{lem9}.
\end{proof}

\begin{rem}\label{rem_4}
In the proof of Lemma~\ref{lem11},
it appears that it is essential that we defined $\mu_{:j}$ as the conditional mean, rather than the best linear estimator $A_{:\sj}\Sigma_{\sj}^{-1}\Sigma_{\sj j}$.
Indeed, if we took $\mu_{:j}:=A_{:\sj}\Sigma_{\sj}^{-1}\Sigma_{\sj j}$,
then the residual $\check{B}^{(j)}_{:j}
:=
B^{(j)}_{:j}
-\mu_{:j}$
would have zero inner product with any \emph{linear} functions of $A_{:\sj}$,
but not with general functions,
so the left side of \eqref{e129} might not be small.
\end{rem}

\subsection{Proof of Theorem~\ref{thm12}}
Define
\begin{align}
\mathcal{I}_3(j,l)
&=
1\{\chi^{\alpha}_l\neq \chi^{\beta^{(j)}}_l\};
\\
\mathcal{I}_4(j)
&:=1\left\{\frac1{n}\sum_l
\mathcal{I}_3(j,l)
>\varepsilon
\right\}.
\end{align}
It follows from Corollary~\ref{cor7} that
\begin{align}
\frac1{p}\sum_{j=1,l}^p\mathbb{E}[\mathcal{I}_3(j,l)]
=
\tilde{O}_{\mathcal{P}}(n^{2/3}).
\end{align}
Then	 (writing $\bar{\mathcal{I}}:=1-\mathcal{I}$), 
\begin{align}
\mathbb{P}
\left[
\frac1{p}\sum_j
\mathcal{I}_4(j)
>\delta_1
\right]
&\le 
\frac1{p\delta_1}
\sum_j
\mathbb{P}\left[\frac1{n}\sum_l
\mathcal{I}_3(j,l)
>\varepsilon
\right]
\\
&\le 
\frac1{\delta_1\varepsilon pn}
\sum_{j,l}
\mathbb{E}[\mathcal{I}_3(j,l)]
\\
&\le
\tilde{O}_{\mathcal{P}}(p^{-1/3}\delta_1^{-1}\varepsilon^{-1})
\label{eqq118}
\end{align}
where $\varepsilon$ and $\delta_1$ will be optimized later (see \eqref{e_opt_d1}).
Let $c$ be as in Corollary~\ref{cor7}, 
and let
\begin{align}
\mathcal{I}_8
:=
1\{
\|A\|_{\rm op}> \sqrt{n}/c
\textrm{ 
or \eqref{e129} fails for some $j$}\}.
\end{align} 
Applying Theorem~\ref{thm_error} with $D=\frac1{c}$ and $\Gamma=10\sqrt{c\ln n}\ln \frac1{\varepsilon}$, 
we see that under the event $\mathcal{I}_4(j)=\mathcal{I}_8=0$,
\begin{align}
|t(j,B^{(j)},Y)-t_j|
\le 
10\sqrt{\frac{\ln n}{c}}
\ln \frac1{\varepsilon}
\cdot
\sqrt{\varepsilon}(|\hat{\alpha}_j|+|\hat{\beta}^{(j)}_j|).
\label{e136}
\end{align}

To finish the proof, we will control the right side of \eqref{e136}.
Let 
\begin{align}
\mathcal{I}_6(j)
:=1\{\max\{
|\hat{\alpha}_j|,
|\hat{\beta}^{(j)}_j|
\}>\delta_3^{-1/2}\},
\end{align} 
and next we will show that $\mathcal{I}_6(j)=0$ for most $j$. 
Using \eqref{e_114}, \eqref{e115} and the Markov inequality, 
we have that for any $\delta_3>0$,
\begin{align}\max\left\{\frac1{p}\sum_{j=1}^p
\mathbb{P}[|\hat{\alpha}_j|>\delta_3^{-1/2}]
,\,
\frac1{p}\sum_{j=1}^p
\mathbb{P}[|\hat{\beta}^{(j)}_j|>\delta_3^{-1/2}]
\right\}
=\tilde{O}_{\mathcal{P}}(\delta_3)
\end{align}
and 
\begin{align}
\mathbb{P}[\frac1{p}\sum_j\mathcal{I}_6(j)>\delta_1]
=\tilde{O}_{\mathcal{P}}(\delta_3/\delta_1).
\label{e138}
\end{align}
Returning to \eqref{e136}, 
we see that under the event $\mathcal{I}_4(j)
=\mathcal{I}_8
=\mathcal{I}_6(j)
=0$ we have 
\begin{align}
|t(j,B^{(j)},Y)-t_j|
=\tilde{O}_{\mathcal{P}}
\left(\sqrt{\frac{\varepsilon}{\delta_3}}\ln \frac1{\varepsilon}
\right).
\label{e143}
\end{align}
From \eqref{eqq118}, \eqref{e138}
and Lemma~\ref{lem11}, 
we have
\begin{align}
\mathbb{P}[\frac1{p}\sum_j\mathcal{I}_4(j)
\vee \mathcal{I}_6(j)>2\delta_1
\textrm{ or }\mathcal{I}_8]
=\tilde{O}_{\mathcal{P}}(p^{-1/3}\delta_1^{-1}\varepsilon^{-1})
+
\tilde{O}_{\mathcal{P}}(\delta_3/\delta_1)
+n^{-8}.
\label{e144}
\end{align}
The values of $\delta_1,\delta_3$, and $\varepsilon$ can be optimized;
here we specify them in a way such that the Levy-Prokhorov error in 
Corollary~\ref{cor_lp} will be minimized.
Thus we will set \eqref{e143}, $2\delta_1$, and the right side of \eqref{e144} to be the same order (up to log factors), which gives 
\begin{align}\sqrt{\frac{\varepsilon}{\delta_3}}
=\delta_1
=p^{-1/3}\delta_1^{-1}\varepsilon^{-1}
=\frac{\delta_3}{\delta_1},
\label{e_opt_d1}
\end{align}
yielding $\delta_1=p^{-1/18}$, 
$\varepsilon=p^{-2/9}$,
and $\delta_3=p^{-1/9}$.
which yields the claimed convergence rate.

It remains to show the claim about further simplification when $B^{(j)}_{:j}$ and $A_{:j}$ are independent of $A_{:\sj}$.
Recall that for each $j$, $(I-P_j)A_{:j}$ is a function of $(A,Y)$, 
and $\check{B}^{(j)}_{:j}$ is $\tilde{O}_{\mathcal{P}}(1)$-sub-Gaussian conditioned on a set $\mathcal{G}_j$ of $A$ with probability $1-O(n^{-10})$ (by Lemma~\ref{lem9}).
Therefore conditioned on $G_j$ and $\mathcal{I}_8$,
$
\frac{1}{n}\check{B}^{(j)\top}_{:j}(I-P_j)A_{:j}
$
is sub-Gaussian with variance proxy $\tilde{O}_{\mathcal{P}}(\|(I-P_j)A_{:j}\|_2^2/n^2)=\tilde{O}_{\mathcal{P}}(\frac1{n})$.
Hence conditioned on $\mathcal{G}_j$ and $\mathcal{I}_8$,
$
|\frac{1}{n}\check{B}^{(j)\top}_{:j}(I-P_j)A_{:j}|
\le 
\tilde{O}_{\mathcal{P}}(\frac1{\sqrt{n}})
$ with probability at least $1-n^{-20}$.
Therefore 
\begin{align}
|\frac{1}{n}\check{B}^{(j)\top}_{:j}(I-P_j)A_{:j}|
\le 
\tilde{O}_{\mathcal{P}}(\frac1{\sqrt{n}}),
\quad \forall j\in\{1,\dots,p\}
\end{align}
with probability at least $1-pn^{-20}-\sum_{j=1}^p\mathbb{P}[\mathcal{G}_j^c]-\mathbb{E}[\mathcal{I}_8]\ge 1-\tilde{O}_{\mathcal{P}}(n^{-1/18})$.
Then
\begin{align}
|\frac1{n}\check{B}^{(j)\top}_{:j}(I-P_j)A_{:j}
\hat{\alpha}_j|=\tilde{O}_{\mathcal{P}}((n\delta_3)^{-1/2}),
\quad 
\forall  j\colon \mathcal{I}_6(j)=0
\end{align}
with probability at least $1-\tilde{O}_{\mathcal{P}}(n^{-1/18})$.
Thus the $\tilde{O}_{\mathcal{P}}(n^{-1/18})$ rate bound also holds when the term 
$\check{B}^{(j)\top}_{:j}(I-P_j)A_{:j}\hat{\alpha}_j$ is dropped from the definition of $t_j$ in Theorem~\ref{thm12}.


\section{The Case of non-vanishing $\mathbb{E}[|\check{B}^{(j)}_{1j}|^2]$}
\label{sec_nonvanish}

In this section we prove Theorem~\ref{thm_lbdd},
which
is essentially based on showing that the denominator $\frac1{n}\check{B}_{:j}^{\top}(I-P_j)B_{:j}$ is bounded below.

\subsection{Auxiliary lemmas}

\begin{lem}\label{lem12}
Assume that condition $\mathcal{P}$ holds and that 
\begin{align}
\mathbb{E}[|\check{B}^{(j)}_{1j}|^2]
\ge \kappa_2^{-1},
\quad
\forall j=1,\dots, p,
\label{e_lb1}
\end{align}
for some $\kappa_2>0$ independent of $n$.
There exists $c>0$ such that the following holds: For all $n\ge 1/c$
and $j\in\{1,\dots,n\}$,
with probability at least $1-\exp(-cn)$ we have
\begin{align}
\frac1{n}\inf_P
\check{B}^{(j)\top}_{:j}
(I-P)
\check{B}^{(j)}_{:j}\ge c^2,
\end{align}
where the infimum is over projection $P$ onto the span of no more than $cn-1$ columns of $A$ excluding the $j$-th column.
\end{lem}

\begin{proof}
Let us assume without loss of generality that $j=1$.
The assumption of the optimality of the linear prediction implies that 
the covariance matrix of
\begin{align}
[
\check{B}^{(1)}_{11},
A_{12},\dots,
A_{1p}]
\label{e_newsg}
\end{align}
is a block-diagonal matrix with diagonal blocks $\mathbb{E}[|\check{B}^{(1)}_{11}|^2]
\in (\kappa_2^{-1},\kappa_1)$, and $\Sigma_{\setminus 1}$. 
In particular, the max and min eigenvalues of such a covariance matrix is bounded between $\kappa_1$
and 
$ (\kappa_1\vee \kappa_2)^{-1}$.
Moreover, note that both $A_{1:}$ and $\mu_{11}$ are $C_{\rm sg}$-sub-Gaussian (the latter can be seen by Jensen's inequality).
Therefore using Cauchy-Schwarz we see that \eqref{e_newsg} is $4C_{\rm sg}$-sub-Gaussian. 
Then we establish the claim by following the same proof as  Lemma~\ref{lem6},
but replacing $A$ with 
$[
\check{B}^{(1)}_{:1},
A_{:2},\dots,
A_{:p}]$.
\end{proof}

\begin{lem}\label{lem10}
Assuming $\mathcal{P}$  and 
\eqref{e_lb1},
and 
$\lambda$ large enough (i.e.\ satisfying \eqref{e_lam}),
there exists $c_1>0$ (depending on $\mathcal{P}$) such that with probability at least $1-O(n^{-10})$, we have 
\begin{align}
\frac1{n}\check{B}_{:j}^{(j)\top}(I-P_j)B^{(j)}_{:j}>c_1,
\quad \forall j=1,\dots,p,
\label{e128}
\end{align} 
for large enough (depending on $\mathcal{P}$ and $\lambda$) $n$.
\end{lem}
\begin{proof}
Let $c>0$ be as in Corollary~\ref{cor7}.
As shown in \eqref{e_Kbound},
with probability at least $1-3e^{-cn}$, $P_j$ is a projection onto a subspace of dimension at most $cn-1$, 
and hence 
\begin{align}
\frac1{n}\check{B}_{:j}^{(j)\top}(I-P_j)\check{B}^{(j)}_{:j}>c^2
\label{e_bcb}
\end{align}
by Lemma~\ref{lem12}
(note that $B^{(j)}$ and $A$ have the same distribution).
Moreover, let $\mathcal{E}$ be as in Lemma~\ref{lem9}.
Conditioned on any $A_{:\sj}\in \mathcal{E}^c$ we have that
$
\frac1{n}\check{B}_{:j}^{(j)\top}(I-P_j)\mu_{:j}
$  is $O_{\mathcal{P}}(\frac1{n^2}\|\mu_{:j}\|_2^2\ln n)$-sub-Gaussian and, since $\mu_{:j}$ is $C_{\rm sg}$-sub-Gaussian (which follows by sub-Gaussianity of $B^{(j)}_{:j}$ and Jensen's inequality), we have $\|\mu_{:j}\|_2^2=O_{\mathcal{P}}(n)$ with probability at least $1-e^{-n}$
\citep[Exercise~6.3.5]{vershynin2009high}.
It follows that $
\frac1{n}\check{B}_{:j}^{(j)\top}(I-P_j)\mu_{:j}
$  is $O_{\mathcal{P}}(\frac1{n}\ln n)$-sub-Gaussian conditioned on $A_{:\sj}$ in a set of probability at least $1-\mathbb{P}[\mathcal{E}]
-e^{-n}
=
1-n^{-10}
-e^{-n}$.
Then,
without conditioning on $A_{:\sj}$, we have $
\frac1{n}\check{B}_{:j}^{(j)\top}(I-P_j)\mu_{:j}
=O_{\mathcal{P}}(\sqrt{\frac1{n}\ln n})$ with probability at least $1-O(n^{-10})$.
Then the claim follows from 
\eqref{e_bcb},
$B^{(j)}_{:j}=\check{B}^{(j)}_{:j}+\mu_{:j}$,
and $c_1:=c^2/2<c^2-O_{\mathcal{P}}(\sqrt{\frac1{n}\ln n})$ for large $n$.
\end{proof}

\subsection{Proof of Theorem~\ref{thm_lbdd}}
Recall that 
in \eqref{e143} and with the choice of parameters in \eqref{e_opt_d1},
we established that 
\begin{align}
|\frac1{n}
\check{B}_{:j}^{(j)\top}(I-P^{(j)})
B^{(j)}_{:j}\hat{\beta}^{(j)U}_j-\frac1{n}\check{B}_{:j}^{(j)\top}R
-
\frac1{n}\check{B}^{(j)\top}_{:j}(I-P_j)A_{:j}
\hat{\alpha}_j|
=\tilde{O}_{\mathcal{P}}
\left(n^{-1/18}
\right)
\label{e155}
\end{align}
for all $j$ except for a set of size at most $2n^{-1/18}p$,
under the event 
\begin{align}
\frac1{p}\sum_j\mathcal{I}_4(j)
\vee \mathcal{I}_6(j)\le 2\delta_1
\textrm{ and }\mathcal{I}_8=0
\label{e_event}
\end{align}
which has probability $1-\tilde{O}_{\mathcal{P}}(n^{-1/18})$.
Now by Lemma~\ref{lem9},
there exists a set $\mathcal{E}$ of $A_{:\sj}$,
$\mathbb{P}[\mathcal{E}]\le n^{-10}$, such that conditioned on any $(A,Y)$ satisfying $A_{:\sj}\in \mathcal{E}^c$,
$
\frac1{n}\check{B}_{:j}^{(j)\top}R
$
is sub-Gaussian with variance proxy
$O_{\mathcal{P}}(\frac1{n^2}\|R\|_2^2\ln n)\le O_{\mathcal{P}}(\frac1{n^2}\|Y\|_2^2\ln n)$.
Thus conditioned on any $A_{:\sj}\in \mathcal{E}^c$,
$
\frac1{n}\check{B}_{:j}^{(j)\top}R
=O_{\mathcal{P}}(\ln n)
$ with probability $1-e^{-n}$.
Thus unconditionally,
\begin{align}
\frac1{n}\check{B}_{:j}^{(j)\top}R
=O_{\mathcal{P}}(\ln n)
\label{e152}
 \end{align}
with probability $1-O(n^{-10})$.
Under the event of \eqref{e_event} and for the same set of $j$,
by a similar argument,
we have $\|\frac1{n}(I-P_j)A_{:j}\hat{\alpha}_j\|_2
\le \frac1{n}\|A_{:j}\|_2\delta_3^{-1/2}=O_{\mathcal{P}}(\frac1{\sqrt{n}}\delta_3^{-1/2})$ and hence
\begin{align}
\frac1{n}\check{B}^{(j)\top}_{:j}(I-P_j)A_{:j}
\hat{\alpha}_j
=\tilde{O}_{\mathcal{P}}(\frac1{\sqrt{n}}\delta_3^{-1/2})
\label{e153}
\end{align}
with probability $1-O(n^{-10})$.
Moreover, 
from \eqref{e_event} 
and Lemma~\ref{lem11}
we see that with probability $1-\tilde{O}_{\mathcal{P}}(n^{-1/18})$,
\begin{align}
\frac1{n}
\check{B}_{:j}^{(j)\top}(P^{(j)}-P_j)
B^{(j)}_{:j}=o_{\mathcal{P}}(1)
\end{align}
for all $j$ except for a set of size $\tilde{O}_{\mathcal{P}}(pn^{-1/18})$.
Then using \eqref{e128} we have that with probability $1-\tilde{O}_{\mathcal{P}}(n^{-1/18})$,
\begin{align}
\frac1{n}
\check{B}_{:j}^{(j)\top}(I-P^{(j)})
B^{(j)}_{:j}
=
\frac1{n}
\check{B}_{:j}^{(j)\top}(I-P_j)
B^{(j)}_{:j}
-o_{\mathcal{P}}(1)
=\Omega_{\mathcal{P}}(1)
\end{align}
for all $j$ except for a set of size $\tilde{O}_{\mathcal{P}}(pn^{-1/18})$.
Thus under 
\eqref{e_event}, \eqref{e152} and 
\eqref{e153},
from 
\eqref{e155} we see that 
\begin{align}
\hat{\beta}^{(j)U}_j
=O_{\mathcal{P}}(\ln n)
\end{align}
and hence
\begin{align}
&\quad|\frac1{n}
\check{B}_{:j}^{(j)\top}(I-P_j)
B^{(j)}_{:j}\hat{\beta}^{(j)U}_j-\frac1{n}\check{B}_{:j}^{(j)\top}R
-
\frac1{n}\check{B}^{(j)\top}_{:j}(I-P_j)A_{:j}
\hat{\alpha}_j|
\nonumber\\
&\le 
\tilde{O}_{\mathcal{P}}
\left(n^{-1/18}
\right)
+
\left|\frac1{n}
\check{B}_{:j}^{(j)\top}(P^{(j)}-P_j)
B^{(j)}_{:j}\hat{\beta}^{(j)U}_j\right|
\\
&\le \tilde{O}_{\mathcal{P}}
\left(n^{-1/18}
\right).
\end{align}
The claim then follows by dividing both sides by $\frac1{n}
\check{B}_{:j}^{(j)\top}(I-P_j)
B^{(j)}_{:j}$,
and noting that \eqref{e_event}, \eqref{e152} and 
\eqref{e153} hold simultaneously with probability at least $1-\tilde{O}_{\mathcal{P}}
\left(n^{-1/18}\right)$.

\section{Gaussian designs: Proof of Theorem~\ref{thm_gaussian}}
\label{sec_gaussian}
In this section we 
prove 
that our definition of the generalized debiased estimators 
$\hat{\beta}^{U(j)}$, $\hat{\alpha}^U$
asymptotically matches the traditional debiased estimators
$\hat{\beta}^{u(j)}$, $\hat{\alpha}^u$ in the case of Gaussian designs.
Recall \ eqref{eqq118} with parameters chosen in \eqref{e_opt_d1} states that 
\begin{align}
\mathbb{P}
\left[
\frac1{p}\sum_j
\mathcal{I}_4(j)
>p^{-1/18}
\right]
\le \tilde{O}_{\mathcal{P}}(p^{-1/18}).
\end{align}
Thus with probability at least $1-\tilde{O}_{\mathcal{P}}(p^{-1/18})$,
there is a set of $j$ of cardinality at least $p(1-p^{-1/18})$, for which 
\begin{align}
\frac1{n}\sum_l
1\{\chi_l^{\alpha}
\neq
\chi_l^{\beta^{(j)}}
 \}
&\le 
p^{-2/9},
\label{e157}
\end{align}
and hence by Lemma~\ref{lem11},
\begin{align}
\max_{\Delta}\|(P_{\mathcal{A}}-P_{\mathcal{B}})\check{B}^{(j)}_{:j}\|_2
&\le
\tilde{O}_{\mathcal{P}}
(p^{\frac{7}{18}}).
\label{e158}
\end{align}
Now in the Gaussian case,
$\check{B}_{:j}$ is independent of the rest of the columns,
and 
$\mu_{:j}=A_{:\setminus j}\Sigma_{\setminus j}^{-1}\Sigma_{\sj j}$.
We have 
\begin{align}
\frac1{n}\check{B}_{:j}^{\top}(I-P_{\mathcal{B}})B_{:j}
&=\frac1{n}\check{B}_{:j}^{\top}(I-P_{\mathcal{B}})(\check{B}_{:j}+A_{:\setminus j}\Sigma_{\setminus j}^{-1}\Sigma_{\sj j}).
\end{align}
Then 
we see that with probability  $1-\tilde{O}_{\mathcal{P}}(p^{-1/18})$,
\begin{align}
\frac1{n}\check{B}_{:j}^{\top}(I-P_{\mathcal{A}})(\check{B}_{:j}+A_{:\setminus j}\Sigma_{\setminus j}^{-1}\Sigma_{\sj j})
&=\frac1{n}\check{B}_{:j}^{\top}(I-P_{\mathcal{A}})\check{B}_{:j}+O_{\mathcal{P}}(\frac1{\sqrt{n}})
\label{eqq11}
\\
&=\Sigma_{j|\sj}(1-\frac1{n}\|\hat{\alpha}\|_0)
+O_{\mathcal{P}}(n^{-2/9})
\label{e161}
\\
&=\Sigma_{j|\sj}(1-\frac1{n}\|\hat{\beta}^{(j)}\|_0)
+O_{\mathcal{P}}(n^{-2/9})
\label{e162}
\end{align}
where \eqref{eqq11} follows since with probability $1-e^{-n}$,
$(I-P_{\mathcal{A}})A_{:\setminus j}\Sigma_{\setminus j}^{-1}\Sigma_{\sj j}
$ 
has norm $O_{\mathcal{P}}(\sqrt{n})$,
and $\check{B}_{:j}$ is
independent of it with norm $O_{\mathcal{P}}(\sqrt{n})$.
\eqref{e161} follows from 
concentration of the Chi-square distribution and the independence of $P_{\mathcal{A}}$ and $\check{B}_{:j}$.
\eqref{e162} follows from 
\eqref{e157}.
Moreover, with probability  $1-\tilde{O}_{\mathcal{P}}(p^{-1/18})$,
we also have the error bound
\begin{align}
&\quad\frac1{n}|\check{B}_{:j}^{\top}(P_{\mathcal{A}}-P_{\mathcal{B}})(\check{B}_{:j}+A_{:\setminus j}\Sigma_{\setminus j}^{-1}\Sigma_{\sj j})|
\nonumber\\
&\le
\frac1{n}|\check{B}_{:j}^{\top}(P_{\mathcal{A}}-P_{\mathcal{B}})\check{B}_{:j}|
+
\frac1{n}\|\check{B}_{:j}^{\top}(P_{\mathcal{A}}-P_{\mathcal{B}})\|_2
\|A_{:\setminus j}\Sigma_{\setminus j}^{-1}\Sigma_{\sj j}\|_2
\\
&= 
\tilde{O}_{\mathcal{P}}(n^{-2/9})
+
\frac1{n}\cdot\tilde{O}_{\mathcal{P}}(n^{7/18})
\cdot O_{\mathcal{P}}(\sqrt{n})
\\
&=\tilde{O}_{\mathcal{P}}
(n^{-1/9})
\label{e140.1}
\end{align}
where we used \eqref{e158}.
This establishes $\frac1{n}\check{B}_{:j}^{\top}(I-P_{\mathcal{B}})B_{:j}
=
\Sigma_{j|\sj}(1-\|\hat{\beta}^{(j)}\|_0/n)
+O_{\mathcal{P}}(n^{-1/9})$.
Moreover, from Lemma~\ref{lem10} and \eqref{e140.1},
there exists $\lambda_{\mathcal{P}}$ such that for all $\lambda \ge\lambda_{\mathcal{P}}$,
with probability
at least $1-\tilde{O}_{\mathcal{P}}(p^{-1/18})$,
set of $j$ of cardinality at least $p(1-p^{-1/18})$,
the denominator $\frac1{n}\check{B}^{\top}_{:j}(I-P^{(j)})B_{:j}$ is bounded away from 0,
which implies $|\hat{\beta}^{U(j)}_j
-\hat{\beta}^{u(j)}_j|
=O_{\mathcal{P}}(n^{-1/9})$.

We now turn to $\hat{\alpha}^U$, 
whose proof is similar to the above.
With probability at least $1-\tilde{O}_{\mathcal{P}}(p^{-1/18})$,
there is a set of $j$ of cardinality at least $p(1-p^{-1/18})$, 
for which \eqref{e157} holds,
and 
\begin{align}
\frac1{n}\check{A}_{:j}^{\top}(I-P_{\mathcal{B}})(\check{A}_{:j}+A_{:\setminus j}\Sigma_{\setminus j}^{-1}\Sigma_{\sj j})
&=
\Sigma_{j|\sj}(1-\frac1{n}\|\hat{\alpha}\|_0)
+O_{\mathcal{P}}(n^{-2/9})
\label{e141.1}
\end{align}
by steps similar to \eqref{e162}.
Now 
Let $\mathcal{I}_1(j)$ be the probability that \eqref{e129} is not satisfied, which, by symmetry, is also the probability that
\begin{align}
\max_{\Delta}\|(P_{\mathcal{B}\cup\Delta}-P_{\mathcal{B}\setminus\Delta})\check{A}^{(j)}_{:j}\|_2
&\le
\tilde{O}_{\mathcal{P}}(p^{\frac{7}{18}})
\label{e141}
\end{align}
does not hold, 
where the max is over the same collection of $\Delta$ as in \eqref{e158}.
By Lemma~\ref{lem11}, $\sum_{j=1}^p\mathcal{I}_1(j)\le n^{-8}p$.
By Markov's inequality, 
with probability at least $1-n^{-4}$, 
\eqref{e141} is violated only by a set of $j$ of size at most $n^{-4}p$.
Then by \eqref{e157}, \eqref{e141}, and the union bound,  
we have 
\begin{align}
\max_{\Delta}\|(P_{\mathcal{B}}-P_{\mathcal{A}})\check{A}^{(j)}_{:j}\|_2
&\le
\tilde{O}_{\mathcal{P}}(p^{\frac{7}{18}})
\label{e142}
\end{align}
for all $j$ in a good set of size at least $p(1-p^{-1/18}-n^{-4})$, 
with probability at least $1-\tilde{O}_{\mathcal{P}}(p^{-1/18})$,
and hence
\begin{align}
&\quad\frac1{n}|\check{A}_{:j}^{\top}(P_{\mathcal{A}}-P_{\mathcal{B}})A_{:j}|
\nonumber\\
&\le
\frac1{n}|\check{A}_{:j}^{\top}(P_{\mathcal{A}}-P_{\mathcal{B}})\check{A}_{:j}|
+
\frac1{n}\|\check{A}_{:j}^{\top}(P_{\mathcal{A}}-P_{\mathcal{B}})\|_2
\|A_{:\setminus j}\Sigma_{\setminus j}^{-1}\Sigma_{\sj j}\|_2
\\
&= 
\tilde{O}_{\mathcal{P}}(n^{-2/9})
+
\frac1{n}\cdot\tilde{O}_{\mathcal{P}}(n^{7/18})
\cdot O_{\mathcal{P}}(\sqrt{n})
\\
&=\tilde{O}_{\mathcal{P}}
(n^{-1/9})
\label{e146}
\end{align}
where we used \eqref{e142}.
Then $\frac1{n}\check{A}_{:j}^{\top}(I-P_{\mathcal{A}})A_{:j}
=
\Sigma_{j|\sj}(1-\|\hat{\alpha}\|_0/n)
+O_{\mathcal{P}}(n^{-1/9})$ follows by \eqref{e141.1} and \eqref{e146} (the cardinality estimate $p(1-3p^{-1/18})$ arises from taking the union bound for sets of $j$).
Finally, from the proof of Lemma~\ref{lem10},
there exists $\lambda_{\mathcal{P}}$ such that for all $\lambda \ge\lambda_{\mathcal{P}}$,
with high probability $1-\|\hat{\alpha}\|_0/n$ is bounded below and away from 0,
and hence by \eqref{e146},
the denominator $\Sigma_{j|\sj}(1-\frac1{n}\|\hat{\alpha}\|_0)$ in the definition of $\hat{\alpha}^u_j$ is bounded away from 0,
which implies $|\hat{\alpha}^U_j
-\hat{\alpha}^u_j|
=O_{\mathcal{P}}(n^{-1/9})$.

\section{Proof of matrix continuity}
\begin{lem}\label{lem16}
The map \eqref{e_gmap} to the set of positive semidefinite matrices,
initially defined on $(0,+\infty)^{p-1}$,
has a unique continuous extension to $[0,+\infty]^{p-1}$.
\end{lem}

\begin{proof}
Let us assume without loss of generality that we have convergence $a\to a^*\in[0,+\infty]^{p-1}$, where 
$a\in (0,+\infty)^{p-1}$, and $a^*_i=+\infty$ iff $i\in \{l+1,\dots, p-1\}$ for some $l$. 
Let 
$D_1\in\mathbb{R}^{l\times l}$ and $D_2\in\mathbb{R}^{(p-1-l)\times (p-1-l)}$ be the diagonal matrices such that
\[
D=
\begin{pmatrix}
D_1 & 0 \\
0 & D_2
\end{pmatrix}
=\diag(a),
\]
and let $D,D_1^*,D_2^*$ be the corresponding limits as $a\to a^*$.
Denote by $L\in\mathbb{R}^{n\times l}$,
$F\in\mathbb{R}^{n\times (p-1-l)}$, such that 
\begin{align}
A_{:\sj}
=(L,F).
\end{align}
By the Schur complement theorem,
\begin{align}
\left[
\begin{pmatrix}
L^{\top}L & L^{\top}F
\\
F^{\top}L & F^{\top}F
\end{pmatrix}
+D
\right]^{-1}
=
\begin{pmatrix}
(L^{\top}L+D_1
-L^{\top}F(F^{\top}F+D_2)^{-1}F^{\top}L)^{-1}
& *\\
* & *
\end{pmatrix}
.
\end{align}
Define
\begin{align}
H:=
L(L^{\top}L+D_1
-L^{\top}F(F^{\top}F+D_2)^{-1}F^{\top}L)^{-1}
L^{\top}.
\end{align}
By the Schur complement theorem,
\begin{align}
A_{:\sj}(A_{:\sj}^{\top}A_{:\sj}+D)^{-1}
A_{:\sj}^{\top}
&=
H
-F(F^{\top} F+D_2)^{-1}F^{\top}H
-HF(F^{\top} F+D_2)^{-1}F^{\top}
\nonumber\\
&\quad+F[(F^{\top} F+D_2)^{-1}+
(F^{\top} F+D_2)^{-1}
F^{\top}HF(F^{\top} F+D_2)^{-1}]F^{\top}
\label{es190}
\end{align}
which converges to $\lim_{a\to a^*}H$ if the latter exists.

It remains to prove the existence of $\lim_{a\to a^*}H$.
Note that $L^{\top}L+D_1
-L^{\top}F(F^{\top}F+D_2)^{-1}F^{\top}L\to L^{\top}L+D_1^*$
as $a\to a^*$, which may be singular, but is non-singular when restricted to the column space $\col(L^{\top})$,
suggesting that $H$ may still have a limit.
To prove this,
choose any orthogonal matrix $E=\begin{pmatrix}
E_1 \\ E_2
\end{pmatrix}\in\mathbb{R}^{l\times l}$,
such that $\col(E_1^{\top})= \col(L^{\top})$.
Define 
\begin{align}
\Phi&:=
\begin{pmatrix}
E_1\\
E_2
\end{pmatrix}
[L^{\top}L+D_1
-L^{\top}F(F^{\top}F+D_2)^{-1}F^{\top}L]^{-1}
(E_1^{\top} E_2^{\top})
\\
&=
\begin{pmatrix}
E_1L^{\top}MLE_1^{\top}
+E_1D_1E_1^{\top}
&
E_1D_1E_2^{\top}
\\
E_2D_1E_1^{\top}
&
E_2D_1E_2^{\top}
\end{pmatrix}
^{-1}
\end{align}
where we defined 
\begin{align}
M:=I-F(F^{\top}F+D_2)^{-1}F^{\top}
\end{align}
which converges to $I$ since the diagonal coordinates in $D_2$ tend to $+\infty$.
Now $H=LE_1^{\top}\Phi_{11}E_1L^{\top}$,
where 
\begin{align}
\Phi_{11}
=(E_1L^{\top}MLE_1^{\top}
+E_1D_1E_1^{\top}
+
E_1D_1E_2^{\top}
(E_2D_1E_2^{\top})^{-1}
E_2D_1E_1^{\top})^{-1}
\end{align}
denotes the upper left block of $\Phi$.
Since $E_1L^{\top}LE_1^{\top}$ is invertible, we see that existence of $\lim_{a\to a^*}H$ would follow from existence of $\lim_{a\to a^*}D_1E_2^{\top}
(E_2D_1E_2^{\top})^{-1}
E_2D_1$.

It remains to prove the existence of $\lim_{a\to a^*}D_1E_2^{\top}
(E_2D_1E_2^{\top})^{-1}
E_2D_1$.
Assume without loss of generality that $D_1=\begin{pmatrix}
    D_{11} & 0\\
    0 & D_{12}
\end{pmatrix}$,
where the limits $D_{11}^*=0$ and $D_{12}^*$ is nonsingular.
Assume that $E_2=(E_{21} E_{22})$, where $E_{21}$ and $D_{11}$ have the same number of columns.
Note that $E_2D_1E_2^{\top}\to E_{22}D_{12}^*E_{22}^{\top}$, which may not be invertible.
To tackle this, choose any orthogonal matrix $G=(G_1 G_2)\in \mathbb{R}^{l_1\times l_1}$,
where $l_1$ is the number of rows in $E_2$, such that $\col(G_1)=\col(E_{22})$.
Then 
\begin{align}
D_1E_2^{\top}
(E_2D_1E_2^{\top})^{-1}
E_2D_1
=
\begin{pmatrix}
D_{11}E_{21}^{\top}G_1 &
D_{11}E_{21}^{\top}G_2 \\
D_{12}E_{22}^{\top}G_1 &
0
\end{pmatrix}
V
\begin{pmatrix}
D_{11}E_{21}^{\top}G_1 &
D_{11}E_{21}^{\top}G_2 \\
D_{12}E_{22}^{\top}G_1 &
0
\end{pmatrix}^{\top},
\label{e154}
\end{align}
where we defined the $2\times 2$ block matrix 
\begin{align}
V^{-1}
:=
\begin{pmatrix}
* & U^{\top} \\
U & W
\end{pmatrix}
=
\begin{pmatrix}
G_1^{\top}
\\
G_2^{\top}
\end{pmatrix}
E_{21}D_{11}E_{21}^{\top}
(G_1 G_2)
+
\begin{pmatrix}
G_1^{\top}E_{22}D_{12}E_{22}^{\top}G_1 &0 \\
0&0
\end{pmatrix}.
\label{e159}
\end{align}
Using the Schur complement theorem
and the fact that the first term in the rightmost side of \eqref{e159} is positive semidefinite, 
we can show that the top left block $V_{11}$ 
is bounded between 
$(G_1^{\top}E_{22}D_{12}E_{22}^{\top}G_1)^{-1}$
and
$(G_1^{\top}E_{22}D_{12}E_{22}^{\top}G_1
+
G_1^{\top}E_{21}D_{11}E_{21}^{\top}G_1
)^{-1}$,
and hence 
$V_{11}$ 
converges to $(G_1^{\top}E_{22}D_{12}^*E_{22}^{\top}G_1)^{-1}$.
Then we can see that  the convergence of \eqref{e154}   to  
$D_{12}^*E_{22}^{\top}G_1(G_1^{\top}E_{22}D_{12}^*E_{22}^{\top}G_1)^{-1}(D_{12}^*E_{22}^{\top}G_1)^{\top}$would follow if we show
\begin{align}
\|D_{11}E_{21}^{\top}G_2
(G_2^{\top}E_{21}D_{11}
E_{21}^{\top}G_2)^{-1}
G_2^{\top}E_{21}D_{11}\|_F
\to 0
\end{align}
as $D_{11}\to 0$.
This is indeed true since $D_{11}^{1/2}E_{21}^{\top}G_2
(G_2^{\top}E_{21}D_{11}
E_{21}^{\top}G_2)^{-1}
G_2^{\top}E_{21}D_{11}^{1/2}$ is a projection matrix.
The proof is completed.
\end{proof}

\section{Convergence for general regularizers}
\label{sec_general}

\subsection{Auxiliary lemmas}
\begin{lem}\label{lems15}
    Suppose $M$ is symmetric and $I+M\succeq cI$ for some $c>0$.
    Then 
    $0\preceq (I-(I+M)^{-1})^2 \preceq \frac1{c^2} M^2$.
\end{lem}
\begin{proof}
The first $\preceq$ is trivial. For the second one, note that by choosing a suitable orthogonal basis the problem is reduced to the case of diagonal $M$, in which case it is also trivial.
\end{proof}

\begin{lem}\label{lems16}
Suppose that $M_1,M_2$ are psd matrices,
$\Delta$ is a symmetric matrix,
and 
\begin{align}
M_1+\Delta&\succeq  c_1 I
\\
M_1&\succeq  c_1 I;
\\
\lambda_{\max}(M_i)&\le c_2,\quad i=1,2,
\end{align}
for some $c_1,c_2>0$.
Then 
\begin{align}
(M_1^{-1}-(M_1+\Delta)^{-1})
M_2
(M_1^{-1}-(M_1+\Delta)^{-1})
\preceq 
\frac{c_2^3}{c_1^4}
M_1^{-1}
\Delta^2
M_1^{-1}.
\end{align}
\end{lem}

\begin{proof}
We have  
\begin{align}
&\quad (M_1^{-1}-(M_1+\Delta)^{-1})
M_2
(M_1^{-1}-(M_1+\Delta)^{-1})
\nonumber
\\
&=
M_1^{-1/2}
(I-(I+M_1^{-1/2}\Delta
M_1^{-1/2})^{-1})
M_1^{-1/2}
M_2
M_1^{-1/2}
(I-(I+M_1^{-1/2}\Delta
M_1^{-1/2})^{-1})
M_1^{-1/2}
\\
&\preceq
\frac{c_2}{c_1}
M_1^{-1/2}
(I-(I+M_1^{-1/2}\Delta
M_1^{-1/2})^{-1})^2
M_1^{-1/2}
\\
&\preceq 
\frac{c_2}{c_1}
\cdot 
(\frac{c_1}{c_2})^{-2}
M_1^{-1/2}
(M_1^{-1/2}\Delta
M_1^{-1/2})^2
M_1^{-1/2}
\\
&\preceq 
\frac{c_2}{c_1}
\cdot 
(\frac{c_1}{c_2})^{-2}
\cdot 
\frac1{c_1}
M_1^{-1}
\Delta^2
M_1^{-1}
\end{align}
where we applied Lemma~\ref{lems15}.
\end{proof}

\begin{lem}\label{lem32}
We have
\begin{align}
\left[ \frac1{n}A_{:\sj}
(\bar{G}_{\mathcal{A}}^{-1}-G^{-1})
A_{:\sj}^{\top}\right]^2
&\preceq
\frac{c_2^3\zeta_3^2}{4c_1^4n}A_{:\sj}
\bar{G}_{\mathcal{A}}^{-1}
\diag((|\hat{\alpha}_l-\hat{\beta}_l|^2)_{l\neq j})
\bar{G}_{\mathcal{A}}^{-1}
A_{:\sj}^{\top}
\label{es152}
\\
\left[ \frac1{n}A_{:\sj}
(\bar{G}_{\mathcal{A}}^{-1}-\bar{G}_{\mathcal{B}}^{-1})
A_{:\sj}^{\top}\right]^2
&\preceq
\frac{c_2^3\zeta_3^2}{c_1^4n}A_{:\sj}
\bar{G}_{\mathcal{A}}^{-1}
\diag((|\hat{\alpha}_l-\hat{\beta}_l|^2)_{l\neq j})
\bar{G}_{\mathcal{A}}^{-1}
A_{:\sj}^{\top}
\label{es153}
\end{align}
where $c_2:=\frac1{n}\lambda_{\max}(A_{:\sj}^{\top}A_{:\sj})+\lambda \zeta_{2+}$ and $c_1:=\lambda \zeta_{2-}$.
\end{lem}

\begin{proof}
Recall that 
\begin{align}
\bar{G}_{\mathcal{A}}
&:=
\frac1{n}A_{:\sj}^{\top}A_{:\sj}+
\lambda \diag((\psi'(\hat{\alpha}_l))_{l\neq j});
\\
G
&:=
\frac1{n}A_{:\sj}^{\top}A_{:\sj}+
\lambda 
\left(\frac{\psi(\hat{\alpha}_l)-\psi(\hat{\beta}_l)}
{\hat{\alpha}_l-\hat{\beta}_l}\right)_{l\neq j}.
\end{align}
Since $\|\psi''\|_{\infty}\le \zeta_3$, we have 
\begin{align}
\left|
\psi'(\hat{\alpha}_l)
-
\frac{\psi(\hat{\alpha}_l)-\psi(\hat{\beta}_l)}
{\hat{\alpha}_l-\hat{\beta}_l}
\right|
\le 
\frac1{2}\zeta_3|\hat{\alpha}_l-\hat{\beta}_l|.
\end{align}
Then \eqref{es152} from applying Lemma~\ref{lems16} with $M_2:=\frac1{n}A_{:\sj}^{\top}A_{:\sj}$, 
$M_1:=\bar{G}_{\mathcal{A}}$,
and 
$\Delta:=\lambda 
\left(\frac{\psi(\hat{\alpha}_l)-\psi(\hat{\beta}_l)}
{\hat{\alpha}_l-\hat{\beta}_l}\right)_{l\neq j}
-
\lambda \diag((\psi'(\hat{\alpha}_l))_{l\neq j})$.
The proof of \eqref{es152} is similar, thus omitted.
\end{proof}

\subsection{Proof of Theorem~\ref{thm7}}
Recall from \eqref{eq82} that
\begin{align}
\lambda\|\psi(\hat{\alpha}_{\sj})-
\psi(\hat{\beta}_{\sj})
\|_2
&\le 
\frac1{n}\|A_{:\sj}^{\top}(R-S)\|_2
\\
&\le 
\frac1{n}\|A\|_{\rm op}
\|A_{:j}\hat{\alpha}_j-B_{:j}\hat{\beta}_j\|_2
\\
&\le
\frac1{n}(\|A\|_{\rm op}^2\vee \|B\|_{\rm op}^2)
(|\hat{\alpha}_j|+|\hat{\beta}_j|).
\end{align}
Under $\inf_{t\in\mathbb{R}}\psi'(t)\ge \zeta_{2-}$, the left side above also upper bounds $\lambda \zeta_{2-}\|\hat{\alpha}_{\sj}-\hat{\beta}_{\sj}\|_2$.
Thus
\begin{align}
\|\frac1{n}A_{:\sj}
(\bar{G}_{\mathcal{A}}^{-1}-\bar{G}_{\mathcal{B}}^{-1})
A_{:\sj}^{\top}
\check{B}_{:j}\|_2
&\le 
\frac{c_2^{1.5}\zeta_3}{c_1^2\sqrt{n}}
\|\bar{G}_{\mathcal{A}}^{-1}
A_{:\sj}^{\top}
\check{B}_{:j}
\|_{\infty}
\cdot
\frac1{n\lambda \zeta_{2-}}(\|A\|_{\rm op}^2\vee \|B\|_{\rm op}^2)
(|\hat{\alpha}_j|+|\hat{\beta}_j|)
\\
&\le
\frac{c_2^{1.5}\zeta_3}{c_1^2\lambda \zeta_{2-}}\Gamma D^2
(|\hat{\alpha}_j|+|\hat{\beta}_j|).
\end{align}
The rest of the proof is the same as Section~\ref{sec_aformula}, using Lemma~\ref{lem32} in place of 
Lemma~\ref{lem5}.

\subsection{Proof of Theorem~\ref{thm8}}
There exists $c>0$ (depending only on $\mathcal{P}$ and $C$) such that the following holds:
there exists a set $\mathcal{E}$ of $A$ such that 
\begin{align}
\mathbb{P}[A\in \mathcal{E}]
\le n^{-C}
\label{e_45}
\end{align}
and that for each $A\in \mathcal{E}^c$,
\begin{align}
\|A\|_{\rm op}
&\vee \max_{j=1}^p\|B^{(j)}\|_{\rm op}\le \frac1{c}\sqrt{n};
\label{e44}
\\
\check{B}_j^{(j)}
&:=B_j^{(j)}-\mu_{:j}
\textrm{ is }\frac{\ln n}{c}\textrm{-sub-Gaussian.}
\label{e45}
\end{align}
Here \eqref{e44} is a standard results following from references around \eqref{e184},
and \eqref{e45} follows from Lemma~\ref{lem9} and the union bound. 
Then conditioned on any $A\in\mathcal{E}^c$, we have 
$\max_{1\le j
\le p}\|\bar{G}_{\mathcal{A}}^{-1}
A_{:\sj}^{\top}\|_{\rm op}
\le \frac{\sqrt{n}}{c}\zeta_{2-}^{-1}$,
and each coordinate of 
$\bar{G}_{\mathcal{A}}^{-1}
A_{:\sj}^{\top}
\check{B}_{:j}^{(j)}$ is 
$\frac{n\ln n}{c^3}\zeta_{2-}^{-2}$-sub-Gaussian,
and we have (by union bound)
\begin{align}
\max_{1\le j\le p}\|\bar{G}_{\mathcal{A}}^{-1}
A_{:\sj}^{\top}
\check{B}_{:j}^{(j)}
\|_{\infty}
\le 
O_{\mathcal{P},\zeta,C}
(\sqrt{n}\ln n ),
\end{align}
with probability at least $1-n^{-C}$,
where the constants in $O_{\mathcal{P},\zeta,C}$ may depend on parameters in $\mathcal{P}$, $\zeta:=(\zeta_{2-},\zeta_{2+},\zeta_3)$, and $C$.
Therefore, by taking $D=\frac1{c}$ and $\Gamma=\gamma_{\mathcal{P},\zeta,C}\ln n$ with sufficiently large $\gamma_{\mathcal{P},\zeta,C}>0$, we have $\mathcal{G}$ with probability at least $1-O_{\mathcal{P}}(n^{-C})$,
without conditioning on $A$,
where we used the union bound and \eqref{e_45}.

Next, we show that 
\begin{align}
\sum_{j=1}^p\mathbb{E}[|\hat{\alpha}_j|^2+|\hat{\beta}_j|^2]=O_{\mathcal{P},\zeta}(n).
\label{e48}
\end{align}
We first observe that by the optimality of $\hat{\alpha}$,
\begin{align}
\frac1{2n}\|Y\|_2^2\ge \sum_{j=1}^n\rho(\hat{\alpha}_j)
\ge\frac{\zeta_{2-}}{2}\|\hat{\alpha}\|_2^2,
\end{align}
implying $\sum_{j=1}^p\mathbb{E}[|\hat{\alpha}_j|^2]=O_{\mathcal{P},\zeta}(n)$.
To control $\mathbb{E}[\hat{\beta}_j^2]$, we use arguments similar to the proof of Corollary~\ref{cor7}.
Define $A(t)\in\mathbb{R}^{n\times p}$ such that $A_{:\sj}(t)=A_{:\sj}$ and $A_{:j}(t)=(1-t)A_{:j}+tB^{(j)}_{:j}$.
Thus $A(0)=A$ and $A(1)=B^{(j)}$.
Let $Y(t):=A(t)\alpha+w$,
and let $\hat{\alpha}(t)$ be the solution to the penalized regression problem for the data $(A,Y(t))$.
Thus $\hat{\alpha}(0)=\hat{\alpha}$, and $\hat{\alpha}(1)$ has the same distribution as $\hat{\beta}^{(j)}$ (since $B^{(j)}$ and $A$ are identically distributed).
The normal equation reads as
\begin{align}
-\frac1{n}A^{\top}(Y(t)-A\hat{\alpha}(t))+\psi(\hat{\alpha}(t))=0,
\end{align}
from which we obtain the derivative
\begin{align}
-\frac1{n}A^{\top}
(B_{:j}^{(j)}-A_{:j})\alpha_j
+\left(\frac1{n}A^{\top}A
+\diag(\psi'(\hat{\alpha}(t)))\right)
\hat{\alpha}'(t)=0.
\end{align}
Since $\left\|\left(\frac1{n}A^{\top}A
+\diag(\psi'(\hat{\alpha}(t)))\right)^{-1}\right\|_{\rm op}\le \zeta_{2-}^{-1}$,
we have $\|\hat{\alpha}'(t)\|_2\le |\alpha_j|O_{\mathcal{P},\zeta}(1)$.
Thus $|\hat{\alpha}_j(0)-
\hat{\alpha}_j(1)|
\le 
\|\hat{\alpha}(0)-
\hat{\alpha}(1)\|_2\le |\alpha_j|O_{\mathcal{P},\zeta}(1)$ 
and 
$|\hat{\alpha}_j(1)|^2
\le 
2|\hat{\alpha}_j(0)|^2
+2|\hat{\alpha}_j(1)-\hat{\alpha}_j(0)|^2
\le 2\hat{\alpha}_j^2+|\alpha_j|O_{\mathcal{P},\zeta}(1)$.
Then $\sum_{j=1}^p\mathbb{E}[|\hat{\beta}_j|^2]=
\sum_{j=1}^p\mathbb{E}[|\hat{\alpha}_j(1)|^2]
=O_{\mathcal{P},\zeta}(n)$, verifying \eqref{e48}.

\section{Experiments}
Results are generated by the following codes, which can be found in \url{https://github.com/jingboliu1/local_change.git}
\begin{itemize}
\item Figures~\ref{fig1_1}--\ref{fig2_3}: \verb|localk1_fig1and2.R|.
\item Table~\ref{tab0}: \verb|err_eval.R|.
\item Table~\ref{tab01}: \verb|err_eval1.R|.
\item Tables~\ref{tab1}, \ref{tab2} and~\ref{tab_new}: \verb|FDR_eq_table1.R|.
\item Table~\ref{tab3}: \verb|FDR_riboflavin_table3.R|.
\item Table~\ref{tab_lam001}: \verb|FDR_HIV1.R|.
\end{itemize}

\subsection{Table of approximation errors in the update formula}
\label{sec_61}

Consider the setting where the rows of the design matrix $A\in \mathbb{R}^{n\times p}$ are i.i.d.\ $\mathcal{N}(0,\Sigma)$,
where 
\begin{align}
\Sigma_{ij}=\rho^{|i-j|}.
\end{align}
Let $\alpha\in\mathbb{R}^p$ be such that the first $s$ entries are generated i.i.d.\  according to $\mathcal{N}(0,1)$, and the remaining entries are 0.
Take 
\begin{align}
Y=A\alpha+w
\end{align}
where $w\sim \mathcal{N}(0,0.01nI_{n\times n})$ is an independent Gaussian noise.
We take $n=800\upalpha$, $p=1000\upalpha$, $s=400\upalpha$, where $\upalpha$ controls the scale of the problem.
We uniformly sample $\lfloor \frac{p}{10\upalpha}\rfloor$ coordinates $j$ in $\{1,\dots,p\}$, 
resample the corresponding column in $A$ with the conditional distribution $P_{A_{:j}|A_{:\sj}}$, 
and compute the Lasso estimator $\hat{\beta}^{(j)}_j$ and the debiased estimator $\hat{\beta}^{(j)U}_j$
(both the exact values as well as the approximate values via the approximation formulas in Section~3.1 in the main paper).
We compute the normalized approximation error (mean squared error normalized by the square of the $\ell_2$ norm) of the approximation formulas across $\upalpha$ and $\rho$.
The results are shown in Table~\ref{tab0}.
These results confirm that the approximation error for the debiased estimator is consistently smaller than that of the standard Lasso, 
with the improvement becoming more pronounced as the correlation $\rho$ increases, in agreement with the theoretical analysis.

\begin{table}[ht]
\centering
\caption{Comparison of normalized square errors in Lasso (numbers on the left) and debiased Lasso (numbers on the right) across $\upalpha$ (dimension scale factor) and $\rho$}
\label{tab0}
\begin{tabular}{cccc}
\hline
$\upalpha$ & $\rho = 0$ & $\rho = 0.5$ & $\rho = 0.95$ \\ \hline
0.1 & 0.849, 0.758 & 0.170, 0.0686 & 0.0958, 0.0265 \\
0.2 & 0.224, 0.121 & 0.532, 0.0976 & 0.0250, 0.00285 \\
0.3 & 0.0778, 0.0557 & 0.0993, 0.0472 & 0.0969, 0.00343 \\
0.4 & 0.153, 0.0717 & 0.0764, 0.0269 & 0.290, 0.00554 \\
0.5 & 0.0555, 0.0381 & 0.112, 0.0241 & 0.369, 0.00157 \\
0.6 & 0.137, 0.0706 & 0.106, 0.0117 & 0.459, 0.00296 \\
0.7 & 0.0367, 0.0309 & 0.0490, 0.00948 & 0.00744, 0.00177 \\
0.8 & 0.0681, 0.0275 & 0.0750, 0.00612 & 0.140, 0.000567 \\
0.9 & 0.0228, 0.0174 & 0.0709, 0.00603 & 0.202, 0.000790 \\
1.0 & 0.0350, 0.0306 & 0.0579, 0.00604 & 0.566, 0.00170 \\
1.1 & 0.0521, 0.0201 & 0.161, 0.0173 & 0.140, 0.000248 \\
1.2 & 0.0235, 0.0114 & 0.303, 0.0173 & 0.199, 0.00156 \\
1.3 & 0.0236, 0.0214 & 0.175, 0.00563 & 0.655, 0.00110 \\
1.4 & 0.0248, 0.00905 & 0.0209, 0.00251 & 0.633, 0.00396 \\
1.5 & 0.0161, 0.00514 & 0.0836, 0.00746 & 0.470, 0.000992 \\ \hline
\end{tabular}
\end{table}

In Table~\ref{tab01}, we choose $\rho=0.5$, and use the same method for generating $A$, but generate $B$ by 
$B_{ij}:=\mathbb{E}[A_{ij}|A_{i\sj}]
+((\Sigma^{-1})_{jj})^{-1/2}\Xi_{ij}$, 
where $\Xi_{ij}$ are i.i.d. random variables satisfying $\Xi_{ij}=\pm (\xi)^{-1/2}$ with probability $\xi/2$, and $\Xi_{ij}=0$ with probability $1-\xi$.
In other words, $\check{B}_{ij}$ is now a 3-point distribution with the same mean and variance as in Table~\ref{tab0}.
This demonstrates the validity of the approximation formula beyond Gaussian designs, consistent with the theoretical analysis.

\begin{table}[ht]
    \centering
\caption{Normalized square errors in debiased Lasso across $\upalpha$ (dimension scale factor) and $\xi$ (parameter for the 3-point distribution)}
\label{tab01}
\begin{tabular}{cccc}
\hline
$\upalpha$ & $\xi = 0.5$ & $\xi = 0.1$ & $\xi = 0.05$ \\ \hline
0.1 & 0.1666 & 0.1330 & 0.1467 \\
0.3 & 0.0294 & 0.1176 & 0.0306 \\
0.5 & 0.0443 & 0.0281 & 0.0117 \\
0.7 & 0.0184 & 0.0151 & 0.0131 \\
0.9 & 0.00559 & 0.0142 & 0.00565 \\
1.1 & 0.00759 & 0.00507 & 0.00803 \\
1.3 & 0.00545 & 0.00425 & 0.00899 \\
1.5 & 0.00216 & 0.00552 & 0.00554 \\ \hline
\end{tabular}
\end{table}

\subsection{Plots of approximation errors in the update formula}

We consider the setting in Section~\ref{sec_61}, fixing $n=1000$, $p=1200$, $s=500$, 
In Fig.~\ref{fig1_1} - Fig.~\ref{fig1_3}, we plot the values of $\hat{\beta}^{(j)U}_j$ and $\hat{\beta}^{(j)U}_j-\tilde{\gamma}_j$
for $\rho=0,0.5,0.95$,
where $\hat{\beta}^{(j)U}_j$ is the debiased Lasso coefficient computed exactly according to the definition,
and 
$\tilde{\gamma}_j$ denotes the value computed using the approximation formula.
Fig.~\ref{fig2_1} - Fig.~\ref{fig2_3} compare $\hat{\beta}^{(j)}_j$ and its approximation error for $\rho=0,0.5,0.95$.
In these plots, we only uniformly select 1/12 of all the coordinates, 
to avoid cluttering of the picture. 
It can be seen that the approximation error for the debiased estimator is better than the plain Lasso for large $\rho$ (in turns of size of the error relative to the magnitude of the debiased coefficients), which is consistent with the theoretical analysis.


\begin{figure}[h]
\centering
    \includegraphics[width=9cm]{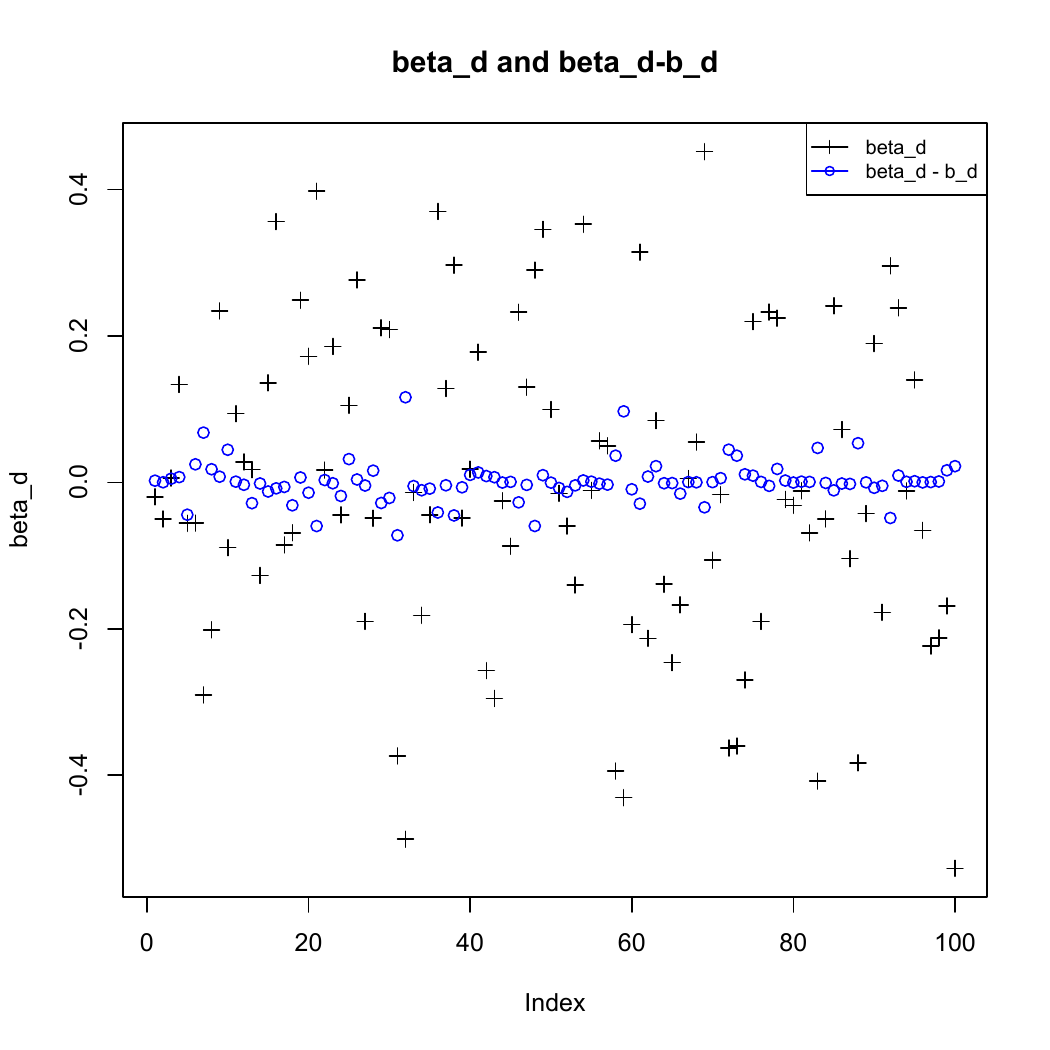}
\caption{Comparison of $\hat{\beta}^{(j)U}_j$ (cross) and its approximation error 
    $\hat{\beta}^{(j)U}_j-\tilde{\gamma}_j$ (circle) for $\rho=0$.}
        \label{fig1_1}
\end{figure}

\begin{figure}[h]
\centering
\includegraphics[width=9cm]{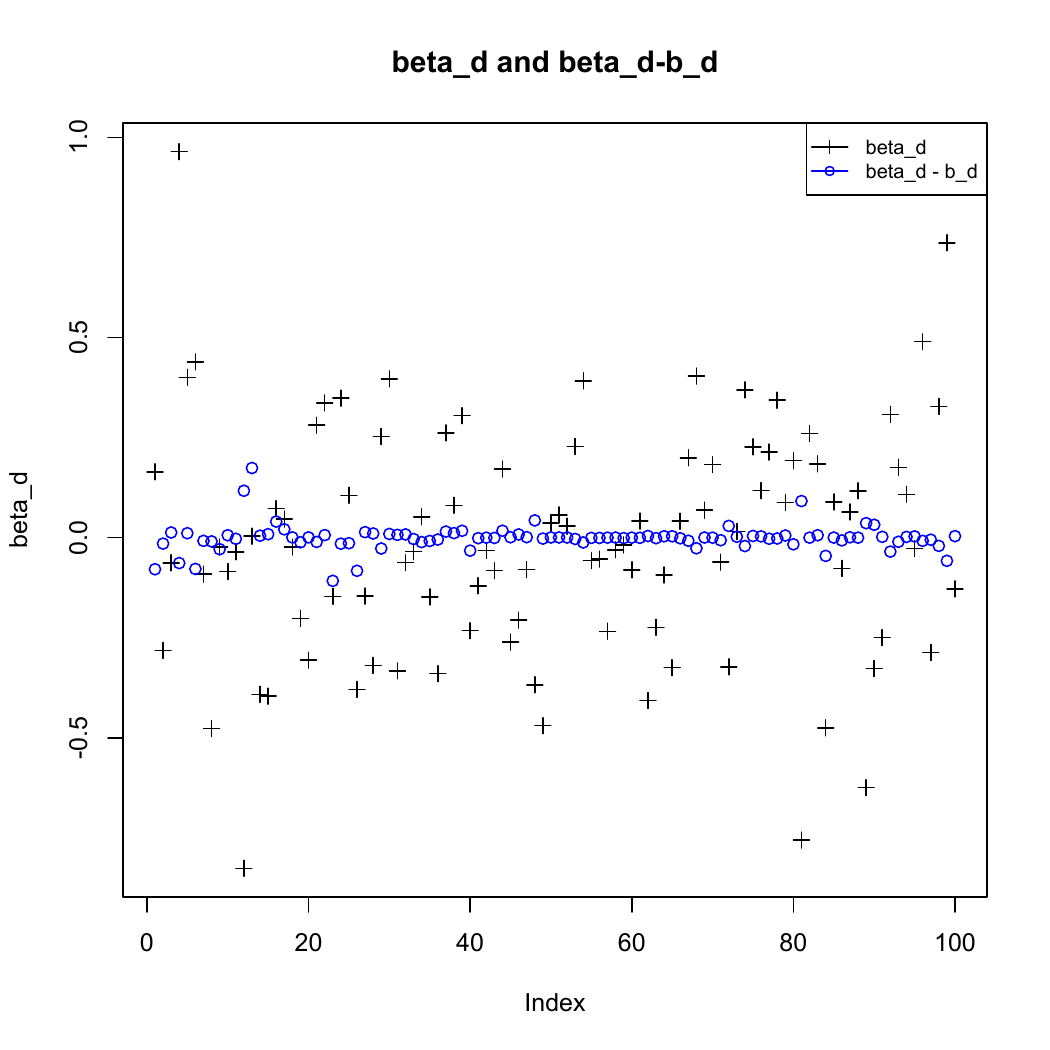}
\caption{Comparison of $\hat{\beta}^{(j)U}_j$ (cross) and its approximation error 
    $\hat{\beta}^{(j)U}_j-\tilde{\gamma}_j$ (circle) for $\rho=0.5$.}
\label{fig1_2}
\end{figure}

\begin{figure}[h]
\centering
\includegraphics[width=9cm]{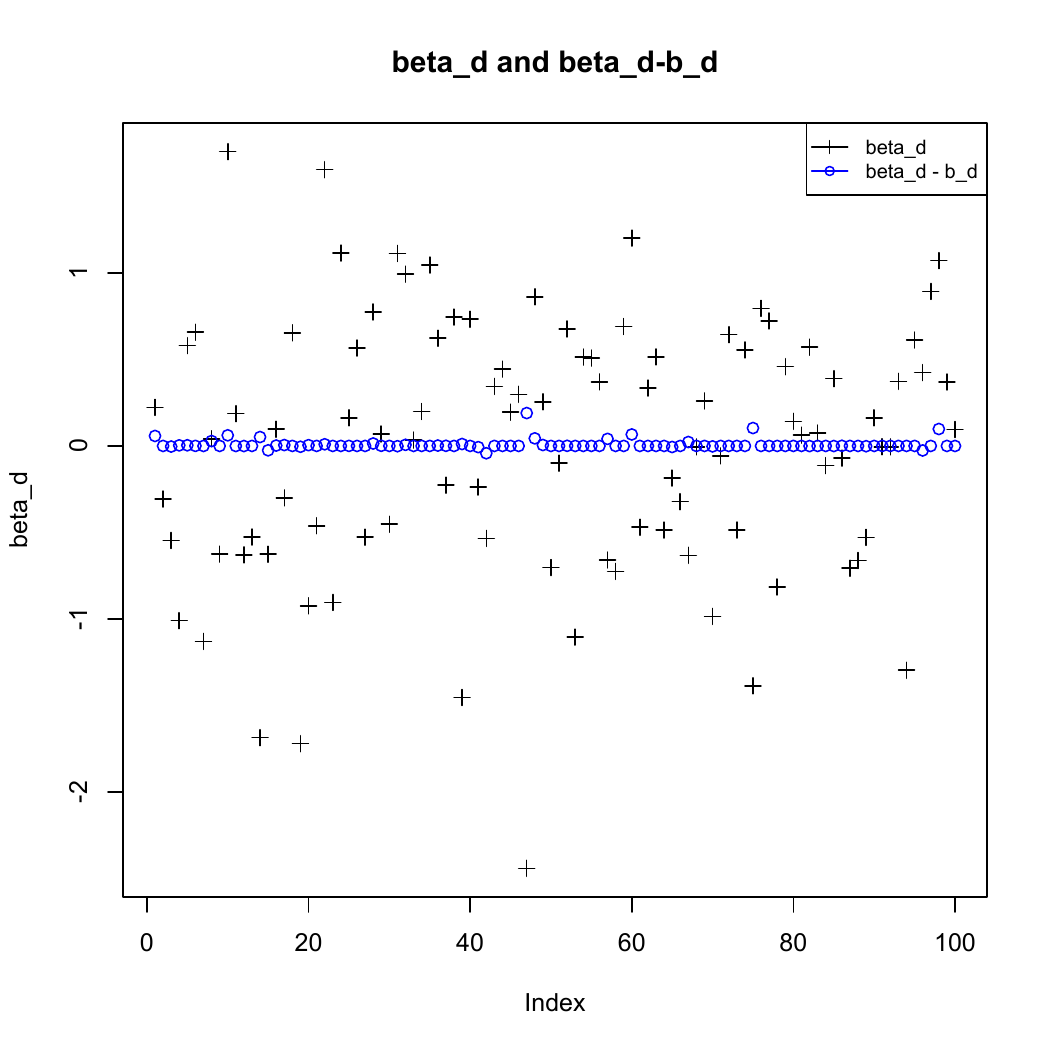}
    \caption{Comparison of $\hat{\beta}^{(j)U}_j$ (cross) and its approximation error 
    $\hat{\beta}^{(j)U}_j-\tilde{\gamma}_j$ (circle) for $\rho=0.95$.}
\label{fig1_3}
\end{figure}


\begin{figure}[h]
\centering
\includegraphics[width=9cm]{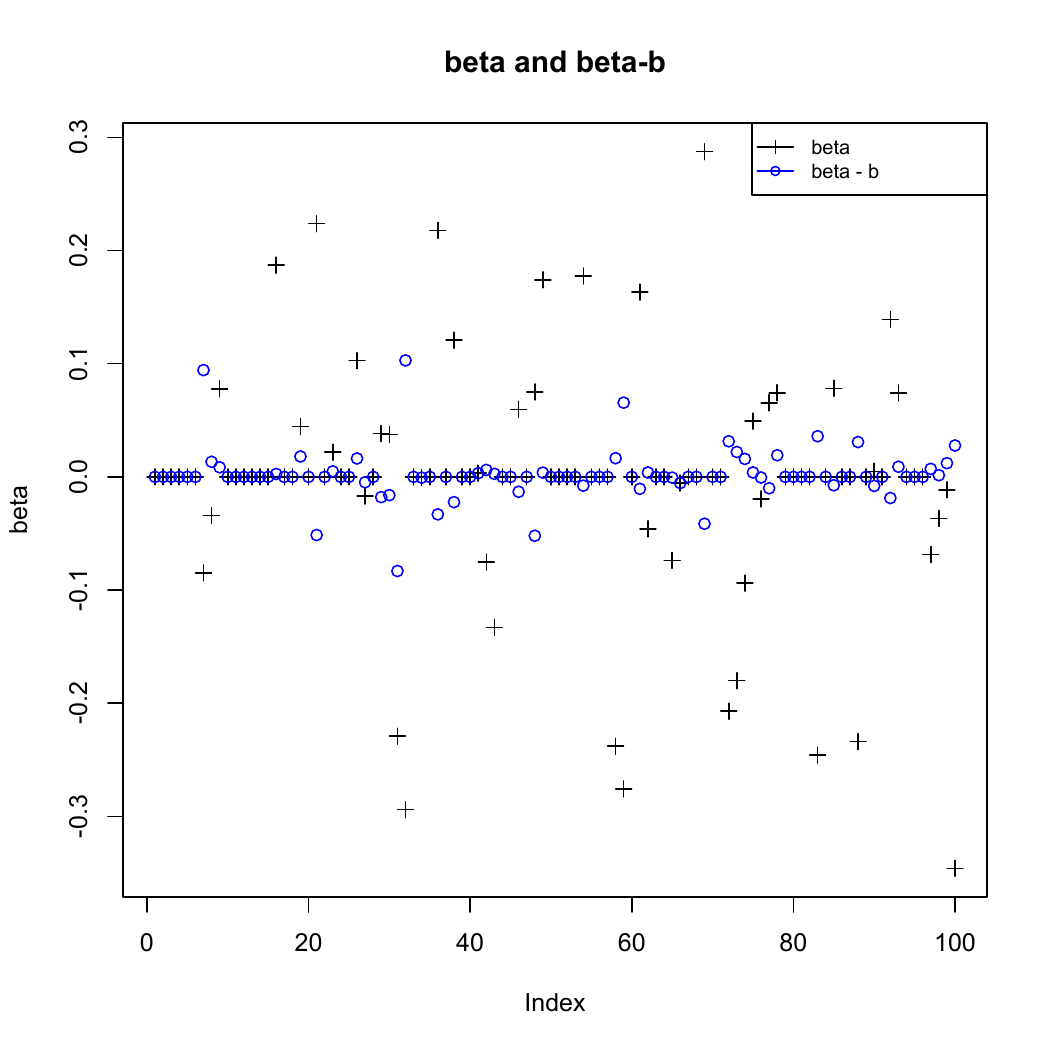}
        \caption{Comparison of $\hat{\beta}^{(j)}_j$ (cross) and its approximation error (circle) for $\rho=0$.}
        \label{fig2_1}
\end{figure}

\begin{figure}[h]
\centering
    \includegraphics[width=9cm]{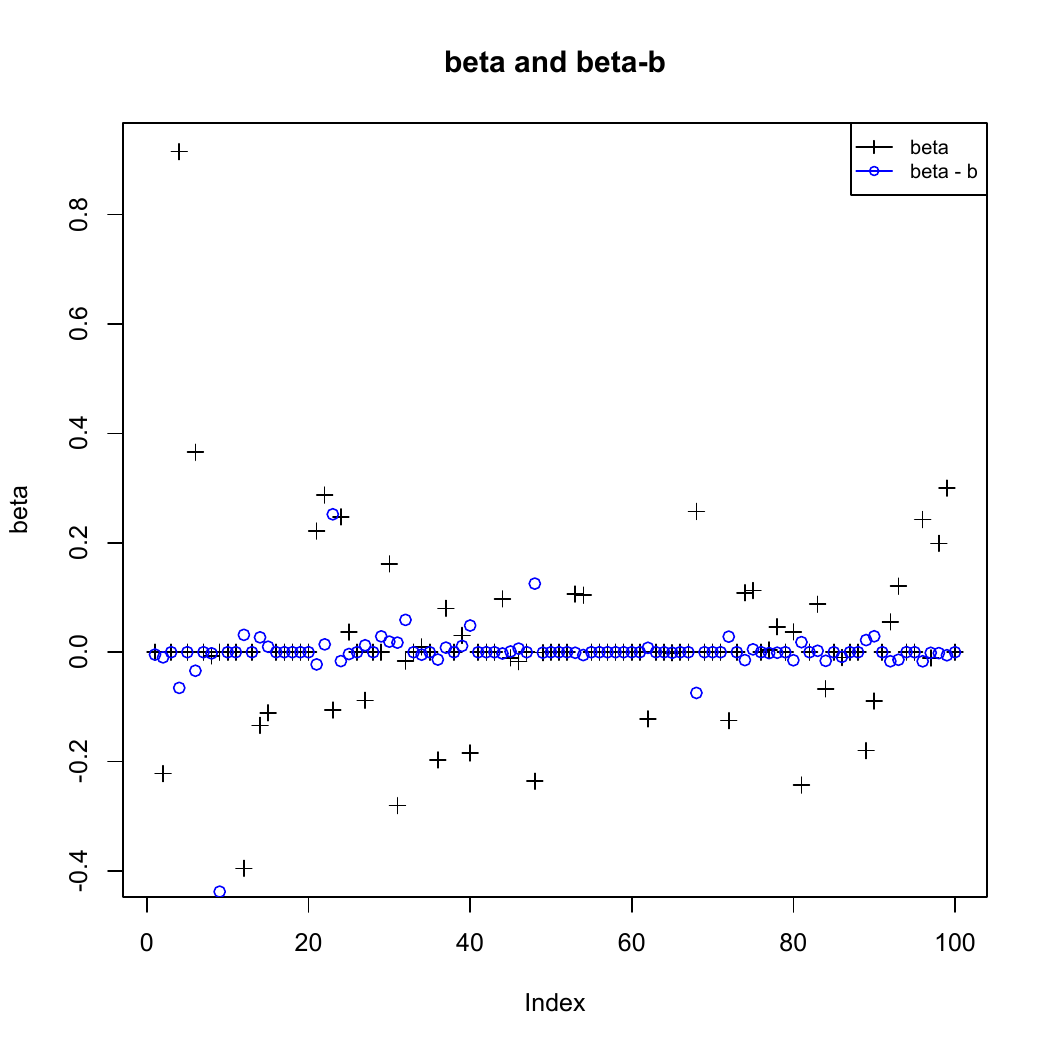}
        \caption{Comparison of $\hat{\beta}^{(j)}_j$ (cross) and its approximation error (circle) for $\rho=0.5$.}
            \label{fig2_2}
\end{figure}

\begin{figure}[h]   
\centering
\includegraphics[width=9cm]{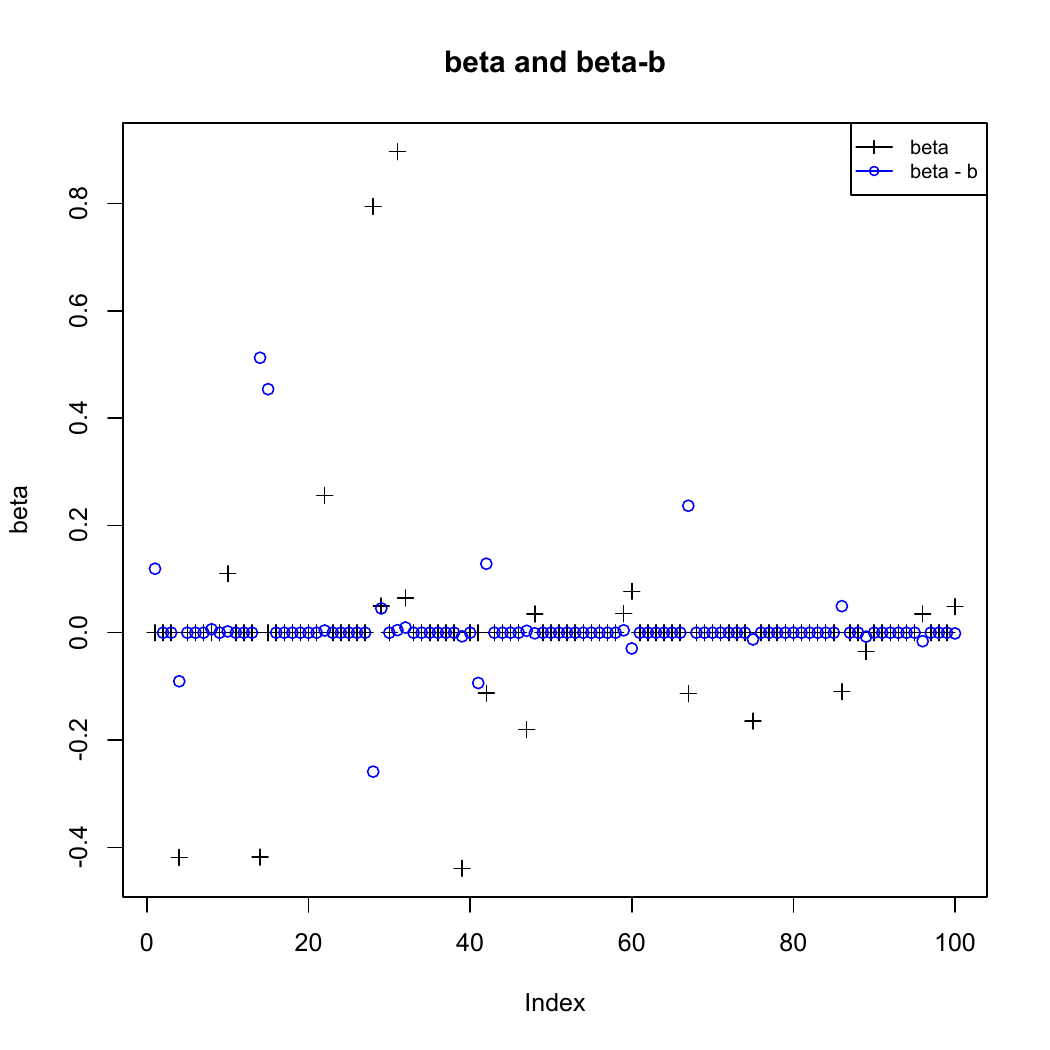}

\caption{Comparison of $\hat{\beta}^{(j)}_j$ (cross) and its approximation error (circle) for $\rho=0.95$.}               \label{fig2_3}
\end{figure}

\subsection{FDR control experiments with synthetic data}
\label{sec_exp}

Consider an FDR control problem similar to the setting in Theorem~13 in the main paper, where
$p=300$,
$n=200$,
and the target FDR
$q= 0.1$.
We take $\Sigma^{-1}=a_p(E+\epsilon I)$, where $E$ denotes the matrix whose entries are all 1,
and $a_p,\epsilon>0$ are parameters to be specified later.
Then we have 
$
\Sigma=\frac1{a_p\epsilon}
\left(
I-\frac{\epsilon^{-1}}{1+\epsilon^{-1}p}E
\right)
$.
We then generate $\alpha$ with a random set of $s$ coordinates equal to $A_{\rm val}/\sqrt{n}$ ($A_{\rm val}>0$ being a parameter to be specified),
and the rest coordinates equal to 0.
The observation is  $Y=A\alpha+w$, where $w\sim \mathcal{N}(0,\sigma I)$.

We compare the performance of 6 variable selection methods in Table~\ref{tab1} and Table~\ref{tab2}.
For the knockoff method, we will use the eq-knockoff construction of  \cite{candes2018panning}, 
which is natural in this setting since $\Sigma$ has equal values for the off-diagonals and for the diagonals.
The max eigenvalue of $\Sigma^{-1}$ is $a_p(p+\epsilon)$, so the condition $2S^{-1}\succeq \Sigma^{-1}$ in eq-knockoff 
becomes $S_{11}=\dots=S_{pp}\le \frac{2}{\lambda_{\rm max}(\Sigma^{-1})}=\frac{2}{a_p(p+\epsilon)}$.
To design the knockoff filter,
one tries to minimize $|\Sigma_{11}-S_{11}|=\left|\frac{1+\epsilon^{-1}(p-1)}{a_p(\epsilon+p)}-S_{11}\right|$ and so for $0<\epsilon<p-1$ the optimal $S_{11}=\frac{2}{a_p(p+\epsilon)}$.
For $\epsilon>p-1$, the optimal $S_{11}=\frac{1+\epsilon^{-1}(p-1)}{a_p(\epsilon+p)}$.

In experiments with relatively small sample sizes,
we report both the FDR estimate and its 95\% confidence interval.
For number of independent experiments $B=50$,
we compute the empirical average
\[
\widehat{FDR}
=
\frac{1}{B}
\sum_{b=1}^{B}
FDP^{(b)}.
\]
We estimate the Monte Carlo standard error by
\[
\widehat{SE}
=
\frac{
\operatorname{sd}
\!\left(
FDP^{(1)},\ldots,FDP^{(B)}
\right)
}
{\sqrt{B}},
\]
and report the corresponding Monte Carlo $95\%$ confidence interval
\[
\widehat{FDR}
\pm
1.96\,\widehat{SE}.
\]

In Table~\ref{tab1},
we use $\epsilon=50,a_p=1/p$, number of nonzero coefficients
$s=20$,
and noise standard deviation $\sigma = 1/\sqrt{p}$,
and vary the signal strength $A_{\rm val}$.
We see that FDR is controlled at approximately $q=0.1$ in most cases, except ``Knockoff-db''.
As $s=20$ is relatively small in this setting, there are a few instances of FDR overflow for ``Knockoff-db'',
which is not in violation of the theory.
Indeed, for threshold selection based on knockoff (which is the case for rows ``Knockoff'' and ``Knockoff-db'' in the table), \citet[Theorem~3.4]{candes2018panning} only offers guarantees on a modified FDR.
An alternative, more conservative threshold selection rule, called knockoff+,
is shown in
\citet[Theorem~3.4]{candes2018panning} to control the ordinary FDR.
In Table~\ref{tab1},
``Knockoff-db+'' corrected the problematic FDR readings for ``Knockoff-db'', at the cost of a lower power.
Meanwhile, the power achieved by the local knockoff filter and CRT
are better than the knockoff filter, with or without debiasing.

In Table~\ref{tab2}, 
we change the sparsity level to $s=60$ and the parameter size to $p=100$, and select a slightly different set of $A_{\rm val}$ values.
Again the FDR is controlled at approximately $q=0.1$ in most cases (with the exception of some instances for ``Knockoff'' and ``Knockoff-db'', due to the modified FDR issue mentioned above), while local knockoff filter and CRT achieve higher power than the knockoff filter. 
Note that in contrast to Table~\ref{tab1}, the power for the debiased versions are noticeably better than without debiasing in Table~\ref{tab2},
which is expected since $s$ increased.
For correlated designs, the debiased coefficients tend to outperform the standard coefficients when $s$ is not too small.

In Table~\ref{tab_new}, we increase the dimensions to $p=1000$, 
$n=600$, and sparsity to
$s=200$. Now FDR is well-controlled at approximately $q=0.1$ in all cases.
The local knockoff filter and CRT (especially the debiased versions) achieve high power, whereas the knockoff filters fail. 
This confirms the theoretical analysis in Theorem~13 in the main paper:
when the precision matrix is close to $E$ (the matrix consisting of 1's), the knockoff filter fails in the high-dimensional limit, 
regardless of the choice of the knockoff mechanism, whereas methods based on more relaxed local exchangeability conditions (such as local knockoff and CRT) remains powerful.

\begin{table}[ht]
\centering
\caption{Experimental Results for CRT and Knockoffs}
\label{tab1}
\begin{tabular}{|c|c|c|c|c|}
\hline
\textbf{Method} & \textbf{$A_{\rm val}$} & \textbf{FDR Average} & \textbf{FDR 95\% CI} & \textbf{Power Average} \\
\hline
\multirow{3}{*}{Knockoff}
 & 0.1 & 0.09266273 & [0.0587, 0.1266] & 0.476 \\
 & 0.2 & 0.1309883  & [0.1059, 0.1561] & 0.974 \\
 & 0.5 & 0.1303479  & [0.1031, 0.1576] & 1.000 \\
\hline

\multirow{3}{*}{Knockoff-db}
 & 0.1 & 0.3019134 & [0.1882, 0.4156] & 0.124 \\
 & 0.2 & 0.3009680 & [0.2018, 0.4001] & 0.434 \\
 & 0.5 & 0.1283717 & [0.1035, 0.1532] & 0.960 \\
\hline

\multirow{3}{*}{Knockoff-db+}
 & 0.1 & 0.008615385  & [0, 0.02046459] & 0.028  \\
 & 0.2 & 0.05472834  & [0.0281, 0.0814] & 0.346  \\
  & 0.5 & 0.08750887   & [0.0660, 0.1091] & 0.959   \\
\hline

\multirow{3}{*}{approx-local-knockoff}
 & 0.1 & 0.1222533 & [0.0915, 0.1530] & 0.673 \\
 & 0.2 & 0.1259715 & [0.1004, 0.1515] & 0.998 \\
 & 0.5 & 0.1241513 & [0.0981, 0.1502] & 1.000 \\
\hline

\multirow{3}{*}{approx-local-knockoff-db}
 & 0.1 & 0.1137460 & [0.0838, 0.1437] & 0.650 \\
 & 0.2 & 0.1111665 & [0.0873, 0.1350] & 0.998 \\
 & 0.5 & 0.1059581 & [0.0830, 0.1289] & 1.000 \\
\hline

\multirow{3}{*}{approx-CRT}
 & 0.1 & 0.08235843 & [0.0613, 0.1034] & 0.529 \\
 & 0.2 & 0.09038103 & [0.0744, 0.1064] & 1.000 \\
 & 0.5 & 0.08639233 & [0.0707, 0.1021] & 1.000 \\
\hline

\multirow{3}{*}{approx-CRT-db}
 & 0.1 & 0.08583246 & [0.0642, 0.1075] & 0.601 \\
 & 0.2 & 0.09356313 & [0.0790, 0.1081] & 1.000 \\
 & 0.5 & 0.09016392 & [0.0753, 0.1050] & 1.000 \\
\hline
\end{tabular}
\end{table}

\begin{table}[ht]
\centering
\caption{Experimental Results for CRT and Knockoffs}
\label{tab2}
\begin{tabular}{|c|c|c|c|c|}
\hline
\textbf{Method} & \textbf{$A_{\rm val}$} & \textbf{FDR Average} & \textbf{FDR 95\% CI} & \textbf{Power Average} \\
\hline

\multirow{3}{*}{Knockoff}
 & 0.2 & 0.1885708 & [0.1174, 0.2597] & 0.0623333 \\
 & 0.4 & 0.1053955 & [0.0834, 0.1274] & 0.4033333 \\
 & 0.6 & 0.09942364 & [0.0864, 0.1124] & 0.8623333 \\
\hline

\multirow{3}{*}{Knockoff-db}
 & 0.2 & 0.1161885 & [0.0502, 0.1822] & 0.2180000 \\
 & 0.4 & 0.1535242 & [0.0795, 0.2276] & 0.4206667 \\
 & 0.6 & 0.1279221 & [0.0700, 0.1859] & 0.5473333 \\
\hline

\multirow{3}{*}{approx-local-knockoff}
 & 0.2 & 0.09007134 & [0.0464, 0.1338] & 0.2230000 \\
 & 0.4 & 0.04435575 & [0.0351, 0.0536] & 0.9666667 \\
 & 0.6 & 0.04446117 & [0.0357, 0.0532] & 1.0000000 \\
\hline

\multirow{3}{*}{approx-local-knockoff-db}
 & 0.2 & 0.04598692 & [0.0331, 0.0589] & 0.4506667 \\
 & 0.4 & 0.04541280 & [0.0374, 0.0534] & 0.9840000 \\
 & 0.6 & 0.04566024 & [0.0369, 0.0544] & 1.0000000 \\
\hline

\multirow{3}{*}{approx-CRT}
 & 0.2 & 0.03237107 & [0.0182, 0.0465] & 0.1313333 \\
 & 0.4 & 0.04356834 & [0.0353, 0.0518] & 0.9660000 \\
 & 0.6 & 0.04103702 & [0.0334, 0.0487] & 1.0000000 \\
\hline

\multirow{3}{*}{approx-CRT-db}
 & 0.2 & 0.04048320 & [0.0298, 0.0512] & 0.3790000 \\
 & 0.4 & 0.04051158 & [0.0325, 0.0485] & 0.9823333 \\
 & 0.6 & 0.03975411 & [0.0318, 0.0477] & 1.0000000 \\
\hline

\end{tabular}
\end{table}

\begin{table}[ht]
\centering
\caption{Experimental Results for CRT and Knockoffs}
\label{tab_new}
\begin{tabular}{|c|c|c|c|c|}
\hline
\multicolumn{1}{|c|}{\textbf{Method}} & \textbf{$A_{\rm val}$} & \textbf{FDR Average} & \textbf{Power Average} & \textbf{Runtime} \\ 
\hline

\multirow{3}{*}{Knockoff} 
& 0.2 & 0 & 0.02 & 13.77 s \\
& 0.6 & 0 & 0.025 & 12.68 s \\
& 1.0 & 0 & 0.015 & 13.67 s \\
\hline

\multirow{3}{*}{Knockoff-db} 
& 0.2 & 0 & 0.13 & 15.81 s \\
& 0.6 & 0 & 0.125 & 15.52 s \\
& 1.0 & 0 & 0.125 & 16.11 s \\
\hline

\multirow{3}{*}{approx-local-knockoff} 
& 0.2 & 0.07777778 & 0.415 & 1.96 min \\
& 0.6 & 0.06862745 & 0.475 & 2.01 min \\
& 1.0 & 0.0733945 & 0.505 & 2.05 min \\
\hline

\multirow{3}{*}{approx-local-knockoff-db} 
& 0.2 & 0.05825243 & 0.485 & 3.61 min \\
& 0.6 & 0.05042017 & 0.565 & 3.50 min \\
& 1.0 & 0.04918033 & 0.58 & 3.41 min \\
\hline

\multirow{3}{*}{approx-CRT} 
& 0.2 & 0.1037736 & 0.475 & 5.15 min \\
& 0.6 & 0.08108108 & 0.51 & 5.60 min \\
& 1.0 & 0.10 & 0.54 & 5.72 min \\
\hline

\multirow{3}{*}{approx-CRT-db} 
& 0.2 & 0.06923077 & 0.605 & 6.37 min \\
& 0.6 & 0.07042254 & 0.66 & 6.01 min \\
& 1.0 & 0.06535948 & 0.715 & 5.98 min \\
\hline

\end{tabular}
\end{table}

\subsection{FDR control with Riboflavin data}\label{sec_ribo}

\begin{table}[ht]
\centering
\caption{FDR experiments with Riboflavin data (averaged over 50 runs)}
\label{tab3}
\begin{tabular}{|c|c|c|c|}
\hline
\textbf{Method} & \textbf{FDR Average} & \textbf{FDR 95\% CI} & \textbf{Power Average} \\
\hline

Knockoff
& 0.07862957
& [0.0606, 0.0967]
& 0.3866667
\\
\hline

Knockoff-db
& 0.03701082
& [0.0176, 0.0564]
& 0.1774359
\\
\hline

approx-local-knockoff
& 0.04449921
& [0.0404, 0.0486]
& 0.5117949
\\
\hline

approx-local-knockoff-db
& 0.1293922
& [0.1105, 0.1483]
& 0.5835897
\\
\hline

approx-CRT
& 0.03975572
& [0.0342, 0.0454]
& 0.4820513
\\
\hline

approx-CRT-db
& 0.1938958
& [0.1837, 0.2041]
& 0.6492308
\\
\hline

\end{tabular}
\end{table}

We use the riboflavin dataset, available in the supplemental materials of \cite{buhlmann2014high},
which was widely used in FDR control experiments 
\citep{javanmard2014confidence,buhlmann2014high,huang2017controlling}.
It contains a $p_0=4088$ by $n=71$ matrix of the logarithm of the expression levels,
and a response vector $Y$ of the logarithm of the riboflavin production rate.
The original measurement matrix contains measurements of many similar (highly correlated) genes, 
so we use the \verb"findCorrelation" function in R
to remove the highly correlated columns with cutoff $\rho=0.5$.
We normalize the means and variance of the columns, and  use graphical Lasso function \verb"glasso" to estimate the covariance matrix $\Sigma$ of the features.
Then we use the best linear estimator $A_{:\sj}\Sigma_{\sj}^{-1}\Sigma_{\sj j}$ 
for the $\mu_{:j}$ in the definition of the debiased estimator.
The FDR and power cannot be precisely evaluated since we do not know the ground truth.
To tackle this issue, we first use cross-validated Lasso to obtain $\alpha$ for the observed $Y$, 
and then generate new $Y=A\alpha+w$,
where the noise level is estimated using the norm of the residual in the previous Lasso regression,
so that we can calculate the FDR and power using the new $Y$ and $\alpha$.
Previously, a similar approach for testing FDR control methods on real datasets was adopted in the literature; see for example \cite{javanmard2020flexible}.
The results of the FDR and power values are shown in Table~\ref{tab3}.
We see an increase in the power by using local knockoffs or CRT, while roughly controlling the FDR in most cases.

\subsection{FDR control with HIV data}
We use the same procedure as in Section~\ref{sec_ribo}, but apply it to the HIV Drug Resistance Database, available at \url{https://hivdb.stanford.edu}. 
We use the code from \cite{guan2025one} to extract the design matrix and response vector. 
For this dataset, we have $n = 2026$, $p = 163$, and sparsity level $s = 79$.
For the knockoff function, we use the implementation from the official package.
The results are shown in Table~\ref{tab_lam001}.
While all methods kept FDR below approximately $q=0.1$, 
the debiased versions of local knockoff and CRT achieve higher powers.


\begin{table}[ht]
\centering
\caption{FDR experiments with HIV data (averaged over 50 runs)}
\label{tab_lam001}
\begin{tabular}{|c|c|c|c|}
\hline
\textbf{Method} & \textbf{FDR (95\% CI)} & \textbf{Power} & \textbf{Runtime} \\
\hline

Knockoff
& 0.007 \,[0.001,0.012]
& 0.133
& 4.66 s \\
\hline

approx-local-knockoff
& 0.067 \,[0.057,0.077]
& 0.562
& 1.69 min \\
\hline

approx-local-knockoff-db
& 0.126 \,[0.116,0.136]
& 0.612
& 2.02 min \\
\hline

approx-CRT
& 0.065 \,[0.057,0.074]
& 0.546
& 1.77 min \\
\hline

approx-CRT-db
& 0.105 \,[0.091,0.118]
& 0.603
& 1.60 min \\
\hline

\end{tabular}
\end{table}





\section{Proofs and implementation details for variable se lection}
\label{sec_pvar}

\subsection{Proof of Lemma~\ref{lem15}}
Set $a=\frac1{\sqrt{2}}S^{1/2}e$ where $e=(1,\dots,1)^{\top}$. 
From $\|s\|_1=s^{\top}S^{-1}s\ge\frac1{2}s^{\top}Es=\frac1{2}\|s\|_1^2$
we obtain $\|a\|^2=\frac1{2}\|s\|_1\le 1$.
Moreover,
\begin{align}
(I-\frac1{2}S^{1/2}ES^{1/2})^{-1}
&=(I-aa^{\top})^{-1}
\\
&=I+\frac1{1-\|a\|_2^2}aa^{\top}.
\end{align}
Therefore, the diagonal values of $(2S-S\Sigma^{-1}S)^{-1}$ are 
\begin{align}
\frac1{2s_j}(1+\frac{a_j^2}{1-\|a\|_2^2})
&=\frac1{2s_j}(1+\frac{s_j/2}{1-\|s\|_1/2})
\\
&=\frac1{2s_j}+\frac1{4-2\|s\|_1}
\\
&>\frac1{2s_j}.
\end{align}
However, by the Markov inequality we have 
\begin{align}
\frac1{p}|\{j\colon \frac1{2s_j}>\frac{p}{10}\}|
&=1-\frac1{p}|\{j\colon s_j\ge\frac{5}{p}\}|
\\
&\ge 1-\frac1{5}\|s\|_1
\\
&\ge \frac{3}{5}.
\end{align}

\subsection{Proof of Theorem~\ref{thm_19}}
From the definition of $\Sigma^{-1}$ it is easy to see that the empirical distribution of its diagonals converges to 0.
From the Schur complement theorem we know that the  diagonals of $\underline{\Sigma}^{-1}$ is two copies of the diagonals of 
$(2S-S\Sigma^{-1}S)^{-1}$, which we denote by $d_1,\dots,d_p$. From Lemma~\ref{lem15} we see that $\frac1{p}|\{j\colon d_j>\frac{pa_p}{10}\}|\ge \frac{3}{5}$, hence $\frac1{2p}
|\{j\colon (\underline{\Sigma}^{-1})_{jj} >\frac1{10}\}|\ge \frac{3}{5}$.

\subsection{Proof of Theorem~\ref{thm14}}
Under Assumption~\ref{assump_2},
for large enough $n$ we have the following with high probability:
\begin{align}
\fdr
&=\mathbb{E}\left[\frac{|\{j\in H_0\colon |\hat{\alpha}^U_j|
>T(\epsilon)\}|}{|\{j\colon |\hat{\alpha}^U_j|>T(\epsilon)\}|}
\right]
\\
&\le 
\frac{|\{j\in H_0\colon |\hat{\alpha}^U_j|>T(\epsilon)-0.1\epsilon\}|}{|\{j\colon |\hat{\alpha}^U_j|>T(\epsilon)\}|}
\\
&=\frac{\hat{P}_{\hat{\alpha}^U\alpha}[(T(\epsilon)-0.1\epsilon,\infty)\times \{0\}]}
{\hat{P}_{\hat{\alpha}^U}[(T(\epsilon),\infty)]}
\\
&\le \frac{P_{\hat{\alpha}^U\alpha}[(T(\epsilon)-0.2\epsilon,\infty)\times \{0\}]}
{\hat{P}_{\hat{\alpha}^U}[(T(\epsilon),\infty)]}
\\
&=\frac{P_{\hat{\gamma}\alpha}[(T(\epsilon)-0.2\epsilon,\infty)\times \{0\}]}
{\hat{P}_{\hat{\alpha}^U}[(T(\epsilon),\infty)]}
\\
&=\frac{P_{\tilde{\gamma}\alpha}[(T(\epsilon)-0.3\epsilon,\infty)\times \{0\}]}
{\hat{P}_{\hat{\alpha}^U}[(T(\epsilon),\infty)]}
\\
&=\frac{\hat{P}_{\tilde{\gamma}\alpha}[(T(\epsilon)-\epsilon,\infty)\times \{0\}]}
{\hat{P}_{\hat{\alpha}^U}[(T(\epsilon),\infty)]}
\\
&\le q.
\end{align}
Since the FDP is bounded by 1, convergence in probability implies convergence of expectations.

\subsection{A note on computation of $\hat{\alpha}^U$}
\label{sec_compute}
\begin{algorithm}[ht]
\caption{Computation of $\hat{\alpha}^U$}
\label{algo2}
\begin{flushleft}
\textbf{Input:} 
Data $A\in\mathbb{R}^{n\times p}$ and $Y\in\mathbb{R}^n$, 
$\lambda>0$.
Assume known $\mu_{:1},\dots,\mu_{:p}$.
\end{flushleft}
\begin{algorithmic}
\State Compute the Lasso solution $\hat{\alpha}$ and the subgradient $\psi$ for $(A,Y)$.
Set $\mathcal{S}:=\{j\colon \psi_j=\pm1\}$ and 
$R:=Y-A\hat{\alpha}$.
\State 
Compute 
$\Omega:=A_{:\mathcal{S}}^{\top}A_{:\mathcal{S}}$, $\Pi:=\Omega^{-1}$,
and $P=A_{:\mathcal{S}}\Pi A_{:\mathcal{S}}^{\top}$.
\For{$1\le j\le p$} 
\If{$j\in \mathcal{S}$}
\State 
Set $\mathcal{S}_j\gets \mathcal{S}\setminus\{j\}$
\State
Compute 
$P_j\gets P-(A_{:\mathcal{S}}\Pi_{\mathcal{S}j})A_{:j}^{\top} -A_{:j}(\Pi_{j\mathcal{S}_j}A_{:\mathcal{S}_j}^{\top})
-A_{:\mathcal{S}_j}\Pi_{\mathcal{S}_jj}\Pi_j^{-1}\Pi_{j\mathcal{S}_j}A_{:\mathcal{S}_j}^{\top}$
\Else 
\State
$P_j\gets P$
\EndIf
\State Compute $\check{A}_{:j}=A_{:j}-\mu_{:j}$;
\quad
$\hat{\alpha}_j^U\leftarrow \hat{\alpha}_j
+\left(\frac1{n}\check{A}_{:j}^{\top}
(I-P_j)A_{:j}\right)^{-1}
\frac{\check{A}_{:j}^{\top}R}{n}$.
\EndFor
\end{algorithmic}
\begin{flushleft}
\textbf{Output:} $\hat{\alpha}^U$.
\end{flushleft}
\end{algorithm}

We discuss how to efficiently compute $\hat{\alpha}^U$; see Algorithm~\ref{algo2}.
Recall the definition of $\hat{\alpha}^U$:
\begin{align}
\hat{\alpha}^U_j
&=\hat{\alpha}_j
+\left(\frac1{n}\check{A}_{:j}^{\top}
(I-P_j)A_{:j}
\right)^{-1}
\frac{\check{A}^{\top}_{:j}R}{n}
\end{align}
where $P_j$ denotes the projection onto the column space of $A_{:\mathcal{S}_j}$, $\mathcal{S}_j:=\{l\neq j\colon \hat{\alpha}_l\neq 0\}$,
and $\Theta:=\Sigma^{-1}$ which can be computed in time $O(n^3)$.
If we use the formula $P_j=A_{:\mathcal{S}_j}(A_{:\mathcal{S}_j}^{\top}A_{:\mathcal{S}_j})^{-1}A_{:\mathcal{S}_j}^{\top}$, then computing $P_j$ takes time $O(n^3)$, 
so that computing $\hat{\alpha}^U$ takes time $O(pn^3)=O(n^4)$.
However, it is possible to efficiently compute the $P_j$'s using a ``rank-one-update'' formula so that the time of computing $\hat{\alpha}^U$ is $O(n^3)$:
First set 
\begin{align}
\Omega:=A_{:\mathcal{S}}^{\top}A_{:\mathcal{S}}
\end{align}
and $\Pi=\Omega^{-1}$, which can be computed in time $O(n^3)$.
Then for each $j$, note that by the Schur complement theorem, 
\begin{align}
(A_{:\mathcal{S}_j}^{\top}A_{:\mathcal{S}_j})^{-1}
={\Omega_{\mathcal{S}_j}}^{-1}
=\Pi_{\mathcal{S}_j}
-\Pi_{\mathcal{S}_jj}\Pi_j^{-1}\Pi_{j\mathcal{S}_j}.
\end{align}
We can compute $\check{A}_{:j}=A_{:j}+A_{:\sj}\Theta_{\sj j}\Theta_j^{-1}$ in time $O(n^2)$.
Since $\Pi_{\mathcal{S}_jj}$ is a vector, we can compute $\check{A}_{:j}^{\top}\Pi_{\mathcal{S}_jj}\Pi_j^{-1}\Pi_{j\mathcal{S}_j}A_{:j}$ in time $O(n^2)$.
We can also compute 
\begin{align}
A_{:\mathcal{S}_j}\Pi_{\mathcal{S}_j}A_{:\mathcal{S}_j}^{\top}=A_{:\mathcal{S}}\Pi A_{:\mathcal{S}}^{\top}-(A_{:\mathcal{S}}\Pi_{\mathcal{S}j})A_{:j}^{\top} -A_{:j}(\Pi_{j\mathcal{S}_j}A_{:\mathcal{S}_j}^{\top})
\end{align}
in time $O(n^2)$.
Then $P_j=A_{:\mathcal{S}_j}(A_{:\mathcal{S}_j}^{\top}A_{:\mathcal{S}_j})^{-1}A_{:\mathcal{S}_j}^{\top}$ can be computed in time $O(n^2)$ for each $j$.
Hence $\hat{\alpha}^U$ is computed in $O(n^3)$ time.

\section*{Acknowledgment}
This research 
was supported in part by NSF Grant DMS-2515510.

\bibliographystyle{plainnat}
\bibliography{KO}

\end{document}